\documentclass[a4paper,11pt]{article}

\usepackage[english]{babel}
\usepackage[utf8x]{inputenc}
\usepackage[T1]{fontenc}
\usepackage[doublespacing]{setspace}
\usepackage[left, mathlines]{lineno}

\usepackage{placeins}
\usepackage{float}
\restylefloat{table}
\restylefloat{figure}

\usepackage[a4paper,top=2cm,bottom=2cm,left=2cm,right=2cm,marginparwidth=1.75cm]{geometry}

\usepackage{lmodern}
\usepackage{subcaption}
\usepackage{mathtools, amsthm, amssymb}
\usepackage[dvipsnames]{xcolor}
\usepackage{graphicx}
\usepackage{array}
\usepackage{calc}
\usepackage{dsfont}
\usepackage[nottoc, notlof, notlot]{tocbibind}
\usepackage{stmaryrd}

\usepackage{amsmath}
\usepackage{amsthm}
\usepackage{graphicx}
\usepackage{hyperref}
\hypersetup{colorlinks,allcolors=black}
\usepackage{float}

\usepackage[numbers]{natbib}

\usepackage{booktabs}
\usepackage{longtable}

\usepackage{authblk} 

\allowdisplaybreaks[4] 
\binoppenalty=\maxdimen
\relpenalty=\maxdimen

\theoremstyle{plain}
\newtheorem{thm}{Theorem}
\newtheorem{lem}[thm]{Lemma}
\newtheorem{prop}[thm]{Proposition}

\theoremstyle{plain}
\newtheorem{defn}[thm]{Definition}

\theoremstyle{remark}
\newtheorem{rem}[thm]{Remark}

\numberwithin{thm}{section} 
\numberwithin{equation}{section} 
\newenvironment{AMS}{\textbf{\textit{MSC 2020 Subject Classification:}}}{}
\newenvironment{keywords}{\textbf{\textit{Keywords:}}}{}
\newenvironment{acknowledgements}{\textbf{Acknowledgements}}{}
\newenvironment{headline}{\textbf{\textit{Running headline:}}}{}

\newcommand{\rcal}{\mathcal{R}}
\newcommand{\bcal}{\mathcal{B}}

\newcommand{\esp}{\mathbb{E}}
\newcommand{\proba}{\mathbb{P}}

\newcommand{\rd}{\mathbb{R}^{d}}

\newcommand{\bxr}{B(x,\rcal)}
\newcommand{\lcal}{\mathcal{L}}
\newcommand{\mcal}{\mathcal{M}}
\newcommand{\gcal}{\mathcal{G}}
\newcommand{\fcal}{\mathcal{F}}
\newcommand{\egras}{\mathbf{E}}

\title{Stochastic measure-valued models for populations expanding in a continuum}
\date{}
\author[,1,2]{Apolline Louvet \thanks{E-mail address : apolline.louvet@polytechnique.edu}}
 
\affil[1]{MAP5, Universit{\'e} de Paris Cit{\'e}, 45 rue des Saints-P{\`e}res, 75006 Paris, France.} 
\affil[2]{CMAP, Ecole Polytechnique, Route de Saclay, 91128 Palaiseau Cedex, France}

\begin{document}
\maketitle

\begin{abstract}
We model spatially expanding populations by means of twospatial $\Lambda$-Fleming Viot processes (or SLFVs) with selection: the $k$-parent SLFV and the $\infty$-parent SLFV. In order to do so, we fill empty areas with type~$0$ "ghost" individuals witha strong selective disadvantage against "real" type~$1$ individuals, quantified by a parameter~$k$. The reproduction of ghost individuals is interpreted as local extinction events due to stochasticity in reproduction. When $k \to + \infty$, the limiting process, corresponding to the $\infty$-parent SLFV, is reminiscent of stochastic growth models from percolation theory,
but is associated to tools making it possible to investigate the genetic diversity in a population sample.
In this article, we provide a rigorous construction of the $\infty$-parent SLFV, and show that it corresponds to the limit of the $k$-parent SLFV when $k \to + \infty$. In order to do so, we introduce an alternative construction of the $k$-parent SLFV which allows us to couple SLFVs with different selection strengths and is of interest in its own right. We exhibit three different characterizations of the $\infty$-parent SLFV, which are valid in different settings and link together population genetics models and stochastic growth models. 
\end{abstract}

\begin{headline}
Spatial Lambda-Fleming-Viot processes for expanding populations
\end{headline}



\tableofcontents

\section{Introduction}\label{sec:1}
Population expansions are common events occurring at all biological scales. The growth of a population in a new environment generates interfaces with distinctive features \cite{huergo2010morphology,kardar1986dynamic} 
and specific patterns of genetic variation \cite{gracia2013surfing,hallatschek2007genetic,hallatschek2010life}, both being a consequence of the stochasticity of reproduction at the front, where local population sizes are small. The models which are used to study expanding populations can be divided in two main categories: growth models, mostly used to investigate the front features, and models coming from population genetics, which are more suited to study genetic diversity patterns. 

Experimental approaches suggest that the dynamics of fronts of real expanding populations belongs to the universality class of the Kardar-Parisi-Zhang (KPZ) equation introduced in \cite{kardar1986dynamic} (see e.g \cite{huergo2010morphology}). It has been conjectured (and demonstrated in the case of the solid-on-solid (SOS) growth model \cite{bertini1997stochastic}) that many growth models generate similar interfaces. Among all stochastic growth models, the most well-known one is probably the Eden model, initially introduced on a lattice in \cite{eden1961two} (see also \cite{jullien1987aggregation}
for alternative update rules or
\cite{wang1995offlattice}
for off-lattice variants). This model belongs to a wider family of discrete stochastic growth models known as \textit{first-passage percolation processes}. More generally, percolation theory is very suited to the modeling of population expansions: see e.g \textit{last-passage percolation processes} (such as the corner growth model, see 
\cite{seppalainen2009lecture}
), or \textit{continuous first-passage percolation processes}, such as the space-continuous equivalent of the Richardson model introduced in 
\cite{deijfen2003asymptotic,deijfen2004stochastic}. However, these models are generally harder to use to study genetic diversity patterns in expanding populations analytically (but see e.g 
\cite{ahlberg2019competition,carstens2011two,deijfen2004coexistence,deijfen2007two}). 

Conversely, models used in population genetics are generally associated with tools allowing one to investigate these patterns. The analysis of the genetic diversity of a population often involves modeling the ancestral lineages of a subset of individuals, and studying how these lineages coalesce into common ancestors \cite{etheridge2011saintflour}. In order to do so, most classical population genetics models assume that populations have constant sizes and that individuals are uniformly distributed over the area of interest. Therefore, they
usually focus on the spread of a genetic type favored by natural selection in an already established population. Such a question was already studied by means of different models including a stochastic component, mostly in one dimension (see e.g \cite{barton2013genetic,etheridge2020genealogies,korolev2010genetic}). 
The most classical one is based on the Fisher-KPP equation \cite{fisher1937wave,kolmogorov1937FKPP}, in which stochasticity is introduced through a Wright-Fisher noise term. If $0 \leq p(t,x) \leq 1$ represents the proportion of individuals of the favoured type at location $x \in \mathbb{R}$ at time $t \geq 0$, then $p(t,x)$ solves the stochastic Fisher-KPP equation if for $x \in \mathbb{R}$ and $t > 0$,
\begin{equation}\label{eqn:fisher_kpp}
\frac{\partial p}{\partial t}(t,x) = \frac{m}{2}\Delta p(t,x)dt + s_{0}p(t,x)(1-p(t,x))
+ \sqrt{\frac{1}{p_{e}}p(t,x)(1-p(t,x))} W(dt,dx)
\end{equation}
where W is a space-time white noise and $p_{e}$ an effective population density. In one dimension, the stochastic Fisher-KPP equation exhibits traveling wave solutions \cite{mueller1995random}, which describe how the advantageous type spreads through space. However, Eq. (\ref{eqn:fisher_kpp}) has no solution in higher dimensions. 
Many variants of the deterministic version of the Fisher-KPP equation have been studied, including versions with individuals having different motilities \cite{bouin2012invasion,calvez2018non}, different growth rates (see e.g \cite{deforet2019evolution}), other diffusion kernels, or other choices for the non-linearity \cite{berestycki2015asymptotic}. 

Other stochastic population genetics models with selection and a spatial structure can broadly be divided into two categories: "individual-based" models and "reproduction event-based" models. The first category comprises adaptations of classical population genetics models, such as the Moran model \cite{durrett2016genealogies,etheridge2020genealogies,hallatschek2008gene} 
or the stepping-stone model \cite{austerlitz1997evolution,barton2013genetic,peischl2015expansion}.
They generally require to divide the space into subunits called \textit{demes}, to which reproduction events are limited and which are connected by migration. 

On the other hand, "reproduction event-based" models, which are based on the spatial $\Lambda$-Fleming Viot process (or SLFV, see \cite{barton2010new,etheridge2008drift}), allow us to keep the spatial continuum. Their main feature is that they model reproduction events affecting whole areas rather than reproduction individual by individual, by means of a Poisson point process of reproduction events. 
Whenever reproduction occurs, a parent is chosen in the area affected by the event, all or part of the other individuals die, while the descendants of the chosen parent refill the now empty area. The Poisson point process encodes both the affected area and the fraction of individuals dying during the reproduction event. The original version of the SLFV does not account for the presence of a selectively favored genetic type, and the parent is chosen uniformly at random in the affected area. However, it is possible to incorporate selection, for instance by sampling several \textit{potential parents} uniformly at random, and then choosing the actual parent among them. See \cite{forien2017central} for different forms of fixed selection mechanisms, \cite{biswas2018spatial,chetwynd2019rare,klimek2020spatial} for ways to introduce fluctuating selection, and 
\cite{etheridge2020rescaling}
for a rigorous construction of the SLFV with selection. Most of the work on the SLFV with selection involved investigating scaling limits under different forms of \textit{weak} selection: in addition to the references mentioned above, see also \cite{etheridge2017branchingbrownian, etheridge2017browniannet} for selection against a specific type, and \cite{etheridge2017branching}
for selection against hybrids in a diploid population.

As one of the underlying assumptions of the SLFV is that each spatial location contains a very large number of individuals, this model can seem ill-suited at first to the study of expansion events. In order to overcome this issue, we adapt an approach originating from interacting particle systems and first applied to population genetics models in \cite{durrett2016genealogies,hallatschek2008gene}
: fill empty areas with "ghost" individuals, which can reproduce but have a very strong selective disadvantage against "real" individuals. In other words, we model a population expansion as the spread of the selectively advantageous "real individual" genetic type in a resident population of individuals with the "ghost individual" genetic type. As a result, and contrary to other works on variants of the SLFV, we will focus on introducing \textit{strong selection} in the SLFV. In particular, we will consider a different limit, when selection goes stronger and stronger, and neither time nor space are rescaled.

In order to modify the SLFV and incorporate strong selection, we consider that the population contains individuals with a "ghost" type, and individuals with potentially different "real" types. The "real" type is selectively advantaged against ghost individuals, but does not give a selective advantage against another real type. In order to incorporate strong selection against ghost individuals, we will consider the two following modifications to the original SLFV:
\begin{itemize}
\item ($k$-parent SLFV) Whenever reproduction occurs, $k$ \textit{potential parents} are chosen uniformly at random in the affected area. If at least one of them is real, then the actual parent is the first real potential parent chosen. Otherwise, the actual parent is the last (ghost) potential parent chosen. 
\item ($\infty$-parent SLFV) Whenever reproduction occurs, if the affected area contains a non-zero fraction of real individuals, then the parent is chosen uniformly at random among real individuals. Otherwise, the parent is chosen uniformly at random among ghost individuals. 
\end{itemize}
In both case, if the parent is real, then it was sampled uniformly at random among all real individuals in the affected area. However, in the $k$-parent SLFV, if the area in which reproduction occurs is partly empty, then the actual parent is not guaranteed to be a real individual. In other words, the $k$-parent SLFV features local extinction events at the front due to stochasticity in reproduction, which become less frequent for increasing values of $k$. Such extinction events no longer occur in the $\infty$-parent SLFV, and this process can be seen as the limit of the $k$-parent SLFV when $k \to + \infty$. 

The $k$-parent SLFV corresponds to a special case of the SLFV with selection introduced in 
\cite{forien2017central}. In this article, we consider two different ways of constructing it:
\begin{enumerate}
\item As the unique solution to a martingale problem 
(see Section~\ref{subsubsec:pb_martingale_k_SLFV}). The corresponding result is a direct generalization to the case $k \geq 2$ of the construction of the SLFV with selection considered in 
\cite{etheridge2020rescaling}. 
\item Adapting the concept of \textit{parental skeleton} from \cite{veber2015spatial}, and sampling parental locations \textit{along with} reproduction events (see Section~\ref{subsec:2.1}). One of the main contributions of this paper is to show that this construction is equivalent to the one based on the martingale problem, and that it can be used to couple $k$-parent SLFVs for different values of~$k$. 
\end{enumerate}

While this alternative construction of the $k$-parent SLFV is interesting in its own right, most of the paper focuses on the construction of the $\infty$-parent SLFV. Indeed, the growth dynamics of the area occupied by real individuals in the $\infty$-parent SLFV is reminiscent of a space-continuous version of the Eden model, and bears resemblance to the space-continuous equivalent of the Richardson model introduced by Deijfen in
\cite{deijfen2003asymptotic,deijfen2004stochastic}. However, in the $\infty$-parent SLFV, compared to
\cite{deijfen2003asymptotic}, the area initially occupied can have infinite Lebesgue measure, and the growth rate at the front is higher. In particular, the $\infty$-parent SLFV is not guaranteed to be well-defined. Moreover, in the model introduced in 
\cite{deijfen2004stochastic}, reproduction no longer occurs in occupied areas, which can potentially lead to different genetic diversity patterns. 

Therefore, regarding the $\infty$-parent SLFV, our goals are twofold:
\begin{enumerate}
\item Provide a rigorous construction of the $\infty$-parent SLFV. 
We actually provide three different characterizations of the $\infty$-parent SLFV, which are valid under different conditions on the area initially occupied and the distribution of reproduction events parameters. 
\item Show to what extend the $\infty$-parent SLFV can be considered as the limit of the $k$-parent SLFV when $k \to + \infty$. 
This convergence results sheds light on the strong selection limit of SLFVs with selection, and links together population genetics models and stochastic growth models. 
\end{enumerate}
In order to do so, we will focus on the area occupied by the population, but future works will include genetic diversity inside the expanding population, using for instance tracer dynamics \cite{durrett2016genealogies,hallatschek2008gene}. 

The article is structured as follows. 
In Section~\ref{sec:nouvelle_section_2}, we define the $k$-parent SLFV and its associated dual process, the $k$-parent ancestral process. We also give an informal construction of the $\infty$-parent SLFV and state our main results. In Section~\ref{sec:2}, we provide a first construction of the $\infty$-parent SLFV and its dual, based on a coupling between $k$-parent SLFV processes with the same initial conditions. This construction is valid for any initial condition  
and for all distributions of reproduction event parameters which satisfy a classical condition for SLFVs
(see Condition~(\ref{eqn:condition_intensite}) in Section~\ref{sec:nouvelle_section_2}). In order to do so, we introduce an alternative construction of the $k$-parent SLFV based on the concept of \textit{parental skeleton} from \cite{veber2015spatial}.
In Section~\ref{sec:3}, we first demonstrate that the $\infty$-parent ancestral process is well defined, and the unique solution to a specific martingale problem. We use this characterization of the dual process to provide an alternative construction of the $\infty$-parent SLFV, which is valid under a stricter condition on the distribution of reproduction event parameters (see Condition~(\ref{eqn:condition_infty_SLFV}) in Section~\ref{sec:nouvelle_section_2})
and if the area initially occupied has finite Lebesgue measure (conditionally on the martingale problem-based construction being valid). We also show a convergence result regarding the dual processes associated to the $k$-parent SLFV and the $\infty$-parent SLFV.  
Section~\ref{sec:4} is devoted to the proof of the duality relation between the $\infty$-parent SLFV and its dual, which is then used to provide yet another characterization of the $\infty$-parent SLFV, this time as the unique solution to a martingale problem. Section~\ref{sec:5} contains technical lemmas used throughout the paper. 

\section{Definitions and main results}\label{sec:nouvelle_section_2}
\subsection{The k-parent SLFV and its dual}\label{subsec:1.1}
\subsubsection{The $k$-parent SLFV}\label{subsubsec:pb_martingale_k_SLFV}
All the random objects introduced in this section will be defined over some probability space $(\Omega, \mathcal{F}, \proba)$. Before introducing the $k$-parent SLFV, we need to set some notation. 

\paragraph{Notation}
Let $d \geq 1$. Let $C_{c}(\rd)$ be the space of all continuous and compactly supported functions $\rd \to \mathbb{R}$, let $C^{1}(\mathbb{R})$ be the space of all continuously differentiable functions on $\mathbb{R}$, let $C_{b}^{1}(\mathbb{R})$ be the space of all bounded functions $\mathbb{R} \to \mathbb{R}$ that are $C^{1}$ and whose first derivative is also bounded, and let $\bcal(\rd)$ be the space of all measurable functions $\rd \to \mathbb{R}$. 

We start by introducing the state space over which the variant of the SLFV with selection we consider is defined. Let $\widetilde{\mcal}_{\lambda}$ be the space of all measures $M$ on $\mathbb{R}^{d} \times \{0,1\}$ such that for all $f \in C_{c}(\mathbb{R}^{d})$,
\begin{equation*}
\int_{\rd \times \{0,1\}} f(x)M(dx,d\kappa) = \int_{\rd} f(x)dx.
\end{equation*}
In other words, $\widetilde{M}_{\lambda}$ is the space of all measures on $\rd \times \{0,1\}$ whose marginal over $\rd$ is Lebesgue measure. 
By a standard decomposition theorem (see e.g \cite{kallenberg2006foundations}, p.561), for all $M \in \widetilde{\mcal}_{\lambda}$, there exists $\omega : \rd \to [0,1]$ measurable such that 
\begin{equation}\label{eqn:decomposition_thm}
M(dx,d\kappa) = ((\omega(x)\delta_{0}(d\kappa) + (1-\omega(x))\delta_{1}(d\kappa))dx.
\end{equation}
Such a $\omega$ is not unique, but defined up to a Lebesgue null set. 
The state space we consider is the set $\mcal_{\lambda}$ of all measures $M \in \widetilde{\mcal}_{\lambda}$ such that there exists a measurable function $\omega : \rd \to \{0,1\}$ (instead of $\omega : \rd \to [0,1]$) satisfying (\ref{eqn:decomposition_thm}). 

We endow $\widetilde{\mcal}_{\lambda}$ 
and $\mcal_{\lambda}$ with the topology of vague convergence. Moreover, let $D_{\mcal_{\lambda}}[0,+\infty)$ 
(resp. $D_{\widetilde{\mcal}_{\lambda}}[0,+\infty)$) 
denote the space of all càdlàg $\mcal_{\lambda}$-valued paths (resp. $\widetilde{\mcal}_{\lambda}$-valued paths), endowed with the standard Skorokhod topology.

Let $M \in \mcal_{\lambda}$, and let $\omega : \rd \to \{0,1\}$ be a measurable function satisfying Eq. (\ref{eqn:decomposition_thm}). The function $\omega$ can be interpreted as the indicator function of a measurable set $E \subset \rd$ corresponding to the area occupied by what will be called "type $0$" individuals, while $\rd \backslash E$ corresponds to the area occupied by "type $1$" individuals. We will consider that type $0$ individuals correspond to the "ghost" individuals mentioned in the introduction, and type $1$ individuals to the "real" individuals. Therefore, type $0$ individuals have a strong selective disadvantage against type $1$ individuals, and $E$ corresponds to the area not yet invaded by the real population (up to a Lebesgue null set). In all that follows, any $\omega : \rd \to \{0,1\}$ such that (\ref{eqn:decomposition_thm}) is true will be called a \textit{density} of $M$, and the notation $\omega_{M}$ will be used to denote an arbitrarily chosen density of $M$.

For all $f \in C_{c}(\mathbb{R}^{d})$, $F \in C^{1}(\mathbb{R})$ and $\omega : \rd \to \{0,1\}$ measurable, we set~:
\begin{equation*}
\langle \omega, f \rangle := \int_{\rd} f(x)\omega(x) dx 
\end{equation*}
and we define the function $\Psi_{F,f} \in C_{b}(\mcal_{\lambda})$ as~:
\begin{align}
\forall M \in \mathcal{M}_{\lambda}, \Psi_{F,f}(M) 
&:= F\left(\int_{\rd \times \{0,1\}} f(x) \mathds{1}_{\{\kappa = 0\}}M(dx,d\kappa)\right) \label{eqn:eqn_fonction_pb_martingale}\\
&= F\left(\int_{\mathbb{R}^{d}} f(x) \omega_{M}(x)dx\right) \nonumber \\
&= F\left(\langle \omega_{M},f\rangle\right). \nonumber
\end{align}

For all $f \in C_{c}(\mathbb{R}^{d})$, we denote the support of $f$ by $Supp(f)$, and for all $\rcal \in \mathbb{R}_{+}^{*}$, we set :
\begin{align*}
Supp^{\mathcal{R}}(f) &:= \{y \in \mathbb{R}^{d} : \exists x \in Supp(f), ||y-x|| \leq \mathcal{R}\} \\
\text{and} \quad \quad \quad \quad \quad V_{\rcal} &:= \mathrm{Vol}(\bcal(0,\rcal)).
\end{align*}
In other words, $V_{\rcal}$ is the volume of a ball of radius $\rcal$, and $Supp^{\rcal}(f)$ is the set of all points which are at a distance of at most $\rcal$ of a point in the support of $f$.  

For all $\omega : \rd \to \{0,1\}$, $\rcal \in \mathbb{R}_{+}^{*}$ and $x \in \rd$, we define the functions $\Theta_{x,\rcal}^{+}(\omega) : \rd \to \{0,1\}$ and $\Theta_{x, \rcal}^{-}(\omega) : \rd \to \{0,1\}$ by~:
\begin{align*}
\Theta_{x,\rcal}^{+}(\omega) &:= \mathds{1}_{\bxr^{c}} \, \omega + \mathds{1}_{\bxr}, \\
\Theta_{x,\rcal}^{-}(\omega) &:= \mathds{1}_{\bxr^{c}}\, \omega.
\end{align*}
$\Theta_{x,\rcal}^{+}(\omega)$ corresponds to filling the ball $\bxr$ with type $0$ individuals (or equivalently, emptying the ball $\bxr$ of all real individuals), while $\Theta_{x,\rcal}^{-}(\omega)$ can be interpreted as filling the ball $\bxr$ with type $1$ individuals. Notice that if $M \in \mcal_{\lambda}$, then $\Theta_{x,\rcal}^{+}(\omega_{M}) \in \mcal_{\lambda}$ and $\Theta_{x,\rcal}^{-}(\omega_{M}) \in \mcal_{\lambda}$. 

\paragraph{Martingale problem}
We now introduce the operator which will be used to define the specific SLFV with selection we will consider as the solution to a well-posed martingale problem. Let $k \in \mathbb{N}\backslash \{0,1\}$, and let $µ$ be a $\sigma$-finite measure on $\mathbb{R}_{+}^{*}$ such that 
\begin{equation}\label{eqn:condition_intensite}
\int_{0}^{\infty}\rcal^{d}µ(d\rcal) < + \infty. 
\end{equation}
Let $\lcal_{µ}^{k}$ be the operator acting on functions of the form $\Psi_{F,f}$ with $f \in C_{c}(\rd)$ and $F \in C^{1}(\mathbb{R})$, defined in the following way. Let $f \in C_{c}(\mathbb{R}^{d})$ and $F \in C^{1}(\mathbb{R})$. Then, for all $M \in \mcal_{\lambda}$, 
\begin{align*}
\lcal_{µ}^{k} \Psi_{F,f}(M) 
:= \int_{\rd}\int_{0}^{\infty}\int_{\bxr^{k}} \frac{1}{V_{\rcal}^{k}} \,  &\left[\left(\prod_{j = 1}^{k} \omega_{M}(y_{j})\right) \, F(\langle\Theta_{x,\rcal}^{+}(\omega_{M}),f\rangle) \right.\\
&+ \left(1-\prod_{j = 1}^{k}\omega_{M}(y_{j})\right) \, F(\langle\Theta_{x,\rcal}^{-}(\omega_{M}),f\rangle)  \\
&- \left. F(\langle\omega_{M},f\rangle) \vphantom{\prod_{j = 1}^{k}}
\right]
dy_{1}...dy_{k}µ(d\rcal)dx. 
\end{align*}
In Section \ref{sec:5}, it is shown that this operator is well-defined, and that it can be rewritten as 
\begin{align*}
\lcal_{µ}^{k} \Psi_{F,f}(M) = \int_{0}^{\infty}\int_{Supp^{\rcal}(f)}\int_{\bxr^{k}} \frac{1}{V_{\rcal}^{k}}\, 
& \left[
\prod_{j = 1}^{k} \omega_{M}(y_{j}) \, F(\langle\Theta_{x,\rcal}^{+}(\omega_{M}),f\rangle)
\right. \\ &+ (1-\prod_{j = 1}^{k}\omega_{M}(y_{j})) \, F(\langle\Theta_{x,\rcal}^{-}(\omega_{M}),f\rangle)
\\
&-  \left. \vphantom{\prod_{j = 1}^{k}} F(\langle\omega_{M},f\rangle)\right]
dy_{1}...dy_{k}dx µ(d\rcal).
\end{align*}

Intuitively, an interpretation of this operator in terms of reproduction events is the following. Whenever a reproduction event affects the ball $\bxr$, $k$ positions $y_{1},...,y_{k}$ are sampled inside the ball, and we take $k$ individuals occupying each one of these positions. Since the density of type $0$ individuals $\omega_{M}$ is $\{0,1\}$-valued, we can consider that all the individuals occupying the position $y_{1}$ (resp. $y_{2},...,y_{k}$) are of type $1 - \omega_{M}(y_{1})$ (resp. $1 - \omega_{M}(y_{2})$, ..., $1 - \omega_{M}(y_{k})$). If $\prod_{j = 1}^{k} \omega_{M}(y_{j}) = 1$, then all the individuals are of type $0$, and we fill the ball $\bxr$ with type $0$ individuals. Conversely, if $1 - \prod_{j = 1}^{k} \omega_{M}(y_{j}) = 1$, then at least one individual is of type $1$, and this time we fill the ball $\bxr$ with type $1$ individuals. Since type $0$ individuals model "ghost" individuals, they are supposed to have a selective disadvantage against "real" type $1$ individuals, hence the formal exclusion of the case $k = 1$ which would not give any advantage to type $1$ individuals. Moreover, $k$ can be interpreted as measuring the strength of the selective advantage of "real" individuals against "ghost" individuals, or in other words, the capacity of "real" individuals to invade an empty environment. 

If $k = 2$, $\lcal_{µ}^{2}$ is the operator introduced in \cite{etheridge2020rescaling} to define and characterize the "selection part" of the SLFV with selection, in the special case for which there are no neutral events and all the individuals in the area affected by a reproduction event are replaced by descendants of the individual reproducing (that is, all reproduction events have an impact of $u = 1$). The proof in \cite{etheridge2020rescaling} of the existence and uniqueness of the $D_{\widetilde{\mcal}_{\lambda}}[0,+\infty)$-valued solution to the martingale problem associated to $\lcal_{µ}^{2}$ can easily be extended to the case $k \geq 2$, by restricting the martingale problem to an increasing sequence of compact subsets of $\rd$ converging to $\rd$. In Section \ref{sec:2}, we will show that this unique solution is in fact $D_{\mcal_{\lambda}}[0,+\infty)$-valued if the initial value belongs to $\mcal_{\lambda}$. 
\begin{thm}\label{thm:forwards}
Let $k \geq 2$, and let $µ$ be a $\sigma$-finite measure on $(0,+\infty)$ satisfying condition (\ref{eqn:condition_intensite}). 
For all $M^{0} \in \mcal_{\lambda}$, there exists a unique $D_{\mcal_{\lambda}}[0,+\infty)$-valued process $(M_{t}^{k})_{t \geq 0}$ such that $M_{0}^{k} = M^{0}$ and, for all $F \in C^{1}(\mathbb{R})$ and $f \in C_{c}(\rd)$, 
\begin{equation*}
\left(\Psi_{F,f}(M_{t}^{k}) - \Psi_{F,f}(M_{0}^{k}) - \int_{0}^{t}\lcal_{µ}^{k}\Psi_{F,f}(M_{s}^{k})ds \right)_{t \geq 0}
\end{equation*}
is a martingale. 
Moreover, the process $(M_{t}^{k})_{t \geq 0}$ is Markovian, and the corresponding semigroup is Feller. 
\end{thm}

The proof of Theorem~\ref{thm:forwards} is a straightforward adaptation of the proof of Theorem~1.2 from  
\cite{etheridge2020rescaling}, combined with Lemma~\ref{lem:bon_espace_SLFV}. Indeed, by Theorem~1.2 from 
\cite{etheridge2020rescaling}, the martingale problem $(\lcal_{µ}^{k},\delta_{M^{0}})$ admits a unique $D_{\widetilde{\mcal}_{\lambda}}[0,+\infty)$-valued solution, and Lemma~\ref{lem:bon_espace_SLFV} from Section~\ref{sec:2} gives a $D_{\mcal_{\lambda}}[0,+\infty)$-valued solution. 

\paragraph{Definition of the $k$-parent SLFV} We can now define the $k$-parent SLFV using the characterization provided by Theorem~\ref{thm:forwards}. 

\begin{defn}[Definition of the $k$-parent SLFV]\label{defn:forwards_k_finite}
Let $k \geq 2$, let $µ$ be a $\sigma$-finite measure on $(0,\infty)$ satisfying (\ref{eqn:condition_intensite}), and let $M^{0} \in \mcal_{\lambda}$. Then, the k-parent spatial $\Lambda$-Fleming-Viot process (or k-parent SLFV) with initial condition $M^{0}$  and associated to $µ$ is the unique solution to the martingale problem $(\lcal_{µ}^{k},\delta_{M^{0}})$ stated in Theorem \ref{thm:forwards}. In particular, the 
k-parent SLFV is a strong Markov process with càdlàg paths a.s. 

Similarly, for all $\omega : \rd \to \{0,1\}$ measurable, the $k$-parent SLFV with initial density $\omega$ and associated to $µ$ is the unique solution to the martingale problem $(\lcal_{µ}^{k},\delta_{M(\omega)})$ with
\begin{equation*}
M(\omega)(dx, d\kappa) := ((\omega(x)\delta_{0}(d\kappa) + (1-\omega(x))\delta_{1}(d\kappa))dx.
\end{equation*}
\end{defn}

Intuitively, the $k$-parent SLFV can be constructed in the following way. Let $M^{0} \in \mcal_{\lambda}$, and let $µ$ be a $\sigma$-finite measure on $(0,\infty)$ satisfying (\ref{eqn:condition_intensite}). Moreover, let $\Pi$ be a Poisson point process on $\mathbb{R} \times \mathbb{R}^{d} \times (0,+\infty)$ with intensity $dt \otimes dx \otimes µ(dr)$. 
Initially, the $k$-parent SLFV is equal to $M^{0}$. The dynamics of the $k$-parent SLFV $(M_{t}^{k})_{t \geq 0}$ is then as follows. If $(t,x,\rcal) \in \Pi$, a reproduction event happens at time $t$ in the ball $\bcal(x,\rcal)$. We sample $k$ types according to the type distribution in the ball $\bcal(x,\rcal)$ at the time $t-$. We interpret these types as the types of $k$ potential "parents". 
With probability
\begin{equation*}
\frac{1}{V_{\rcal}^{k}}\int_{\bxr^{k}}\left[\prod_{j = 1}^{k}\omega_{M_{t-}^{k}}(y_{j})\right]dy_{1}...dy_{k},
\end{equation*}
the $k$ types sampled are $0$, so the $k$ potential parents are of type $0$. In this case, all the individuals in the ball $\bxr$ die, the $k$-th potential parent (of type $0$) fills the ball $\bxr$ with its descendants, which means that we set :
\begin{equation*}
\forall z \in \bxr, \omega_{M_{t}^{k}}(z) = 1.
\end{equation*}
Conversely, with probability
\begin{equation*}
1 - \frac{1}{V_{\rcal}^{k}}\int_{\bxr^{k}}\left[\prod_{j = 1}^{k}\omega_{M_{t-}^{k}}(y_{j})\right]dy_{1}...dy_{k},
\end{equation*}
at least one of the $k$ types sampled is $1$. As in the other case, all the individuals in the ball $\bxr$ die, but this time the first potential parent to be of type $1$ fills the ball $\bxr$ with its descendants, which amounts to setting
\begin{equation*}
\forall z \in \bxr, \omega_{M_{t}^{k}}(z) = 0.
\end{equation*}
The value taken by the density out of the ball $\bxr$ at time $t$ is not affected by this reproduction event. 
We repeat this for each $(t,x,\rcal) \in \Pi$. 

This construction can be made rigorous using arguments adapted from \cite{veber2015spatial}, and will be used in Section \ref{sec:2} to complete the proof of Theorem \ref{thm:forwards}. 

\begin{rem}
The $k$-parent SLFV process is a special case of the SLFV with selection from \cite{forien2017central}, with impact parameter $u = 1$, selection parameter $s = 1$, and selection function $F : x \to x - x^{k}$. It is possible to modify the definition of the $k$-parent SLFV in order to include an impact parameter. However, while this is straightforward for the construction based on the martingale problem, this is not the case for the alternative construction introduced in Section~\ref{sec:2}, which is the one we use to construct the limiting process. 
Indeed, this alternative construction relies on the observation that when $u = 1$ a.s., for appropriate initial conditions, sampling a parent is more or less equivalent to sampling a location. This is no longer true when $u \neq 1$ with positive probability. 
Moreover, in this case, the limiting process obtained when $k \to + \infty$ is expected to exhibit significantly different properties. Therefore, the introduction of an impact parameter is deferred to future work.  
\end{rem}

\begin{rem}\label{rem:tx_pt_affecte}
The condition (\ref{eqn:condition_intensite}) on $µ$ matches the standard condition for the existence of the SLFV \cite{barton2010new}. 
It comes from the fact that a point $x \in \mathbb{R}^{d}$ is affected by a reproduction event at rate :
\begin{equation*}
\int_{\mathbb{R}^{d}}\int_{0}^{+ \infty} \mathds{1}_{y \in \bxr}µ(d\rcal)dy = \int_{0}^{+ \infty} V_{\rcal}µ(d\rcal) \propto \int_{0}^{+ \infty} \rcal^{d}µ(d\rcal).
\end{equation*}
\end{rem}

\begin{rem}\label{rem:type_undefined}
Since the density $\omega_{M}$ is only defined up to a Lebesgue null set, the type of individuals present in a given position $y \in \rd$ cannot be uniquely defined. Therefore, even though intuitively we can first sample parental positions, and deduce parental types from $\omega_{M}$, we cannot formally sample positions in order to sample parental types. 
\end{rem}

A particularly interesting feature of this model is that there exists a dual process of potential ancestors associated to it, which follows the locations of the potential ancestors of a set of individuals. In other words, the genetic diversity in a sample of the population can be determined by going \textit{backwards in time}, and reconstructing the genealogical tree of the sample. For $k = 2$, the dual process is analogous to the Ancestral Selection Graph (ASG) \cite{krone1997ancestral,neuhauser1997genealogy}, but with a spatial structure. We now introduce this dual process, called the $k$-parent ancestral process. 

\subsubsection{The $k$-parent ancestral process}

All the new objects introduced in relation with the dual process will be defined on a new probability space $(\boldsymbol{\Omega}, \boldsymbol{\mathcal{F}}, \boldsymbol{P})$. Let $\egras$ denote the expectation with respect to $\boldsymbol{P}$. As before, we let $µ$ be a $\sigma$-finite measure on $(0,+\infty)$ satisfying condition (\ref{eqn:condition_intensite}), and we let $\overleftarrow{\Pi}$ be a Poisson point process on $\mathbb{R} \times \mathbb{R}^{d} \times (0,+\infty)$ with intensity $dt \otimes dx \otimes µ(d\rcal)$.

\paragraph{Definition}
Let $\mcal_{p}(\rd)$ denote the set of all finite point measures on $\rd$, equipped of the topology of weak convergence. For all $\Xi = \sum_{i = 1}^{l} \delta_{\xi_{i}}  \in \mcal_{p}(\rd)$, for all $x \in \rd$ and $\rcal > 0$, we define 
\begin{align*}
I_{x,\rcal}(\Xi) &= \{i \in \llbracket1,l \rrbracket : ||x-\xi_{i}|| \leq \rcal\} \\
\text{and }S^{\rcal}(\Xi) &= \{x \in \rd : \exists i \in \llbracket 1,l \rrbracket : ||x-\xi_{i}|| \leq \rcal\}.
\end{align*}
In other words, $I_{x,\rcal}(\Xi)$ is the set of all the indices of the points in $\Xi$ which are at distance at most $\rcal$ of $x$, while $S^{\rcal}(\Xi)$ is the set of all the points in $\rd$ which are at distance at most $\rcal$ of a point of $\Xi$. 

\begin{defn}
($k$-parent ancestral process)
Let $\Xi^{0} \in \mcal_{p}(\rd)$. The $k$-parent ancestral process $(\Xi_{t}^{k})_{t \geq 0}$ associated to $µ$ (or equivalently to $\overleftarrow{\Pi}$) and with initial condition $\Xi^{0}$ is the $\mcal_{p}(\rd)$-valued Markov jump process defined as follows. 
\begin{itemize}
\item First, we set $\Xi_{0}^{k} = \Xi^{0}$. 
\item Then, for all $(t,x,\rcal) \in \overleftarrow{\Pi}$, if $I_{x,\rcal}(\Xi_{t-}^{k}) \neq \emptyset$ and if we write
\begin{equation*}
\Xi_{t-}^{k} = \sum_{i = 1}^{N_{t-}^{k}} \delta_{\xi_{t-}^{k,i}},
\end{equation*}
we sample $k$ points $y_{1},...,y_{k}$ independently and uniformly at random in $\bxr$, and we set
\begin{equation*}
\Xi_{t}^{k} := \sum_{i = 1}^{N_{t-}^{k}}\delta_{\xi_{t-}^{k,i}} - \sum_{i \in I_{x,\rcal}(\Xi_{t-}^{k})} \delta_{\xi_{t-}^{k,i}} + \sum_{j = 1}^{k}\delta_{y_{j}}.
\end{equation*}
In other words, we remove all the atoms of $\Xi_{t-}^{k}$ sitting in $\bxr$, and we add $k$ atoms at locations that are i.i.d and uniformly distributed over the ball $\bxr$. 
\end{itemize}
\end{defn}

This process is well-defined, since $N_{t}$ is stochastically bounded by the number $(Y_{t}^{k})_{t \geq 0}$ of particles in a Yule process with $k$ children and with individual branching rate $\int_{0}^{\infty}V_{\rcal}µ(d\rcal) < +\infty$ (see \cite{etheridge2020rescaling} for a proof in the case $k = 2$, which can be generalized to the case $k \geq 2$). 

\paragraph{Martingale problem}
The $k$-parent ancestral process solves a martingale problem that we now introduce. 
For all $F \in C_{b}^{1}(\mathbb{R})$ and $f \in \mathcal{B}(\rd)$, we define the function $\Phi_{F,f} : \mcal_{p}(\rd) \to \mathbb{R}$ by : 
\begin{equation*}
\forall \Xi \in \mcal_{p}(\rd), \Phi_{F,f}(\Xi) = F\left(\int_{\rd} f(x)\Xi(dx)\right) = F(\langle\Xi,f\rangle).
\end{equation*}

We now define the operator $\mathcal{G}_{µ}^{k}$ on the set of functions of the form $\Phi_{F,f}$, which will be at the basis of the martingale problem satisfied by $(\Xi_{t})_{t \geq 0}$. Let $F \in C_{b}^{1}(\mathbb{R})$ and $f \in \mathcal{B}(\rd)$, then for all $\Xi = \sum_{i = 1}^{l}\delta_{\xi_{i}} \in \mcal_{p}(\rd)$, we set~:
\begin{align*}
\mathcal{G}_{µ}^{k}\Phi_{F,f}(\Xi) := \int_{\rd}\int_{0}^{+ \infty}\int_{\bxr^{k}} \frac{\mathds{1}_{x \in S^{\rcal}(\Xi)}}{V_{\rcal}^{k}}&\left[ F\left(\langle\Xi,f\rangle - \sum_{i \in I_{x,\rcal}(\Xi)}f(x_{i}) + \sum_{j = 1}^{k}f(y_{j})\right)\right. \\
& \quad \left.- F(\langle\Xi,f\rangle)\vphantom{\sum_{i = 1}^{k}}\right]
dy_{1}...dy_{k}µ(d\rcal)dx 
\end{align*}
This operator is well defined. Indeed, for all $\Xi \in \mcal_{p}(\rd)$, by Fubini's theorem,
\begin{align*}
|\gcal_{µ}^{k}\Phi_{F,f}(\Xi)| &\leq \int_{0}^{+ \infty}\int_{S^{\rcal}(\Xi)}\int_{\bxr^{k}} \frac{2}{V_{\rcal}^{k}}\, ||F||_{\infty}dy_{1}...dy_{k}dxµ(d\rcal) \\
&\leq 2||F||_{\infty} \,  \int_{0}^{+ \infty}\mathrm{Vol}(S^{\rcal(\Xi)})µ(d\rcal) \\
&\leq 2||F||_{\infty} \, \Xi(\rd) \, \int_{0}^{+ \infty}V_{\rcal}µ(d\rcal)\\
&< + \infty 
\end{align*}
by Condition (\ref{eqn:condition_intensite}).

\begin{prop}\label{prop:backwards_martingale}
Let $\Xi^{0} \in \mcal_{p}(\rd)$, and let $(\Xi_{t})_{t \geq 0}$ be the $k$-parent ancestral process of initial condition $\Xi^{0}$ associated to $µ$. Then, for all $F \in C_{b}^{1}(\mathbb{R})$ and for all $f \in \mathcal{B}(\rd)$, the process
\begin{equation*}
\left(
\Phi_{F,f}(\Xi_{t}) - \Phi_{F,f}(\Xi_{0}) - \int_{0}^{t}\gcal_{µ}^{k} \Phi_{F,f}(\Xi_{s})ds
\right)_{t \geq 0}
\end{equation*}
is a martingale. 
\end{prop}

The proof of Proposition~\ref{prop:backwards_martingale} is a direct adaptation of the proof of Proposition~1.5 from \cite{etheridge2020rescaling}.

\paragraph{Duality relation}
Intuitively, the $k$-parent ancestral process records the locations of the potential ancestors of a given sample of individuals. However, because densities are only defined up to a Lebesgue null set, it is not possible to assign uniquely a type to an individual located at $x \in \rd$ looking at the value of the density at this point. Therefore, as in \cite{etheridge2020rescaling}, in order to give a duality relation between the $k$-parent SLFV and the $k$-parent ancestral process, we will need to consider a distribution of sampling locations, rather than fixed locations.

More specifically, for all $l \geq 1$ and $x_{1},...,x_{l} \in (\rd)^{l}$, we define :
\begin{equation*}
\Xi[x_{1},...,x_{l}] := \sum_{i = 1}^{l}\delta_{x_{i}} \in \mcal_{p}(\rd).
\end{equation*}
If $\Psi$ is a density function on $(\rd)^{l}$ with respect to Lebesgue measure, let $µ_{\Psi}$ be the law of the random point measure $\sum_{i = 1}^{l}\delta_{X_{i}}$, where $(X_{1},...,X_{l})$ is sampled according to $\Psi$. 
If $M \in \mcal_{\lambda}$ and $\Xi = \sum_{i = 1}^{l}\delta_{x_{i}} \in \mcal_{p}(\rd)$, we set :
\begin{equation*}
D(M, \Xi) := \prod_{i = 1}^{l} \omega_{M}(x_{i}).
\end{equation*}
Notice that for all $l \in \mathbb{N}^{*}$ and for all density functions $\Psi$ on $(\rd)^{l}$, 
\begin{align*}
\int_{\mcal_{p}(\rd)}D(M,\Xi)µ_{\Psi}(d\Xi)
&= \int_{(\rd)^{l}}\Psi(x_{1},...,x_{l})\left\{\prod_{j = 1}^{l}\omega_{M}(x_{j})\right\}dx_{1}...dx_{l} \\
&= \int_{(\rd\times\{0,1\})^{l}}\Psi(x_{1},...,x_{l})\left\{\prod_{j = 1}^{l}\mathds{1}_{0}(\kappa_{j})\right\}M(dx_{1},d\kappa_{1})...M(dx_{l},d\kappa_{l})
\end{align*}
does not depend on the choice of a density $\omega_{M}$ of $M$. 

A straightforward adaptation of the proof of Proposition 1.7 in \cite{etheridge2020rescaling} to the case $k \geq 2$ leads to the following proposition.

\begin{prop}\label{prop:dualite}
Let $k \geq 2$.
Let $M^{0} \in \mcal_{\lambda}$, let $l \in \mathbb{N}^{*}$, and let $\Psi$ be a density function on $(\rd)^{l}$. Let $µ$ be a $\sigma$-finite measure on $(0,+\infty)$ satisfying (\ref{eqn:condition_intensite}). Let $(M_{t}^{k})_{t \geq 0}$ be the k-SLFV with initial condition $M^{0}$ associated to $µ$ and $(\Xi_{t}^{k})_{t \geq 0}$ be the $k$-parent ancestral process and associated to $µ$. Then, for all $t \geq 0$,
\begin{equation*}
\int_{\mcal_{p}(\rd)}\esp[D(M_{t}^{k},\xi)|M_{0}^{k} = M^{0}]µ_{\Psi}(d\xi)
= \egras[D(M^{0},\Xi_{t}^{k})|\Xi_{0}^{k} \sim µ_{\Psi}].
\end{equation*}
Equivalently, for all $t \geq 0$, 
\begin{equation*}
\esp_{M^{0}}\left[
\int_{(\rd)^{l}}\Psi(x_{1},...,x_{l})\left\{\prod_{j = 1}^{l}\omega_{M_{t}^{k}}(x_{j})\right\}dx_{1}...dx_{l}
\right]
= \int_{(\rd)^{l}}\Psi(x_{1},...,x_{l})\egras_{\Xi[x_{1},...,x_{l}]} \left[\prod_{j = 1}^{N_{t}^{k}}\omega_{M^{0}}(\xi_{t}^{k,j})\right]dx_{1}...dx_{l}.
\end{equation*}
\end{prop}

In other words, informally, the individuals living in locations $x_{1}$,~...,~$x_{l}$ at time~$t$ are all of type~$0$ if, and only if all the potential ancestors of this sample of individuals at time~$0$ are of type~$0$.

\subsection{The $\infty$-parent SLFV}\label{subsec:1.2}
When $k \to + \infty$, the $k$-parent SLFV is expected to converge to the $\infty$-parent SLFV. Before defining this process rigorously in Theorem~\ref{thm:first_construction_infty_slfv}, we first introduce it somewhat informally. 

\paragraph{Definition}
Let $µ$ be a $\sigma$-finite measure on $(0,+\infty)$ satisfying Condition (\ref{eqn:condition_intensite}), and let $\Pi$ be a Poisson point process on $\mathbb{R}\times\rd\times (0,+\infty)$ with intensity $dt \otimes dx \otimes µ(d\rcal)$. Let also $M^{0} \in \mcal_{\lambda}$. We start the $\infty$-parent spatial $\Lambda$-Fleming Viot process $(M_{t}^{\infty})_{t \geq 0}$, or $\infty$-parent SLFV, at $M_{0}^{\infty} = M^{0}$. Then, if $(t,x,\rcal) \in \Pi$, as before, we consider that a reproduction event occurs in the ball $\bxr$ at time $t$. However, this time we do not sample a finite number of potential parents. Instead, we look at the value of the integral
\begin{equation*}
\int_{\bxr} \left(1 - \omega_{M_{t-}^{\infty}}(z)\right) dz,
\end{equation*}
which amounts to sampling an infinite number of potential parents over the ball $\bxr$ and looking at the proportion of them which are of the "existing" type (i.e, type $1$). 

If $\int_{\bxr} \left(1 - \omega_{M_{t-}^{\infty}}(z)\right) dz = 0$, we consider that the parent which reproduces is of type $0$, and we set~:
\begin{equation*}
\forall z \in \bxr, \omega_{M_{t}^{\infty}}(z) = 1.
\end{equation*}
Note that in this case, the "parent" which reproduces was "sampled" at a location which is uniformly distributed over the ball $\bxr$. 

Conversely, if $\int_{\bxr} \left(1 - \omega_{M_{t-}^{\infty}}(z)\right) dz \neq 0$, there is a non negligible amount of individuals of type $1$ in $\bxr$. We impose that it is one of them which reproduces, and in such a way that its offspring invades the whole region. That is, we set :
\begin{equation*}
\forall z \in \bxr, \omega_{M_{t}^{\infty}}(z) = 0.
\end{equation*}

Formally, the $\infty$-parent SLFV is constructed in Section~\ref{sec:2} using coupled $k$-parent SLFVs, leading to the following characterization. 
\begin{thm}\label{thm:first_construction_infty_slfv}
Let $\omega : \mathbb{R}^{d} \to \{0,1\}$ be measurable, and for all $k \geq 2$, let $(M_{t}^{k})_{t \geq 0}$ be the $k$-parent SLFV with initial density $\omega$ and associated to $µ$. Then, it is possible to couple the random variables $(M^{k})_{k \geq 2}$ in such a way that each admits a density $(\omega_{t}^{k})_{t \geq 0}$ such that
\begin{equation*}
\forall t \geq 0, \forall x \in \mathbb{R}^{d}, (\omega_{t}^{k}(x))_{t \geq 0} \text{ is nonincreasing with respect to $k$}. 
\end{equation*}
For all $t \geq 0$ and $x \in \mathbb{R}^{d}$, we then set
\begin{equation*}
\omega_{t}^{\infty}(x) = \lim\limits_{k \to + \infty} \omega_{t}^{k}(x),
\end{equation*}
and we define the $\infty$-parent SLFV $(M_{t}^{\infty})_{t \geq 0}$ as the unique $\mcal_{\lambda}$-valued process with density $(\omega_{t}^{\infty})_{t \geq 0}$. 
\end{thm}
Theorem~\ref{thm:first_construction_infty_slfv} will be a direct consequence of Definition~\ref{defn:infty_SLFV} and Lemma~\ref{lem:infty_markovian} from Section~\ref{sec:2}.

This first construction of the $\infty$-parent SLFV shows that this process can be considered as the limit of the $k$-parent SLFV when $k \to + \infty$, in the sense that the finite dimensional distributions of the sequence of coupled $k$-parent SLFVs converge almost surely to the ones of $(M_{t}^{\infty})_{t \geq 0}$ (see Lemma~\ref{lem:cvg_vague_processus_tps}). In Section~\ref{subsec:convergence_result_forward}, we establish the following convergence result, which states that the $k$-parent SLFV converges in distribution towards the $\infty$-parent SLFV in $D_{\mcal_{\lambda}}[0,\infty)$. 

\begin{thm}\label{thm:convergence_result}
Let $M^{0} \in \mcal_{\lambda}$, and let $µ$ be a $\sigma$-finite measure on $(0,+\infty)$ satisfying Condition~(\ref{eqn:condition_intensite}). Let $(M_{t}^{\infty})_{t \geq 0}$ be the $\infty$-parent SLFV with initial condition $M^{0}$ and associated to $µ$. For all $k \geq 2$, let $(M_{t}^{k})_{t \geq 0}$ be the $k$-parent SLFV with initial condition~$M^{0}$ and associated to~$µ$. Then, $(M^{k})_{k \geq 2}$ converges in distribution towards $M^{\infty}$ in $D_{\mcal_{\lambda}}[0,\infty)$. 
\end{thm}
The proof of this result can be found at the end of Section~\ref{subsec:convergence_result_forward}.  

\begin{rem}
While this article focuses on the link between the $k$-parent SLFV and the $\infty$-parent SLFV, this process is also conjectured to appear as the limit of SFLVs featuring different forms of strong selection against ghost individuals, for instance if real individuals only reproduce when at least $k_{1} \leq k$ potential parents sampled are real (with $k_{1}$ constant when $k \to + \infty$). 
Other forms of selection are expected to yield different limiting processes. For instance, if real individuals only reproduce when at least $\alpha k$, $\alpha \in (0,1)$ potential parents sampled are real, then we expect local extinctions to be still visible in the limit, when the density of real individuals in the affected area is too low. The approach used in this article can possibly be applied to the study of these other limiting regimes as well. 
\end{rem}

\begin{rem}
Since the density of type~$0$ individuals is $\{0,1\}$-valued, the $\infty$-parent SLFV can also be seen as taking its values in the equivalence class of all the measurable subsets of $\mathbb{R}^{d}$, where two measurable sets $S_{1}, S_{2} \subseteq \mathbb{R}^{d}$ are equivalent if, and only if $S_{1} \Delta S_{2}$ has null Lebesgue measure. 
In this case, it will actually be easier to consider that the subset represents the area occupied by \textit{real} individuals. This corresponds to the characterization of the $\infty$-parent SLFV which will be introduced in Section~\ref{sec:3} and Proposition~\ref{prop:stochastic_growth_model}. 
\end{rem}

\paragraph{Martingale problem}
Like the $k$-parent SLFV, the $\infty$-parent SLFV (as constructed in Theorem~\ref{thm:first_construction_infty_slfv}) is solution to a martingale problem. However, and in contrast with the case of the $k$-parent SLFV, the condition (\ref{eqn:condition_intensite}) on $µ$ will not be sufficient to ensure that this solution is unique. Instead, we will need the following stronger condition.

\begin{defn}
Let $a_{d} > 0$ be such that the minimal number of d-dimensional balls of radius $1$ needed to cover the boundary of an hypersphere of radius $n$ in $d$ dimensions is bounded from above by $a_{d}\times n^{d-1}$ for every $n \geq 1$. 
A $\sigma$-finite measure $µ$ on $\mathbb{R}_{+}^{*}$ is said to satisfy Condition~(\ref{eqn:condition_infty_SLFV}) if it satisfies Condition~(\ref{eqn:condition_intensite}), and if there exists $\rcal > 0$ such that 
\begin{equation}\label{eqn:condition_infty_SLFV}
\sum_{n = 1}^{+ \infty} \left(\int_{(n-1)\rcal}^{n\rcal} (\rcal + r)^{d} µ(dr)\right) \left(a_{d} \, n^{d-1} + 1\right) < + \infty.
\end{equation}
\end{defn}
This stronger condition appears in Section~\ref{subsec:3.1}, and is needed to ensure that the dual of the $\infty$-parent SLFV is well-defined. 

Examples of $\sigma$-finite measures $µ$ on $\mathbb{R}_{+}^{*}$ satisfying Condition (\ref{eqn:condition_infty_SLFV}) are the following :
\begin{enumerate}
\item Measures $µ$ on $\mathbb{R}_{+}^{*}$ having a bounded support.
\item Measures $µ$ on $\mathbb{R}_{+}^{*}$ of the form $\alpha \times (1+r)^{-3d-1} dr$, with $\alpha > 0$.
\end{enumerate}

We recall that for all $F \in C^{1}(\mathbb{R})$ and $f \in C_{c}(\rd)$, the function $\Psi_{F,f}$ is defined as
\begin{equation*}
\forall M \in \mathcal{M}_{\lambda}, \Psi_{F,f}(M) = F(\langle \omega_{M},f \rangle). 
\end{equation*}
We define the operator $\lcal_{µ}^{\infty}$ on functions of the form $\Psi_{F,f}$ where $F \in C^{1}(\mathbb{R})$ and $f \in C_{c}(\rd)$ in the following way. For all $M \in \mcal_{\lambda}$, we set~:
\begin{align*}
\lcal_{µ}^{\infty}\Psi_{F,f}(M) := \int_{0}^{+ \infty}\int_{Supp^{\rcal}(f)}
&\left[ \delta_{0}\left(\int_{\bxr}\left(1-\omega_{M}(z)\right)dz\right) \,  F(\langle\Theta_{x,\rcal}^{+}(\omega_{M}),f\rangle)\right. \\
&+  \left(1-\delta_{0}\left(\int_{\bxr}\left(1-\omega_{M}(z)\right)dz\right)\right) \, F(\langle\Theta_{x,\rcal}^{-}(\omega_{M}),f\rangle) \\
&\left.- F(\langle\omega_{M},f\rangle) \vphantom{\int_{\bxr}}
\right]dxµ(d\rcal), 
\end{align*}
where the function $\delta_{0}$ is defined as $\delta_{0} : x \to \mathds{1}_{\{0\}}(x)$. 
Note that if $\delta_{0}\left(\int_{\bxr}\left(1-\omega_{M}(z)\right)dz\right)= 1$, then for all $y \in \bxr$ except possibly on a Lebesgue null set, 
\begin{equation*}
\Theta_{x,\rcal}^{+}(\omega_{M})(y) = \omega_{M}(y).
\end{equation*}
In other words, the ball $\bxr$ is already completely void of "existing" individuals, and filling it with "ghost" individuals does not change anything. 
Therefore, we also have
\begin{align*}
\lcal_{µ}^{\infty}\Psi_{F,f}(M) = \int_{0}^{+ \infty}\int_{Supp^{\rcal}(f)}&\left(1-\delta_{0}\left(\int_{\bxr}\left(1-\omega_{M}(z)\right)dz\right)\right) \\
&\times \Big[F(\langle\Theta_{x,\rcal}^{-}(\omega_{M}),f\rangle) - F(\langle\omega_{M},f\rangle)
\Big]dxµ(d\rcal).
\end{align*}
In Section \ref{sec:5}, we show that this operator is well-defined, even if $µ$ satisfies Condition (\ref{eqn:condition_intensite}) rather than Condition (\ref{eqn:condition_infty_SLFV}). 
If $µ$ satisfies Condition (\ref{eqn:condition_infty_SLFV}), then the associated martingale problem can be used to define and fully characterize the $\infty$-parent SLFV. If $µ$ satisfies only Condition (\ref{eqn:condition_intensite}), then the $\infty$-parent SLFV is still solution to the martingale problem, but we no longer know whether this solution is unique, as stated below.

\begin{thm}\label{thm:characterization_infty_SLFV}
Let $\omega : \rd \to \{0,1\}$, let $M^{0} \in \mcal_{\lambda}$, and let $µ$ be a $\sigma$-finite measure on $(0,+\infty)$ satisfying Condition (\ref{eqn:condition_intensite}). Then, the $\infty$-parent SLFV with initial condition $M^{0}$ and associated to $µ$  
is a solution to the martingale problem for $(\lcal_{µ}^{\infty},\delta_{M^{0}})$.

Moreover, if $µ$ satisfies Condition (\ref{eqn:condition_infty_SLFV}), the martingale problem associated to $(\lcal_{µ}^{\infty},\delta_{M^{0}})$is well-posed, and the $\infty$-parent SLFV with initial condition $M^{0}$ and associated to $µ$ is the unique solution to it in $D_{\mcal_{\lambda}}[0,+\infty)$.
\end{thm}

The proof of Theorem~\ref{thm:characterization_infty_SLFV} will be carried out in Sections~\ref{sec:3} and~\ref{sec:4}, and comprises three steps. First, in Section~\ref{subsec:2.2}, we show that the martingale problem $(\lcal_{µ}^{\infty},\delta_{M^{0}})$ admits at least one solution: the $\infty$-parent SLFV constructed in Theorem~\ref{thm:first_construction_infty_slfv}. 
Then, in Section~\ref{sec:3}, we construct a candidate for the dual process associated to the $\infty$-parent SLFV (see Definition~\ref{defn:infty_parents_ancestors}), and in Section~\ref{subsec:4.3}, we establish the corresponding duality relation (which is stated in Proposition~\ref{prop:dualite_kinf}). 

Theorems \ref{thm:characterization_infty_SLFV} and \ref{thm:first_construction_infty_slfv} give us two distinct characterizations of the $\infty$-parent SLFV. If $µ$ satisfies Condition~(\ref{eqn:condition_infty_SLFV}) and if the area occupied by real individuals has finite Lebesgue measure, then we can provide a third construction of the $\infty$-parent SLFV, which is closer to stochastic growth models, and which can be considered as self-dual. This third construction is introduced in Proposition~\ref{prop:stochastic_growth_model}. 

\subsection{Dual of the $\infty$-parent SLFV}\label{subsec:1.3}
As for the $k$-parent SLFV, the $\infty$-parent SLFV also has a dual process of potential ancestors, that we now introduce.

\paragraph{Definition}
Let $\mathcal{E}^{c}$ be the set of Lebesgue measurable and connected subsets of $\rd$ whose Lebesgue measure is finite and strictly positive. Let $\mathcal{E}^{cf}$ be the set of all finite unions of elements of $\mathcal{E}^{c}$. If $E \in \mathcal{E}^{cf}$ can be written as $E = \cup_{i = 1}^{l} E^{i}$ where for all $1 \leq i \leq l$, $E^{i} \in \mathcal{E}^{c}$, we let $m(E) = m(E^{1},...,E^{l})$ be the measure on $\rd$ defined by $m(E)(dx) := \mathds{1}_{x \in E} dx$, and we set :
\begin{equation*}
\mcal^{cf} := \{m(E) : E \in \mathcal{E}^{cf}\}. 
\end{equation*}

\begin{defn}[$\infty$-parent ancestral process]\label{defn:infty_parents_ancestors}  
Let $µ$ be a $\sigma$-finite measure on $(0,+\infty)$ satisfying Condition (\ref{eqn:condition_infty_SLFV}). Let $\overleftarrow{\Pi}$ be a Poisson point process on $\mathbb{R}_{+} \times \rd \times (0,+\infty)$ with intensity $ dt \otimes dx \otimes µ(d\rcal)$, defined on the probability space $(\boldsymbol{\Omega}, \boldsymbol{\mathcal{F}}, \boldsymbol{P})$. \\
Let $\Xi^{0} \in \mathcal{M}^{cf}$. Then, the $\mcal^{cf}$-valued $\infty$-parent ancestral process $(\Xi_{t}^{\infty})_{t \geq 0}$ with initial condition $\Xi^{0}$ and associated to $µ$ (or equivalently to $\overleftarrow{\Pi}$) is defined in the following way.
\begin{itemize}
\item First, we set $\Xi_{0}^{\infty} = \Xi^{0}$. 
\item Then, if for all $t \geq 0$, we write $\Xi_{t}^{\infty}$ as
\begin{equation*}
\Xi_{t}^{\infty} = m(E_{t}^{\infty}),
\end{equation*}
and for all $(t,x, \rcal) \in \overleftarrow{\Pi}$, if $E_{t-}^{\infty} \cap \bxr$ has a non zero Lebesgue measure,
\begin{equation*}
\Xi_{t}^{\infty} = m(E_{t-}^{\infty} \cup \bxr), 
\end{equation*}
while nothing happens if $\mathrm{Vol}(E_{t-}^{\infty} \cap \bxr) = 0$. 
\end{itemize}
Moreover, this process is Markovian. 
\end{defn}

In Section~\ref{sec:3}, we will show that the jump rate of this process is finite. Therefore, this definition gives us an explicit construction of the $\infty$-parent ancestral process. 

Informally, the process evolves as follows. Initially, the area covered by the $\infty$-parent ancestral process is $E_{0}^{\infty}$. Then, whenever a new reproduction event intersects the covered area, it is added to the process. 

\begin{rem}
Note that the case $E_{t-}^{\infty} \cap \bxr = \bxr$ is equivalent to $E_{t-}^{\infty} \cup \bxr = E_{t-}^{\infty}$, and hence does not correspond to a visible jump of $(\Xi_{t}^{\infty})_{t \geq 0}$. 
\end{rem}

In Section~\ref{subsec:convergence_dual_process} (see Proposition~\ref{prop:cvg_dual}), we will explain to what extent this process can be considered as the limit of the $k$-parent ancestral process when $k \to + \infty$, even though both ancestral processes are defined on different state spaces. 

\begin{rem}
Like the $\infty$-parent SLFV, the $\infty$-parent ancestral process can also be seen as taking its values in the equivalence class of measurable subsets of $\mathbb{R}^{d}$. However, in order to keep the analogy with the $k$-parent ancestral process, we consider it as a measure-valued process. 
\end{rem}

\paragraph{A third characterization of the $\infty$-parent SLFV}
The definition of the $\infty$-parent ancestral process is very close to the informal definition of the $\infty$-parent SLFV, the main difference (apart from the state space) being the initial condition. And indeed, 
if the area initially occupied by real individuals in the $\infty$-parent SLFV has finite Lebesgue measure, then the $\infty$-parent SLFV follows the same dynamics as the $\infty$-parent ancestral process, in the following sense. 

\begin{prop}\label{prop:stochastic_growth_model}
Let $E^{0} \in \mathcal{E}^{cf}$, and let $µ$ be a $\sigma$-finite measure on $(0,+\infty)$ satisfying Condition~(\ref{eqn:condition_infty_SLFV}). Let $(\Xi_{t}^{\infty})_{t \geq 0} = (m(E_{t}^{\infty}))_{t \geq 0}$ be the $\infty$-parent ancestral process with initial condition $m(E^{0})$ and associated to $µ$. For all $t \geq 0$, we set
\begin{equation*}
\widehat{\omega}^{\infty}_{t} = 1 - \mathds{1}_{E^{\infty}_{t}}, 
\end{equation*}
and let $(\widehat{M}^{\infty}_{t})_{t \geq 0}$ be the $\mcal_{\lambda}$-valued process with density $(\widehat{\omega}^{\infty}_{t})_{t \geq 0}$. Then, $(\widehat{M}^{\infty}_{t})_{t \geq 0}$ is solution to the martingale problem $(\lcal_{µ}^{\infty}, \delta_{M_{0}})$. 
\end{prop}

By Theorem~\ref{thm:characterization_infty_SLFV}, this means that $(\widehat{M}^{\infty}_{t})_{t \geq 0}$ is then the $\infty$-parent SLFV with initial condition $1 - \mathds{1}_{E^{0}}$ and associated to $µ$. This provides a third characterization of the $\infty$-parent SLFV, which corresponds to its original informal definition and is close to models of continuous first-passage percolation. 

The proof of Proposition~\ref{prop:stochastic_growth_model} can be found at the end of Section~\ref{sec:3}, and uses the characterizations of the $\infty$-parent SLFV and its dual as the unique solutions to martingale problems. 

\paragraph{Duality relation}
For all $M \in \mcal_{\lambda}$ with density $\omega$ and for all $\Xi = m(E) \in \mcal^{cf}$, we set :
\begin{equation*}
\widetilde{D}(M, \Xi) := \delta_{0}\left( \int_{E}\left(1 - \omega(x)\right)dx\right).
\end{equation*}
If we know the value of $\widetilde{D}(M,\Xi)$ for all $\Xi \in \mcal^{cf}$, since $\omega$ is $\{0,1\}$-valued, we know the value of $\omega$ everywhere up to a Lebesgue null set, and so we have completely characterized $M$. Therefore, the following duality result shows that the solution to the martingale problem associated to $\lcal_{µ}^{\infty}$ is unique. 

\begin{prop}\label{prop:dualite_kinf}
Let $µ$ be a $\sigma$-finite measure on $(0,+\infty)$ satisfying Condition (\ref{eqn:condition_infty_SLFV}). 
Let $M^{0} \in \mcal_{\lambda}$, and let $(M_{t}^{\infty})_{t \geq 0}$ be a solution to the martingale problem associated to $(\lcal_{µ}^{\infty},\delta_{M^{0}})$. Then, for all $t \geq 0$ and for all $E^{0} \in \mcal^{cf}$, 
\begin{equation*}
\esp_{M^{0}}\left[
\widetilde{D}(M_{t}^{\infty},m(E^{0}))
\right] = \mathbf{E}_{m(E^{0})} \left[
\widetilde{D}(M^{0},\Xi_{t}^{\infty})\right],
\end{equation*}
where $(\Xi_{t}^{\infty})$ is the $\infty$-parent ancestral process of initial condition $m(E^{0})$ and associated to $µ$. Equivalently, for every $t \geq 0$, if $\omega_{t}^{\infty}$ and $\omega_{0}^{\infty}$ are $\{0,1\}$-valued densities of $M_{t}^{\infty}$ and $M^{0}$,
\begin{equation*}
\esp\left[
\delta_{0}\left(\int_{E^{0}}\left(1-\omega_{t}^{\infty}(x)\right)dx
\right)\right] = \mathbf{E}_{m(E^{0})}\left[
\delta_{0}\left(\int_{E_{t}^{\infty}}\left(1 - \omega_{0}^{\infty}(x)\right)dx
\right)\right].
\end{equation*}
\end{prop}

In particular, this result implies that if the area initially occupied by real individuals in the $\infty$-parent SLFV has finite Lebesgue measure, then the process can be considered as self-dual. 
The proof of Proposition \ref{prop:dualite_kinf} can be found in Section~\ref{subsec:4.3}. 

\section{First characterization of the $\infty$-parent SLFV}\label{sec:2}
In this section, we provide the rigorous construction of the process introduced in Section~\ref{subsec:1.2}. 

\subsection{Alternative construction of the k-parent SLFV}\label{subsec:2.1}
In order to construct the $\infty$-parent SLFV rigorously, we start by introducing an alternative construction of the $k$-parent SLFV, based on a variant of its dual. It relies on the sampling of parental locations \textit{along with} reproduction events, and is an adaptation of the concept of \textit{parental skeleton} presented in Section 2.3.1 of \cite{veber2015spatial}. This construction allows to couple $k$-parent SLFVs for different values of $k$, and is therefore also of interest in its own right. 

In all that follows, let $µ$ be a $\sigma$-finite measure on $(0,+\infty)$ satisfying Condition (\ref{eqn:condition_intensite}).
Let $U = \bcal(0,1)^{\mathbb{N}}$, and let $\tilde{u}$ be the law of a sequence of i.i.d random variables $(\mathcal{P}_{n})_{n \geq 1}$ uniformly distributed over $\bcal(0,1)$. We will call an element of $U$ a \textit{sequence of potential parents}. 
Let us now extend the Poisson point process $\Pi$ considered earlier by adding to each event a countable sequence of locations of potential parents. Indeed, let $\Pi^{c}$ be a Poisson point process on $\mathbb{R} \times \rd \times (0,+\infty) \times U$ with intensity 
\begin{equation*} 
dt \otimes dx \otimes µ(d\rcal) \otimes \tilde{u}(d(p_{n})_{n \geq 1}).
\end{equation*}
Then for all $(t,x,\rcal,(p_{n})_{n \geq 1}) \in \Pi^{c}$,
\begin{itemize}
\item as before, $t$ can be interpreted as the time at which the reproduction event occurs, and we can see $\bxr$ as being the area affected by the reproduction event.
\item For all $n \geq 1$, $x + \rcal \times p_{n}$
is uniformly distributed over the ball $\bxr$, and can be interpreted as the location of the $n$-th potential parent sampled, if at least $n$ potential parents have to be sampled.\end{itemize}

We start by defining the variant of the $k$-parent ancestral process, on which the alternative construction of the $k$-parent SLFV is based. 

\begin{defn}[Quenched $k$-parent ancestral process] \label{defn:quenched_k_dual}
Let $k \geq 2$, let $\Xi^{0} \in M_{p}(\rd)$, and let $\tilde{t} \geq 0$. The $k$-parent ancestral process $(\Xi_{k,t}^{\Pi^{c},\tilde{t},\Xi^{0}})_{t \geq 0}$ associated to $\Pi^{c}$, started at time $\tilde{t}$ and with initial condition $\Xi^{0}$ is the $\mcal_{p}(\rd)$-valued Markov jump process defined as follows. 
\begin{itemize}
\item First, we set $\Xi_{k,0}^{\Pi^{c},\tilde{t},\Xi^{0}} = \Xi^{0}$. 
\item Then, for all $(t,x,\rcal,(p_{n})_{n \geq 1}) \in \Pi^{c}$ such that $t \leq \tilde{t}$, recalling that for $\Xi = \sum_{i = 1}^{l} \delta_{\xi_{i}} \in \mcal_{p}(\rd)$, $I_{x,\rcal}(\Xi) = \{i \in \llbracket 1,l \rrbracket : ||x - \xi_{i}|| \leq \rcal\}$, if 
\begin{equation*}
I_{x,\rcal}(\Xi_{k,(\tilde{t}-t)_{-}}^{\Pi^{c},\tilde{t},\Xi^{0}}) \neq \emptyset,
\end{equation*}
then for all $1 \leq l \leq k$, we set
\begin{align*}
y_{l} &:= x + \rcal  p_{l}\\
\text{and \quad\quad   }
\Xi_{k,\tilde{t}-t}^{\Pi^{c},\tilde{t},\Xi^{0}} 
&:= \Xi_{k,(\tilde{t}-t)-}^{\Pi^{c},\tilde{t},\Xi^{0}} 
- \sum_{x' \in I_{x,\rcal}(\Xi_{k,(\tilde{t}-t)-}^{\Pi^{c},\tilde{t},\Xi^{0}})} \delta_{x'}
+ \sum_{l = 1}^{k} \delta_{y_{l}}.
\end{align*}
\end{itemize}
\end{defn}

It is straightforward to check that this process has the same distribution as the $k$-parent ancestral process associated to $µ$ and with initial condition $\Xi^{0}$. Its interest is twofold. First, conditionally on $\Pi^{c}$, $(\Xi_{k,t}^{\Pi^{c},\tilde{t},\Xi^{0}})_{t \geq 0}$ is deterministic. Moreover, if for all $\Xi = \sum_{i = 1}^{l} \delta_{x_{i}} \in \mcal_{p}(\rd)$, we denote the set of atoms of $\Xi$ by
\begin{equation*}
A(\Xi) := \{x_{i} : i \in \llbracket 1,l \rrbracket\},
\end{equation*}
then the process satisfies the following property, which will be useful in the coupling that we will introduce later.

\begin{lem}\label{lem:couplage_k_k'}
Let $2 \leq k \leq k'$, let $\Xi^{0} \in \mcal_{p}(\rd)$, let $\tilde{t} \geq 0$, and let $\Pi^{c}$ be a Poisson point process on $\mathbb{R} \times \rd \times (0,+\infty) \times U$ with intensity $dt \otimes dx \otimes µ(d\rcal) \otimes \tilde{u}(d(p_{n})_{n \geq 1})$. 

Then, for all $t \geq 0$,
\begin{align*}
A(\Xi_{k,t}^{\Pi^{c},\tilde{t},\Xi^{0}}) &\subseteq A(\Xi_{k',t}^{\Pi^{c},\tilde{t},\Xi^{0}}). \\
\intertext{In particular, for all $t \geq 0$ and $x \in \rd$,}
A(\Xi_{k,t}^{\Pi^{c},\tilde{t},\delta_{x}}) &\subseteq A(\Xi_{k',t}^{\Pi^{c},\tilde{t},\delta_{x}}).
\end{align*}
\end{lem}

\begin{rem}
Since $A(\Xi)$ is a set, if there exists $i \neq j$ such that $x_{i} = x_{j}$, then $x_{i}$ appears only once in $A(\Xi)$. 
\end{rem}

Intuitively, the idea behind this lemma is the following. Since the coupled $k$-parent and $k'$-parent ancestral processes are based on the same extended Poisson point process of reproduction events, their evolutions are determined by the same reproduction events. Moreover, since $k' \geq k$, all the potential parents which are involved in the dynamics of the $k$-parent ancestral process are also potential parents for the $k'$-ancestral process. Therefore, we can consider that the $k$-parent ancestral process is embedded in the $k'$-parent ancestral process. 

We now introduce an alternative way of constructing the $k$-parent SLFV, by associating it to the extended Poisson point process $\Pi^{c}$.

\begin{defn}[Quenched $k$-parent SLFV]\label{defn:quenched_k_SLFV}
Let $k \geq 2$, and let $\omega : \rd \to \{0,1\}$ be a measurable function. The $k$-parent SLFV $(M_{k,t}^{\Pi^{c},\omega})_{t \geq 0}$ associated to $\Pi^{c}$ and of initial density $\omega$ is the $\mcal_{\lambda}$-valued Markov process defined as follows. 
\begin{itemize}
\item First, we set $\omega_{k,0}^{\Pi^{c},\omega} = \omega$.
\item Then, for all $t \geq 0$ and for all $x \in \rd$, we set
\begin{equation}\label{eqn:densite_quenched_k_SLFV}
\omega_{k,t}^{\Pi^{c},\omega}(x) := \prod_{y \in A(\Xi_{k,t}^{\Pi^{c},t,\delta_{x}})} \omega(y).
\end{equation}
\item We conclude by setting for all $t \geq 0$, 
\begin{equation*}
M_{k,t}^{\Pi^{c},\omega}
:= ((\omega_{k,t}^{\Pi^{c},\omega}(x) \delta_{0}(d\kappa)
+ (1-\omega_{k,t}^{\Pi^{c},\omega}(x)) \delta_{1}(d\kappa))dx.
\end{equation*}
$(\omega_{k,t}^{\Pi^{c},\omega})_{t \geq 0}$ will be called the density of the $k$-parent SLFV associated to $\Pi^{c}$ and of initial condition~$\omega$.
\end{itemize}
\end{defn}

Note that $\omega_{k,t}^{\Pi^{c},\omega}(x)$ in Eq. (\ref{eqn:densite_quenched_k_SLFV}) is thus equal to $1$ if and only if all potential ancestors at time $0$ of the individuals at $x$ at time $t$ are of type $0$, i.e are all ghosts.

We show below that this process corresponds to another way of constructing the $k$-parent SLFV using the parental skeleton, and in particular, that $(M_{k,t}^{\Pi^{c},\omega})_{t \geq 0} \in D_{M_{\lambda}}[0,+\infty)$. This alternative construction will allow us to couple SLFV processes with different numbers of potential parents, using the same Poisson process. However, even though it is possible to define the $k$-parent SLFV for an \textit{initial condition} $M \in \mcal_{\lambda}$ instead of an initial density $\omega$ of $M$, this coupling can only be used if all processes are constructed using the same initial density. 

\begin{proof}
In order for the process to have a chance to correspond to the $k$-parent SLFV, we first need to check that
\begin{equation*}
(M_{k,t}^{\Pi^{c},\omega})_{t \geq 0} \in D_{M_{\lambda}}[0,+\infty).
\end{equation*}

Let $t \geq 0$. Since $\omega$ is $\{0,1\}$-valued, by definition $\omega_{k,t}^{\Pi^{c},\omega}$ is $\{0,1\}$-valued. Moreover, the values taken by $\omega$ are changed over balls of the form $\bxr$ in order to compute $\omega_{k,t}^{\Pi^{c},\omega}$. Therefore, as $\omega$ is measurable, $\omega_{k,t}^{\Pi^{c},\omega}$ is measurable as well, and we obtain that for all $t \geq 0$, $M_{k,t}^{\Pi^{c},\omega} \in \mcal_{\lambda}$. 

We now show that the process is càdlàg. 
That is, we want to show that for all $t \geq 0$, 
\begin{enumerate}
\item $\lim\limits_{s \uparrow t} M_{k,s}^{\Pi^{c},\omega}$ exists,
\item $\lim\limits_{s \downarrow t} M_{k,s}^{\Pi^{c},\omega}$ exists and is equal to $M_{k,t}^{\Pi^{c},\omega}$. 
\end{enumerate}

Let $(f_{m})_{m \in \mathbb{N}}$ be a convergence determining class. It is then sufficient to show that for all $m \in \mathbb{N}$ and $n \in \mathbb{N} \backslash \{0\}$, we almost surely have that for all $t \in [0,n]$, 
\begin{enumerate}
\item $\lim\limits_{s \uparrow t} \langle \omega_{k,s}^{\Pi^{c},\omega},f_{m} \rangle$ exists,
\item $\lim\limits_{s \downarrow t} \langle \omega_{k,s}^{\Pi^{c},\omega},f_{m} \rangle$ exists and is equal to $\langle \omega_{k,t}^{\Pi^{c},\omega},f_{m} \rangle$.
\end{enumerate}
Therefore, let $m \in \mathbb{N}$ and $n \in \mathbb{N}\backslash \{0\}$. For all $t \in [0,n]$, we set 
\begin{align*}
T_{-}^{(m)}(t) &:=  \sup\left\{
t' < t : \exists \rcal > 0, \exists x \in \mathrm{Supp}^{\rcal}(f_{m}), \exists (p_{i})_{i \geq 1} \in U, (t',x,\rcal,(p_{i})_{i \geq 1}) \in \Pi^{c}
\right\} \\
\text{and } \quad T_{+}^{(m)}(t) &:= \inf \left\{
t' > t : \exists \rcal > 0, \exists x \in \mathrm{Supp}^{\rcal}(f_{m}), \exists (p_{n})_{n \geq 1} \in U, (t',x,\rcal,(p_{i})_{i \geq 1}) \in \Pi^{c}
\right\}. 
\end{align*}
Observe that if $T_{-}^{(m)}(t) < t$ and $T_{+}^{(m)}(t) > t$, then 
\begin{align*}
\lim\limits_{s \uparrow t} \langle \omega_{k,s}^{\Pi^{c},\omega},f_{m} \rangle &= \langle \omega_{k,\max(0,T_{-}^{(m)}(t))}^{\Pi^{c},\omega},f_{m} \rangle  \\
\text{and } \quad \lim\limits_{s \downarrow t} \langle \omega_{k,s}^{\Pi^{c},\omega},f_{m} \rangle &= \langle \omega_{k,t}^{\Pi^{c},\omega},f_{m} \rangle .
\end{align*}
Moreover, if the support of $f_{m}$ is affected by a finite number of reproduction events over the time interval $[0,n]$, then for all $t \in [0,n]$, $T_{-}^{(m)}(t) < t < T_{+}^{(m)}(t)$. 
As there exists $C^{(m)} > 0$ such that for all $\rcal > 0$, 
\begin{equation*}
\mathrm{Vol}(\mathrm{Supp}^{\rcal}(f_{m})) \leq C^{(m)}(\rcal^{d} \vee 1), 
\end{equation*}
the support of $f_{m}$ is affected by reproduction events at rate
\begin{align*}
\int_{0}^{\infty} \mathrm{Vol}(\mathrm{Supp}^{\rcal}(f_{m})) &\leq C^{(m)} \int_{0}^{\infty} (\rcal^{d} \vee 1) µ(d\rcal) \\
&< + \infty
\end{align*}
since $µ$ satisfies Condition~(\ref{eqn:condition_intensite}).
Therefore, the number of reproduction events affecting the support of $f_{m}$ is almost surely finite, allowing us to conclude. 

\end{proof}

\begin{lem}\label{lem:bon_espace_SLFV}
Under the notation of Definition \ref{defn:quenched_k_SLFV}, $(M_{k,t}^{\Pi^{c},\omega})_{t \geq 0}$ has the same distribution as the $k$-parent SLFV associated to $µ$ and with initial density $\omega$. 
\end{lem}

\begin{proof}
We set $M^{0} = M_{k,0}^{\Pi^{c},\omega}$, and we use the characterization of the $k$-parent SLFV by the duality relation in Proposition~\ref{prop:dualite}. 

Let $l \in \mathbb{N}^{*}$, let $\Psi$ be a density function on $(\rd)^{l}$, and let $t \geq 0$. Then,
\begin{align*}
&\esp_{M^{0}}\left[
\int_{(\rd)^{l}} \Psi(x_{1},...,x_{l})\left(
\prod_{j = 1}^{l} \omega_{k,t}^{\Pi^{c},\omega}(x_{j})
\right)dx_{1}...dx_{l}
\right]\\
= &\int_{(\rd)^{l}} \Psi(x_{1},...,x_{l}) \esp_{M^{0}}
\left[
\prod_{j = 1}^{l} \omega_{k,t}^{\Pi^{c},\omega}(x_{j})
\right]dx_{1}...dx_{l} \\
= &\int_{(\rd)^{l}} \Psi(x_{1},...,x_{l})
\esp_{M^{0}}\left[
\prod_{j = 1}^{l} \prod_{y \in A(\Xi_{k,t}^{\Pi^{c},t,\delta_{x_{j}}})} \omega(y)
\right]dx_{1}...dx_{l}
\\
= &\int_{(\rd)^{l}} \Psi(x_{1},...,x_{l})
\esp_{M^{0}}\left[
\prod_{y \in A(\Xi_{k,t}^{\Pi^{c},t,\sum_{j = 1}^{l}\delta_{x_{j}}})} \omega(y)
\right]dx_{1}...dx_{l} 
\\
= &\int_{(\rd)^{l}} \Psi(x_{1},...,x_{l}) 
\egras_{\Xi[x_{1},...,x_{l}]} \left[
\prod_{y \in A(\Xi_{t}^{k})} \omega(y)
\right]dx_{1}...dx_{l},
\end{align*}
with $(\Xi_{t}^{k})_{t \geq 0}$ the $k$-parent ancestral process associated to $µ$ with initial condition $\Xi[x_{1},...,x_{l}]$. We used the definition of the quenched $k$-parent SLFV to pass from line $2$ to line $3$, and the fact that $\omega$ is $\{0,1\}$-valued to pass from line $3$ to line $4$.

Writing $\Xi_{t}^{k} = \sum_{j = 1}^{N_{t}^{k}}\xi_{t}^{k,j}$, we obtain
\begin{align*}
&\esp_{M^{0}}\left[
\int_{(\rd)^{l}} \Psi(x_{1},...,x_{l})\left(
\prod_{j = 1}^{l} \omega_{k,t}^{\Pi^{c},\omega}(x_{j})
\right)dx_{1}...dx_{l}
\right]\\
= &\int_{(\rd)^{l}} \Psi(x_{1},...,x_{l}) 
\egras_{\Xi[x_{1},...,x_{l}]} \left[
\prod_{j  =1}^{N_{t}^{k}} \omega(\xi_{t}^{k,j})
\right]dx_{1}...dx_{l}.
\end{align*}
This concludes the proof.
\end{proof}

This lemma has two direct consequences. First, $(M_{k,t}^{\Pi^{c},\omega})_{t \geq 0}$ is Markovian. Moreover, since this process is $\mcal_{\lambda}$-valued, we have proved the second part of Theorem \ref{thm:forwards}, that is, that the unique solution to the martingale problem characterizing the $k$-parent SLFV is $\mcal_{\lambda}$-valued. 

The interest of the coupling lies in the fact that given a sequence of coupled $k$-parent SLFV constructed using the same extended Poisson point process $\Pi^{c}$, their corresponding densities, as constructed in Definition \ref{defn:quenched_k_SLFV}, satisfy the following property. 

\begin{lem}\label{lem:property_coupling}
Let $2 \leq k < k'$, and let $\omega : \rd \to \{0,1\}$ be a measurable function. Let $\Pi^{c}$ be a Poisson point process on $\mathbb{R} \times \rd \times (0,+\infty) \times U$ with intensity $dt \otimes dx \otimes µ(d\rcal) \otimes \tilde{u}((p_{n})_{n \geq 1})$. 

Then, for all $t \geq 0$ and $x \in \rd$, 
\begin{equation*}
\omega_{k',t}^{\Pi^{c},\omega}(x) \leq \omega_{k,t}^{\Pi^{c},\omega}(x).
\end{equation*}
In particular, for all $t \geq 0$ and $x \in \rd$, $(\omega_{k,t}^{\Pi^{c},\omega}(x))_{k \geq 2}$ converges to some $\omega_{t}^{\infty}(x) \in \{0,1\}$ as $k \to + \infty$.
\end{lem}

\begin{proof}
Let $t \geq 0$ and $x \in \rd$. By Lemma \ref{lem:couplage_k_k'}, 
\begin{equation*}
A(\Xi_{k,t}^{\Pi^{c},t,\delta_{x}}) \subseteq A(\Xi_{k',t}^{\Pi^{c},t,\delta_{x}}).
\end{equation*}
Therefore, as $\omega$ is $\{0,1\}$-valued,
\begin{align*}
\omega_{k',t}^{\Pi^{c},\omega}(x) &= \prod_{y \in A(\Xi_{k',t}^{\Pi^{c},t,\delta_{x}})}\omega(y) \\
&\leq \prod_{y \in A(\Xi_{k,t}^{\Pi^{c},t,\delta_{x}})} \omega(y) \\
&\leq \omega_{k,t}^{\Pi^{c},\omega}(x).
\end{align*}

The second part of the lemma is a consequence of the fact that $(\omega_{k,t}^{\Pi^{c},\omega}(x))_{k \geq 2}$ is a non-increasing $\{0,1\}$-valued sequence.
\end{proof}

\subsection{Definition of the $\infty$-parent SLFV}\label{subsec:2.11}

We can now construct the $\infty$-parent SLFV as in Theorem~\ref{thm:first_construction_infty_slfv}.

\begin{defn}\label{defn:infty_SLFV}
Let $M^{0} \in \mcal_{\lambda}$ with density $\omega : \rd \to \{0,1\}$. The $\infty$-parent spatial $\Lambda$-Fleming Viot process, or $\infty$-parent SLFV, with initial density $\omega$ and associated to the extended Poisson point process $\Pi^{c}$ is the $\mcal_{\lambda}$-valued process $(M_{t}^{\infty})_{t \geq 0}$ defined the following way. 

First, we set $M_{0}^{\infty} = M^{0}$. Then, for all $t \geq 0$ and $x \in \rd$, we set
\begin{align*}
\omega_{t}^{\infty}(x) &:= \lim\limits_{k \to + \infty} \omega_{k,t}^{\Pi^{c},\omega}(x) \\
\intertext{and we set}
M_{t}^{\infty}(dx, d\kappa)
&:= (\omega_{t}^{\infty}(x)\delta_{0}(d\kappa) + (1-\omega_{t}^{\infty}(x))\delta_{1}(d\kappa))dx.
\end{align*}
$\Pi^{c}$ will be called the associated extended Poisson point process, and $(\omega_{t}^{\infty})_{t \geq 0}$ will be called the density of the $\infty$-parent SLFV associated to $\Pi^{c}$ and of initial density $\omega$.
\end{defn}

In its more general form, the $\infty$-parent SLFV is defined for an initial condition $M^{0} \in \mcal_{\lambda}$ and a $\sigma$-finite measure $µ$. However, we construct it using a density $\omega$ of $M^{0}$, and an extended Poisson point process $\Pi^{c}$, and in the following, we will need both the initial density and the extended Poisson process used in order to prove some properties satisfied by the $\infty$-parent SLFV. Therefore, we considered two complementary definitions of the process, one based on the initial condition and the measure $µ$, and the other one based on the initial density and the extended Poisson point process, both definitions corresponding to the same process. In the following, we will use one or the other of the two definitions, depending on whether the initial density and extended Poisson point process used to construct the process are needed or not. 

As in the proof of Definition \ref{defn:quenched_k_SLFV}, we can show that $(M_{t}^{\infty})_{t \geq 0} \in D_{\mcal_{\lambda}}[0,+\infty)$. 

\begin{lem}\label{lem:infty_markovian}
Under the notation of Definition \ref{defn:infty_SLFV}, $(M_{t}^{\infty})_{t \geq 0}$ is Markovian.
\end{lem}

\begin{proof}
First, notice that the definition of $(M_{t}^{\infty})_{t \geq 0}$ implies that we only need to demonstrate that $(\omega_{t}^{\infty})_{t \geq 0}$ is Markovian. 

Let $0 \leq s \leq t$ and let $x \in \rd$. Our goal is to show that
\begin{equation*}
\omega_{t}^{\infty}(x) = \lim\limits_{\tilde{k} \to + \infty} \prod_{x' \in A\left(
\Xi_{\tilde{k},t-s}^{\Pi^{c},t,\delta_{x}}
\right)} \omega_{s}^{\infty}(x').
\end{equation*}
Indeed, if this result is true, since $A\left(
\Xi_{\tilde{k},t-s}^{\Pi^{c},t,\delta_{x}}
\right)$ depends on events occuring during the interval $[s,t]$, it is independent from $(\omega_{s'}^{\infty})_{0 \leq s' \leq s}$ and we can conclude. 

By definition of the $\infty$-parent SLFV,
\begin{align*}
\omega_{t}^{\infty}(x) &= \lim\limits_{k \to + \infty} \omega_{k,t}^{\Pi^{c},\omega}(x). \\
\intertext{Using Lemma \ref{lem:appendixC_2} from Section \ref{sec:5}, we obtain}
\omega_{t}^{\infty}(x) &= \lim\limits_{k \to + \infty}
\underset{x' \in A\left(\Xi_{k,t-s}^{\Pi^{c},t,\delta_{x}}
\right)}{\prod} \omega_{k,s}^{\Pi^{c},\omega}(x'). \\
\intertext{Let $\tilde{k} \geq 2$. By Lemma \ref{lem:property_coupling} and since for all $k \geq 2$, $\omega_{k,s}^{\Pi^{c},\omega}$ is $\{0,1\}$-valued,}
\omega_{t}^{\infty}(x) &\leq \lim\limits_{k \to + \infty}
\underset{x' \in A\left(\Xi_{\tilde{k},t-s}^{\Pi^{c},t,\delta_{x}}
\right)}{\prod} \omega_{k,s}^{\Pi^{c},\omega}(x') \\
&\leq \underset{x' \in A\left(\Xi_{\tilde{k},t-s}^{\Pi^{c},t,\delta_{x}}
\right)}{\prod} \lim\limits_{k \to + \infty} \omega_{k,s}^{\Pi^{c},\omega}(x')  \\
&\leq \underset{x' \in A\left(\Xi_{\tilde{k},t-s}^{\Pi^{c},t,\delta_{x}}
\right)}{\prod} \omega_{s}^{\infty}(x').
\end{align*}
Here we used Lemma \ref{lem:appendixC_3} to pass from the first to the second line, and the definition of the $\infty$-parent SLFV to pass from the second to the third line.  

Since this is true for all $\tilde{k} \geq 2$, 
\begin{align*}
\omega_{t}^{\infty}(x) &\leq \lim\limits_{\tilde{k} \to + \infty} \underset{x' \in A\left(\Xi_{\tilde{k},t-s}^{\Pi^{c},t,\delta_{x}}
\right)}{\prod} \omega_{s}^{\infty}(x'). \\
\intertext{Then, starting back from the equation}
\omega_{t}^{\infty}(x) &= \lim\limits_{k \to + \infty}
\underset{x' \in A\left(\Xi_{k,t-s}^{\Pi^{c},t,\delta_{x}}
\right)}{\prod} \omega_{k,s}^{\Pi^{c},\omega}(x'),
\intertext{as for all $x \in \rd$, $\left(\omega_{k,s}^{\Pi^{c},\omega}(x')\right)_{k \geq 2}$ is decreasing, we obtain that}
\omega_{t}^{\infty}(x) &\geq \lim\limits_{k \to + \infty}
\underset{x' \in A\left(\Xi_{k,t-s}^{\Pi^{c},t,\delta_{x}}
\right)}{\prod} \omega_{s}^{\infty}(x')
\end{align*}
and we can conclude.
\end{proof}

\subsection{Convergence of the $k$-parent SLFV towards the $\infty$-parent SLFV}\label{subsec:convergence_result_forward}
We now show Theorem~\ref{thm:convergence_result}, that is,
that the $\infty$-parent SLFV can be considered as the limit in distribution of the $k$-parent SLFV in $D_{\mcal_{\lambda}}[0,\infty)$ when $k \to + \infty$.  

In all that follows, let $(M_{t}^{\infty})_{t \geq 0}$ be the $\infty$-parent SLFV with initial density $\omega$ associated to $\Pi^{c}$. By definition, for all $t \geq 0$, $(M_{k,t}^{\Pi^{c},\omega})_{k \geq 2}$ converges almost surely to $M_{t}^{\infty}$, as stated below. 

\begin{lem}\label{lem:cvg_vague_processus_tps}
Conditionally on $\Pi^{c}$, for all $t \geq 0$ and $k \geq 2$, $M_{k,t}^{\Pi^{c},\omega}$ is deterministic, 
and $(M_{k,t}^{\Pi^{c},\omega})_{k \geq 2}$ converges vaguely to $M_{t}^{\infty}$ as $k \to + \infty$. 
Moreover, as a sequence of random variables, $(M_{k,t}^{\Pi^{c},\omega})_{k \geq 2}$ converges almost surely to $M_{t}^{\infty}$ as $k \to + \infty$. 
\end{lem}

\begin{proof}
Let $t \geq 0$, and let $\omega_{t}^{\infty}$ be the density of the $\infty$-parent SLFV with initial density $\omega$ associated to $\Pi^{c}$, considered at time $t$. Let $f \in C_{c}(\rd)$. Then $f$ is integrable and
\begin{align*}
\forall x \in \rd, f(x) \omega_{k,t}^{\Pi^{c},\omega}(x) & \xrightarrow[k \to + \infty]{} f(x) \omega_{t}^{\infty}(x) \\
\forall x \in \rd, \left|f(x) \omega_{k,t}^{\Pi^{c},\omega}(x)\right| &\leq |f(x)|.
\end{align*}
Therefore, by the dominated convergence theorem,
\begin{align*}
\lim\limits_{k \to + \infty} \int_{\rd}f(x)\omega_{k,t}^{\Pi^{c},\omega}(x)dx &= \int_{\rd}f(x)\omega_{t}^{\infty}(x)dx \\
\text{and \quad} \lim\limits_{k \to + \infty} \int_{\bxr}f(y)\omega_{k,t}^{\Pi^{c},\omega}(y)dx &= \int_{\bxr}f(y)\omega_{t}^{\infty}(y)dy.
\end{align*}

We now consider $\tilde{f} \in C_{c}(\rd \times \{0,1\})$. Then, there exists $f_{0},f_{1} \in C_{c}(\rd)$ such that
\begin{equation*}
\forall (x,\kappa) \in \rd \times \{0,1\}, 
\tilde{f}(x,\kappa) = f_{0}(x)\mathds{1}_{\{0\}}(\kappa) + f_{1}(x) \mathds{1}_{\{1\}}(\kappa).
\end{equation*}
Therefore, for all $k \geq 2$, 
\begin{align*}
\int_{\rd \times \{0,1\}} \tilde{f}(x,\kappa)M_{k,t}^{\Pi^{c},k}(dx,d\kappa) 
&= \int_{\rd}f_{0}(x) \omega_{k,t}^{\Pi^{c},k}(x)dx
+ \int_{\rd}f_{1}(x) \left(1-\omega_{k,t}^{\Pi^{c},k}(x)\right)dx \\
&\xrightarrow[k \to + \infty]{} \int_{\rd}f_{0}(x) \omega_{t}^{\infty}(x)dx 
+ \int_{\rd} f_{1}(x)\left(1-\omega_{t}^{\infty}(x)\right)dx \\
&= \int_{\rd \times \{0,1\}} \tilde{f}(x,\kappa)M_{t}^{\infty}(dx,d\kappa)
\end{align*}
and we conclude that $(M_{k,t}^{\Pi^{c},\omega})_{k \geq 2}$ converges (almost surely) vaguely to $M_{t}^{\infty}$ when $k \to + \infty$.
\end{proof}

Combined with Lemma~\ref{lem:bon_espace_SLFV}, a direct generalization of Lemma~\ref{lem:cvg_vague_processus_tps} gives that the finite dimensional distributions of the $k$-parent SLFV converge (weakly) towards the ones of the $\infty$-parent SLFV. 

We now focus on the convergence of the $k$-parent SLFV as a càdlàg process towards the $\infty$-parent SLFV. 
In order to do so, let $(M^{k})_{k \geq 2}$ be a sequence of $k$-parent SLFVs with initial density $\omega$ and associated to $µ$, which are independent from $(M_{k,t}^{\Pi^{c},\omega})_{t \geq 0}$. Moreover, let $(M_{t}^{\infty})_{t \geq 0}$ be the $\infty$-parent SLFV with initial density $\omega$ associated to $\Pi^{c}$. 

\begin{lem}\label{lem:sequence_tight}
The sequence $(M^{k})_{k \geq 2}$ is tight in $D_{\mcal_{\lambda}}[0,\infty)$.
\end{lem}

\begin{proof}
We adapt the proof of Theorem 1.2 (step (i), item (c)) from \cite{etheridge2020rescaling}. Since $\mcal_{\lambda}$ equipped with the topology of vague convergence is compact and since the set of functions of the form $\Psi_{F,f}$, $F \in C^{1}(\mathbb{R})$, $f \in^ C_{c}(\mathbb{R}^{d})$ is dense in $C(\mcal_{\lambda})$ (see e.g Lemma~1.1 from \cite{etheridge2020rescaling}), by Theorem~3.9.1 from \cite{ethier1986}, the relative compactness of $(M^{k})_{k \geq 2}$ is equivalent to the sequence $(\Psi_{F,f}(M^{k}))_{k \geq 2}$ being relatively compact for all $F \in C^{1}(\mathbb{R})$ and $f \in C_{c}(\mathbb{R}^{d})$. 

Therefore, let $F \in C^{1}(\mathbb{R})$ and $f \in C_{c}(\mathbb{R}^{d})$. By the Aldous-Rebolledo criterion 
\cite{aldous1978stopping,rebolledo1980existence}, it is sufficient to show that
\begin{enumerate}
\item For all $t \geq 0$, the sequence $(\Psi_{F,f}(M_{t}^{k}))_{k \geq 2}$ is tight. 
\item For all $T > 0$, given a sequence of stopping times $(\tau_{k})_{k \geq 2}$ bounded by $T$, for all $\epsilon > 0$, there exists $\delta > 0$ such that
\begin{equation*}
\limsup\limits_{k \to + \infty} \sup\limits_{\theta \in [0,\delta]} \proba\left(
\int_{\tau_{k}}^{\tau_{k} + \theta} \lcal_{µ}^{k} \Psi_{F,f}(M_{s}^{k}) ds > \epsilon
\right) \leq \epsilon. 
\end{equation*}
\item In the same notation as in 2., for all $\epsilon > 0$, there exists $\delta > 0$ such that
\begin{align*}
\limsup\limits_{k \to + \infty} \sup\limits_{\theta \in [0,\delta]} \proba\left(
\int_{\tau_{k}}^{\tau_{k}+\theta} \int_{\mathbb{R}^{d}} \int_{0}^{\infty} \int_{\bcal(x,\rcal)^{k}}\right. & \frac{1}{V_{\rcal}^{k}}\, \left[\vphantom{\prod_{j = 1}^{k}}\left(
F(\langle \Theta_{x,\rcal}^{-}(\omega_{M_{s}^{k}}),f\rangle) - F(\langle \omega_{M_{s}^{k}},f \rangle)
\right)^{2} \right. \\
&\quad \quad \quad
\times \left.
\left(
1-\prod_{j = 1}^{k}\omega_{M_{s}^{k}}(y_{j})
\right) \right.\\
&\quad \quad 
+ \left.
\left(
F(\langle \Theta_{x,\rcal}^{+}(\omega_{M_{s}^{k}}),f\rangle) - F(\langle \omega_{M_{s}^{k}},f \rangle)
\right)^{2} \right.\\
&\quad \quad \quad
\times \left.
\prod_{j = 1}^{k}\omega_{M_{s}^{k}}(y_{j})
\right]\left. 
\vphantom{\int_{\tau_{k}}^{\tau_{k}+\theta}}
dy_{1}...dy_{k}µ(d\rcal)dx ds > \epsilon\right) \leq \epsilon.
\end{align*}
\end{enumerate}

\paragraph{Point 1.} Let $t \geq 0$. Then, for all $k \geq 2$, 
\begin{equation*}
\left|
\langle \omega_{M_{t}^{k}},f \rangle 
\right| \leq ||f||_{\infty} \mathrm{Vol}(\mathrm{Supp}(f)). 
\end{equation*}
Since $F$ is continuous, this implies that
\begin{equation*}
\left|
\Psi_{F,f}(M_{t}^{k})
\right| \leq \max\limits_{z \in [-||f||_{\infty} \mathrm{Vol}(\mathrm{Supp}(f)),||f||_{\infty} \mathrm{Vol}(\mathrm{Supp}(f))]} |F(z)|,
\end{equation*}
from which we deduce that the sequence $(\Psi_{F,f}(M_{t}^{k}))_{k \geq 2}$ is tight. 

\paragraph{Point 2.} Let $T > 0$, let $(\tau_{k})_{k \geq 2}$ be a sequence of stopping times bounded by $T$, and let $\epsilon > 0$. In the proof of Lemma~\ref{lem:Lk_well_defined}, we show the existence of a constant $C_{F,f}$ depending only on the choice of $F$ and $f$ such that for all $t \geq 0$, $k \geq 2$ and $M \in \mcal_{\lambda}$, 
\begin{equation*}
\left|
\lcal_{µ}^{k} \Psi_{F,f}(M)
\right| \leq C_{F,f} \int_{0}^{\infty} \rcal^{d} µ(d\rcal).
\end{equation*}
Therefore, for all $k \geq 2$ and $\theta \geq 0$, 
\begin{align*}
\left|
\int_{\tau_{k}}^{\tau_{k}+\theta} \lcal_{µ}^{k} \Psi_{F,f}(M_{s}^{k}) ds
\right| &\leq \int_{\tau_{k}}^{\tau_{k}+\theta} \left|
\lcal_{µ}^{k} \Psi_{F,f} (M_{s}^{k})
\right| ds \\
&\leq \theta C_{F,f} \, \int_{0}^{\infty} \rcal^{d} µ(d\rcal),
\end{align*} 
and setting $\delta = (\epsilon / 2) \, C_{F,f}^{-1}\, (\int_{0}^{\infty} \rcal^{d} µ(d\rcal))^{-1}$ yields the desired result. 

\paragraph{Point 3.} By Lemma~\ref{lem:appendix_A_2}, there exists $\tilde{C}_{F,f}$ depending only on the choice of $F$ and $f$ such that for all $k \geq 2$, $s \geq 0$, $x \in \mathbb{R}^{d}$ and $\rcal > 0$, 
\begin{align*}
\left(
F(\langle \Theta_{x,\rcal}^{+}(\omega_{M_{s}^{k}}),f \rangle) - F(\langle \omega_{M_{s}^{k}},f \rangle)
\right)^{2} &\leq \tilde{C}_{F,f} \left(\rcal^{d} \wedge 1\right)^{2} \\
\text{and } \quad \left(
F(\langle \Theta_{x,\rcal}^{-}(\omega_{M_{s}^{k}}),f \rangle) - F(\langle \omega_{M_{s}^{k}},f \rangle)
\right)^{2} &\leq \tilde{C}_{F,f} \left(\rcal^{d} \wedge 1\right)^{2}. 
\end{align*}
Therefore, for all $k \geq 2$ and $\theta \geq 0$, we can bound from above
\begin{align*}
\int_{\tau_{k}}^{\tau_{k}+\theta} \int_{\mathbb{R}^{d}} \int_{0}^{\infty} \int_{\bcal(x,\rcal)^{k}}
& \frac{1}{V_{\rcal}^{k}}\, \left[\vphantom{\prod_{j = 1}^{k}}\left(
F(\langle \Theta_{x,\rcal}^{-}(\omega_{M_{s}^{k}}),f\rangle) - F(\langle \omega_{M_{s}^{k}},f \rangle)
\right)^{2} \right. \\
&\quad \quad \quad
\times \left.
\left(
1-\prod_{j = 1}^{k}\omega_{M_{s}^{k}}(y_{j})
\right) \right.\\
&\quad \quad 
+ \left.
\left(
F(\langle \Theta_{x,\rcal}^{+}(\omega_{M_{s}^{k}}),f\rangle) - F(\langle \omega_{M_{s}^{k}},f \rangle)
\right)^{2} \right.
\\
&\quad \quad \quad
\times \left.
\prod_{j = 1}^{k}\omega_{M_{s}^{k}}(y_{j})
\right]
dy_{1}...dy_{k}µ(d\rcal)dx ds
\end{align*}
by 
\begin{align*}
& \int_{\tau_{k}}^{\tau_{k}+\theta} \int_{0}^{\infty} \int_{\mathrm{Supp}^{\rcal}(f)} \int_{\bcal(x,\rcal)^{k}}
\frac{1}{V_{\rcal}^{k}} \, \tilde{C}_{F,f} (\rcal^{d} \wedge 1)^{2} dy_{1}...dy_{k} dx µ(d\rcal)ds \\
&\leq \int_{\tau_{k}}^{\tau_{k}+\theta} \int_{0}^{\infty} \int_{\mathrm{Supp}^{\rcal}(f)} \tilde{C}_{F,f} (\rcal^{d} \wedge 1) dx µ(d\rcal)ds \\
&\leq \int_{\tau_{k}}^{\tau_{k}+\theta} \int_{0}^{\infty} \tilde{C}_{F,f} \,(\rcal^{d} \wedge 1) \, C_{2} \, (\rcal^{d} \vee 1) ds µ(d\rcal) \\
&\leq \theta \, C_{2} \, \tilde{C}_{F,f} \, \int_{0}^{\infty} \rcal^{d}µ(d\rcal). 
\end{align*}
Here we used the fact that $0 \leq \rcal^{d} \wedge 1 \leq 1$ to pass from the first to the second line, and Eq.~(\ref{eqn:eqn_1}) to pass from the second to the third line. Setting $\delta = (3/2) \, C_{2}^{-1} \, \tilde{C}_{F,f}^{-1} \, (\int_{0}^{\infty} \rcal^{d}µ(d\rcal))^{-1}$ allows us to  conclude. 
\end{proof}

We can now show Theorem~\ref{thm:convergence_result}. 

\begin{proof}(Theorem~\ref{thm:convergence_result})
As the finite-dimensional distributions of $(M^{k})_{k \geq 2}$ converge to the ones of $M^{\infty}$ and as $(M^{k})_{k \geq 2}$ is tight (by Lemma~\ref{lem:sequence_tight}), we can use Prokhorov's theorem 
\cite{prokhorov1956convergence}
and conclude. 
\end{proof}

\section{The $\infty$-parent ancestral process : definition and characterization}\label{sec:3}
The goal of this section is to rigorously construct and study the dual process associated to the $\infty$-parent SLFV, called the $\infty$-parent ancestral process. In Section~\ref{subsec:3.1}, we exhibit sufficient conditions under which the $\infty$-parent ancestral process is well-defined, and state that it is the unique solution to a martingale problem. 
In Section~\ref{subsec:convergence_dual_process}, we show to what extent the $\infty$-parent ancestral process can be considered as the limit of the $k$-parent ancestral process when $k \to + \infty$, the main difficulty being that the two processes are defined on different state spaces. 
Then, in Section~\ref{subsec:3.2}, we use the $\infty$-parent ancestral process and its associated martingale problem to provide another characterization of the $\infty$-parent SLFV. 

\subsection{Definition and first properties}\label{subsec:3.1}
In order to show that the $\infty$-parent ancestral process $(\Xi_{t}^{\infty})_{t \geq 0}$ introduced in Definition \ref{defn:infty_parents_ancestors} is well-defined, we start by observing that the only reproduction events affecting $\Xi_{t}^{\infty}$ are the ones intersecting its boundary $\overline{\Xi_{t}^{\infty}}\backslash \mathring{\Xi_{t}^{\infty}}$. Therefore, it is sufficient to consider only the reproduction events affecting its border, or the ones affecting a well-chosen space containing it. 

\begin{figure}[ht]
\centering
\includegraphics[width = 0.6\linewidth]{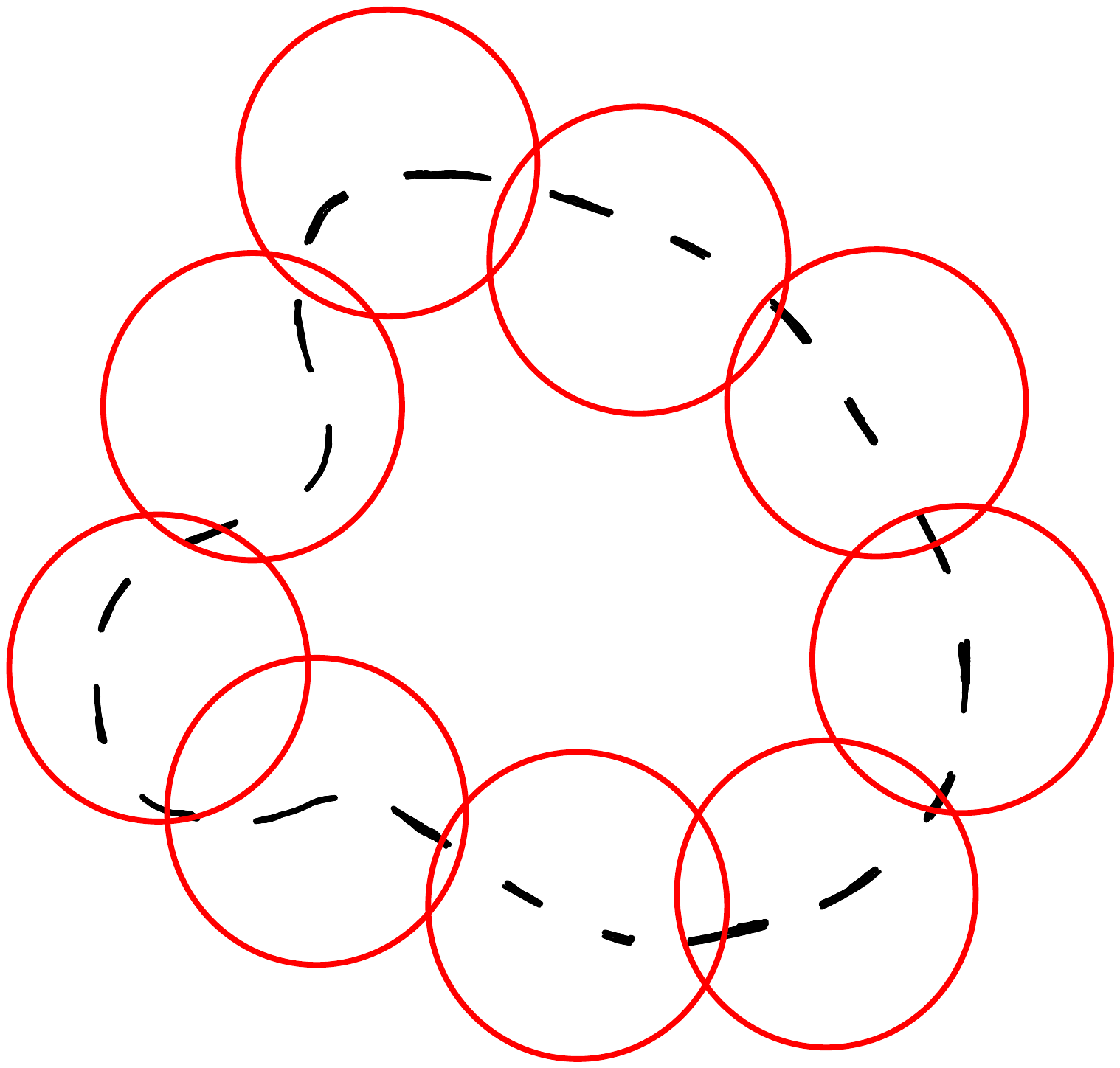}
\caption{Initial state of the $\infty$-parent ancestral process (dashed line), and a covering of its border by balls of radius $\widetilde{\rcal}$.}\label{fig:border_covering_t_0}
\end{figure}

In order to control the rate at which the $\infty$-parent ancestral process jumps, we start by taking $\widetilde{\rcal} > 0$ satisfying some condition which will be introduced later, and we cover the border $\overline{\Xi_{0}^{\infty}}\backslash \mathring{\Xi_{0}^{\infty}}$ of $\Xi_{0}^{\infty}$ with balls of radius $\widetilde{\rcal}$ (see Figure \ref{fig:border_covering_t_0}). Then, informally, whenever a reproduction event overlaps what we will call the $\widetilde{\rcal}$-covering :
\begin{itemize}
\item if this reproduction event has a radius of at most $\widetilde{\rcal}$, it is included in the ball of same center but of radius $\widetilde{\rcal}$. We add this ball of radius $\widetilde{\rcal}$ to the covering. 
\item Otherwise, we cover the border of the area of the reproduction event by balls of radius $\widetilde{\rcal}$, and we add these balls to the covering.
\end{itemize}
See Figure \ref{fig:border_covering_t_sup_0} for an illustration of this dynamics. 

\begin{figure}[ht]
\centering
\begin{subfigure}[b]{0.45\textwidth}
\includegraphics[width = \linewidth]{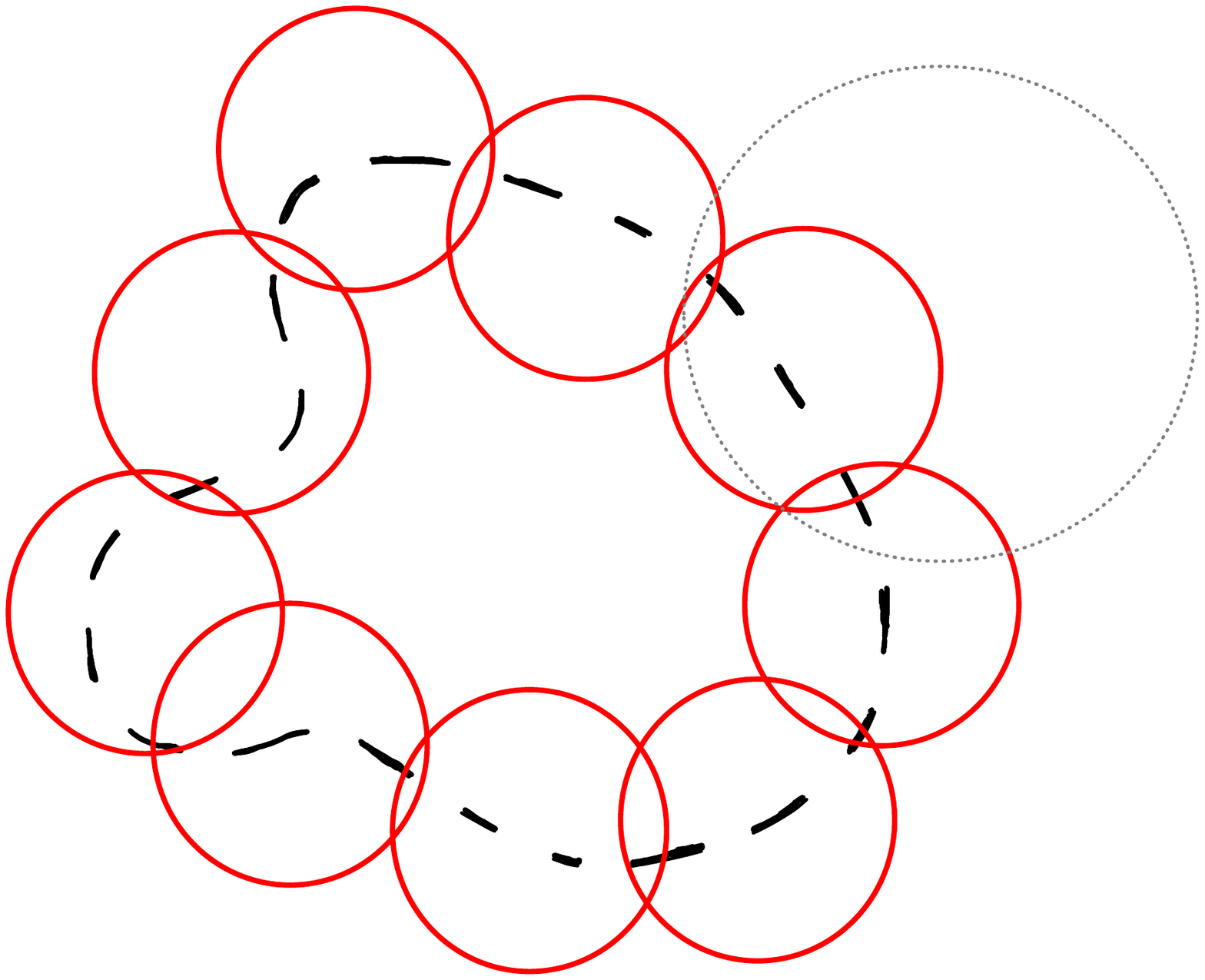}
\subcaption{Reproduction event (grey disk) affecting the $\infty$-parent ancestral process at time $t > 0$.}
\end{subfigure}
\hfill
\begin{subfigure}[b]{0.45\textwidth}
\includegraphics[width = \linewidth]{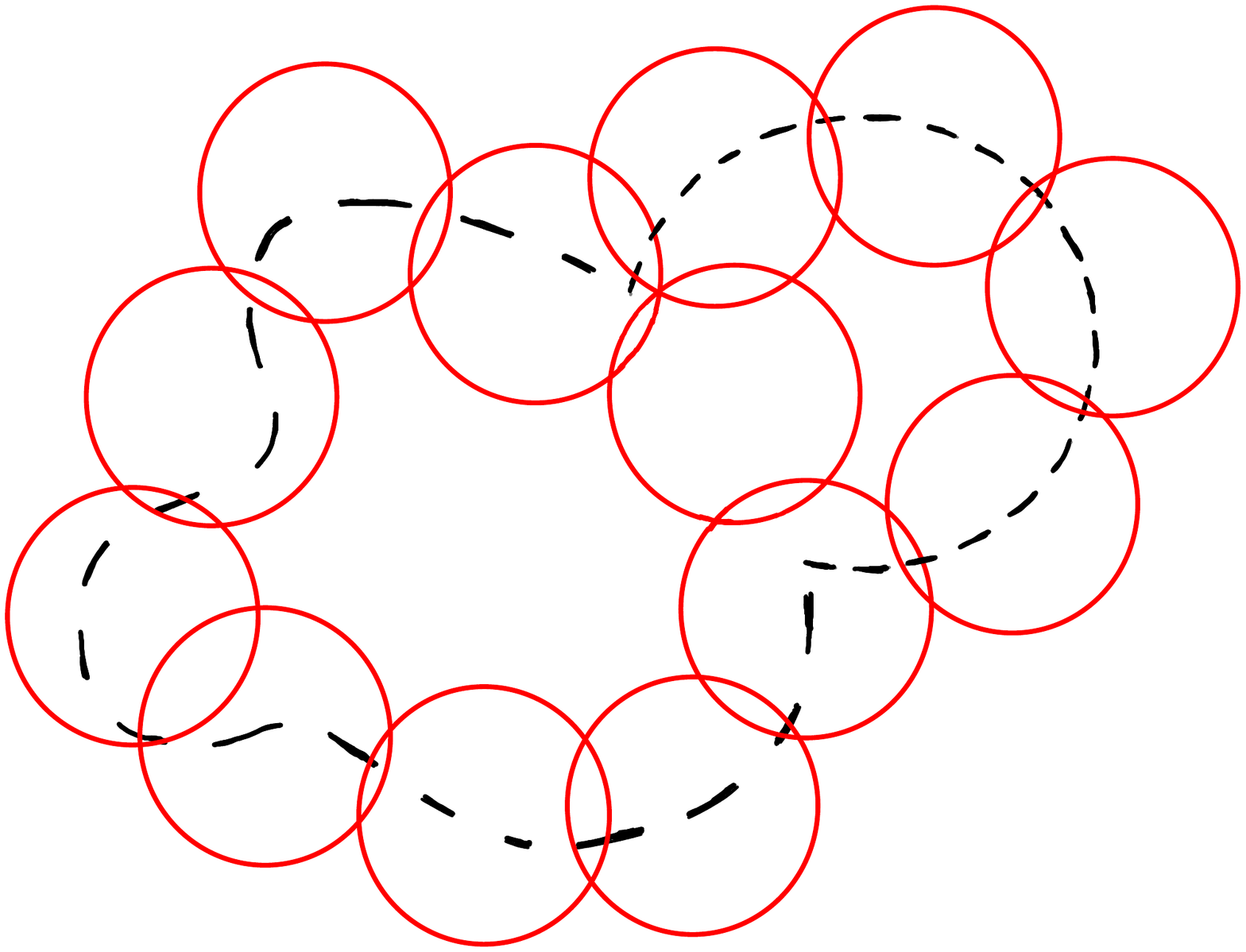}
\subcaption{The $\infty$-parent ancestral process is updated, and a covering of the border of the reproduction event by balls of radius $\widetilde{\rcal}$ is added to the $\widetilde{\rcal}$-covering process.}
\end{subfigure}

\begin{subfigure}[b]{0.45\textwidth}
\includegraphics[width = \linewidth]{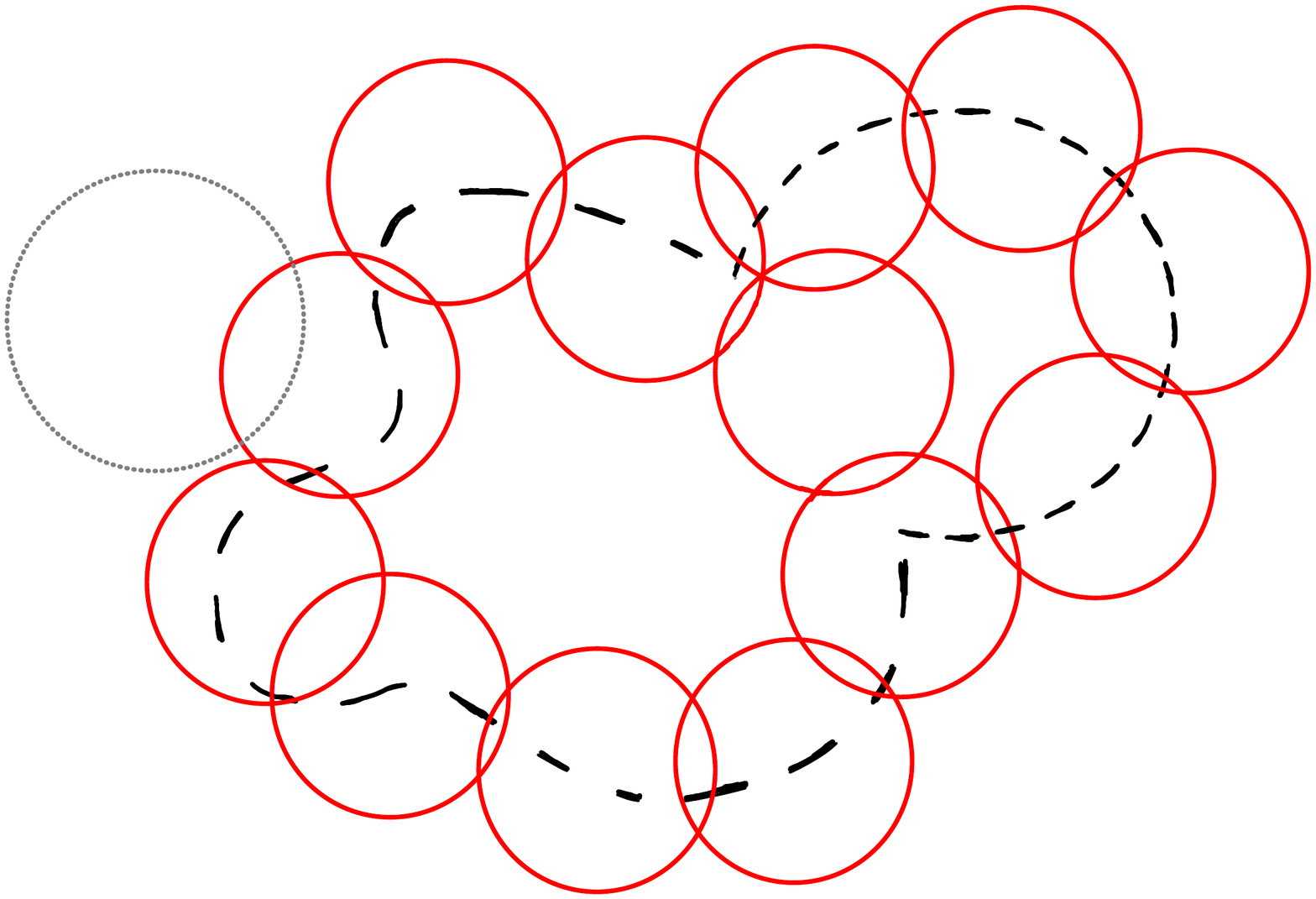}
\subcaption{Since the $\widetilde{\rcal}$-covering process is bigger than the border of the $\infty$-parent ancestral process, it can be affected by reproduction events (grey disk) which do not intersect the $\infty$-parent ancestral process.}
\end{subfigure}
\hfill
\begin{subfigure}[b]{0.45\textwidth}
\includegraphics[width = \linewidth]{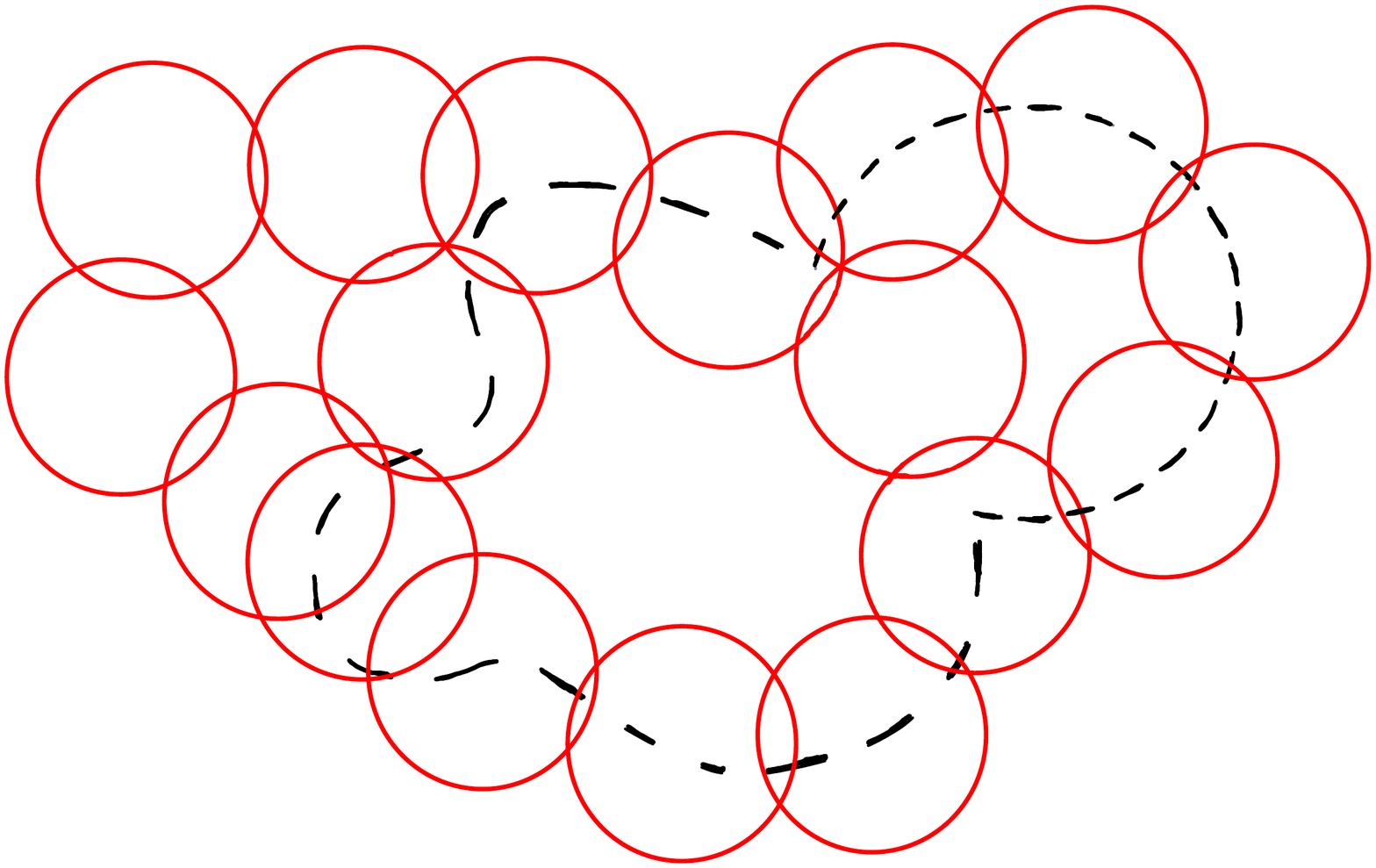}
\subcaption{Updated $\widetilde{\rcal}$-covering process after a reproduction event affecting it while not intersecting the $\infty$-parent ancestral process.}
\end{subfigure}
\caption{Illustration of the dynamics of the $\infty$-parent ancestral process (dashed line) and its associated $\widetilde{\rcal}$-covering process.}\label{fig:border_covering_t_sup_0}
\end{figure}
Note that since the covering contains the border $\overline{\Xi_{t}^{\infty}}\backslash \mathring{\Xi_{t}^{\infty}}$ of $\Xi_{t}^{\infty}$ but is not equal to it, there are more reproduction events affecting the $\widetilde{\rcal}$-covering than reproduction events affecting $\Xi_{t}^{\infty}$. 

Constructed this way, the $\widetilde{\rcal}$-covering contains only balls of radius $\widetilde{\rcal}$, each one being overlapped by a reproduction event at rate 
\begin{equation*}
\int_{0}^{\infty} V_{1} (\widetilde{\rcal} + r)^{d}µ(dr).
\end{equation*}
Moreover, since the covering is constructed using the same Poisson point process as for $(\Xi_{t}^{\infty})_{t \geq 0}$, at any time $t$ the current state of the covering contains the border $\overline{\Xi_{t}^{\infty}}\backslash \mathring{\Xi_{t}^{\infty}}$ of $\Xi_{t}^{\infty}$. Since the rate at which $(\Xi_{t}^{\infty})_{t \geq 0}$ jumps is bounded by the rate at which the covering we just constructed is updated, we can show that $(\Xi_{t}^{\infty})_{t \geq 0}$ is well-defined by controlling the rate at which new balls are added to the $\widetilde{\rcal}$-covering.

Let us now define the border covering process we just introduced rigorously. 
\begin{defn}[Border covering process]\label{defn:border_covering} 
In the notation of Definition \ref{defn:infty_parents_ancestors}, let $\widetilde{\rcal} > 0$ be such that $µ$ satisfies Condition (\ref{eqn:condition_infty_SLFV}). Let $x_{1},...,x_{N} \in \rd$, $N \geq 1$ be such that initially the border of $\Xi^{0}$ is entirely covered by the $N$ balls of radius $\widetilde{\rcal}$ $(B(x_{i},\widetilde{\rcal}))_{1 \leq i \leq N}$. Then, the $\widetilde{\rcal}$-covering process $(C_{t})_{t \geq 0}$ associated to $(\Xi_{t}^{\infty})_{t \geq 0}$ is constructed in the following way. 

Let $\widetilde{\Pi}$ be a Poisson point process on $\mathbb{R} \times \mathbb{R}^{d} \times (0,+\infty)$ with intensity $dt \otimes dx \otimes µ(dr)$. 
First, we set $C_{0} = \{x_{1},...,x_{N} : 1 \leq i \leq N\}$. Then, for all $(t, x, \rcal) \in \widetilde{\Pi}$, if $C_{t-} \cap \bxr \neq \emptyset$, let $n \in \mathbb{N}^{*}$ such that $(n-1)\widetilde{\rcal} \leq \rcal \leq n \widetilde{\rcal}$. We construct a covering of the border of $\bxr$ by at most $a_{d}\times n^{d-1}$ balls of radius $\widetilde{\rcal}$, and $C_{t}$ is obtained by adding the center of these balls to $C_{t-}$. 
\end{defn}

The interest of the border covering process lies in the fact that, as we argued earlier, for all $t \geq 0$, 
\begin{equation*}
\overline{\Xi_{t}^{\infty}}\backslash \mathring{\Xi_{t}^{\infty}} \subseteq C_{t}.
\end{equation*}
Therefore, the jump rate of $\Xi_{t}^{\infty}$ is bounded above by 
\begin{equation*}
Card(C_{t}) \times \int_{0}^{\infty} V_{1} (\widetilde{\rcal} + r)^{d}µ(dr).
\end{equation*}

\begin{lem}\label{lem:border_covering_bounded}
In the notation of Definitions \ref{defn:infty_parents_ancestors} and \ref{defn:border_covering}, $(Card(C_{t}))_{t \geq 0}$ is bounded from above by $(Y_{t})_{t \geq 0}$ the number of particles in a branching process in which each particle branches independently of the others at rate
\begin{equation*}
\int_{0}^{\infty} V_{1}(\widetilde{\rcal} + r)^{d} µ(dr),
\end{equation*}
and in which at each branching event, the number of descendants is equal to $a_{d} n^{d-1} + 1$, $n \geq 1$ with probability
\begin{equation*}
\frac{\int_{(n-1)\widetilde{\rcal}}^{n\widetilde{\rcal}} (\widetilde{\rcal} + r)^{d}µ(dr)}{\int_{0}^{\infty} (\widetilde{\rcal}+r)^{d} µ(dr)}.
\end{equation*}
Moreover, for all $t \geq 0$, $Y_{t} < + \infty$ a.s, and $\esp[Y_{t}] < + \infty$. 
\end{lem}

\begin{proof}
How to construct the branching process $(Y_{t})_{t \geq 0}$ from $(C_{t})_{t \geq 0}$ is clear. The jump rates and transition probabilities come from the fact that for any point $x \in C_{t}$ and for all $n \geq 1$, the ball $\mathcal{B}(x,\widetilde{\rcal})$ is affected by a reproduction event of radius $(n-1)\widetilde{\rcal} \leq \rcal \leq n\widetilde{\rcal}$ at rate 
\begin{equation*}
\int_{(n-1)\widetilde{\rcal}}^{n\widetilde{\rcal}} V_{1}(\widetilde{\rcal} + \rcal)^{d} µ(d\rcal),
\end{equation*}
and such a reproduction event generates 
$a_{d} \, n^{d-1}$ new balls in the border covering process. 

Then, if $\Phi$ is the probability generating function of the number of descendants, 
\begin{equation*}
\Phi'(1) = \sum_{n = 1}^{+ \infty} \left(\int_{(n-1)\widetilde{\rcal}}^{n \widetilde{\rcal}} V_{1} (\widetilde{\rcal} + r)^{d}µ(dr)\right) \times (a_{d} \, n^{d-1} + 1) < + \infty
\end{equation*}
since $µ$ satisfies Condition (\ref{eqn:condition_infty_SLFV}). Therefore, by Theorem III.2.1 in \cite{athreya1972}, $Y_{t}$ is finite for all $t \geq 0$ a.s, and $\esp[Y_{t}] < + \infty$ for all $t \geq 0$. 
\end{proof}

We can then conclude that $(\Xi_{t}^{\infty})_{t \geq 0}$ is well-defined using the fact that the jump rate of $\Xi_{t}^{\infty}$ is bounded from above by
\begin{equation*}
Y_{t}  \, \int_{0}^{\infty} V_{1}  \, (\widetilde{\rcal} + \rcal)^{d}µ(d\rcal) < + \infty \text{ a.s }.
\end{equation*}

We now introduce how to characterize the $\infty$-parent ancestral process as the unique solution to a martingale problem. 

In all that follows, let $F \in C_{b}^{1}(\mathbb{R})$ and $f \in \mathbb{B}(\rd)$. We extend the definition of the function $\Phi_{F,f}$ to the space of measures $m(E) \in \mcal^{cf}$, setting
\begin{align*}
\Phi_{F,f}(m(E)) := &F\left(\int_{\rd} f(x)m(E)dx \right) \\
= &F\left(\int_{E} f(x)dx \right).
\end{align*}
Moreover, for all $E \in \mathcal{E}^{cf}$ and $\rcal > 0$, we set
\begin{equation*}
S^{\rcal}(E) := \{x \in \rd : \exists y \in E, ||x-y|| \leq \rcal\}.
\end{equation*}
Note that this definition is reminiscent of the definition of $S^{\rcal}(\Xi)$ with $\Xi \in \mcal_{p}(\rd)$. 

Let $µ$ be a $\sigma$-finite measure on $\mathbb{R}_{+}^{*}$ satisfying Condition (\ref{eqn:condition_infty_SLFV}). We define the operator $\gcal_{µ}^{\infty}$ on functions of the form $\Phi_{F,f}$ the following way. For all $m(E) \in \mcal^{cf}$, we set
\begin{equation*}
\gcal_{µ}^{\infty} \Phi_{F,f} (m(E)) := \int_{0}^{\infty}\int_{S^{\rcal}(E)}
F\left(\langle m(E \cup \bxr), f \rangle \right) - F \left( \langle m(E),f \rangle \right) dx µ(d\rcal). 
\end{equation*}
We show in Section \ref{sec:5} that this operator is well-defined, and give some properties that it satisfies. 
The $\infty$-parent ancestral process is then solution to the following martingale problem. 
\begin{prop}\label{prop:backwards_kinf}
Let $µ$ be a $\sigma$-finite measure on $(0,+\infty)$ satisfying Condition (\ref{eqn:condition_infty_SLFV}). Let $\Xi^{0} \in \mathcal{M}^{cf}$, and let $(\Xi_{t}^{\infty})_{t \geq 0}$ be the $\infty$-parent ancestral process associated to $µ$ with initial condition $\Xi^{0}$. 

Then, for all $F \in C_{b}^{1}(\mathbb{R})$ and for all measurable function $f : \rd \to \{0,1\}$, the process
\begin{equation*}
\left(
\Phi_{F,f}(\Xi_{t}^{\infty}) - \Phi_{F,f}(\Xi_{0}^{\infty}) - \int_{0}^{t}\gcal_{µ}^{\infty}\Phi_{F,f}(\Xi_{s}^{\infty})ds
\right)_{t \geq 0}
\end{equation*}
is a martingale.
\end{prop}

The proof is a direct adaptation of the proof of Proposition 4.1.7. from 
\cite{ethier1986}.

\subsection{Convergence of the $k$-parent ancestral process towards the $\infty$-parent ancestral process}\label{subsec:convergence_dual_process}
As in the case of the $\infty$-parent SLFV, the $\infty$-parent ancestral process can be considered as the limit of the $k$-parent ancestral process when $k \to + \infty$, though since the dual processes are defined on different state spaces (that is, $\mcal_{p}(\mathbb{R}^{d})$ in the case of the $k$-parent ancestral process and $\mcal^{cf}$ in the case of the $\infty$-parent ancestral process), the result we obtain is weaker than the corresponding one on the $\infty$-parent SLFV.

In all that follows, let $µ$ be a $\sigma$-finite measure on $(0,+\infty)$ satisfying Condition~(\ref{eqn:condition_infty_SLFV}). We recall that for all $\Xi = \sum_{i = 1}^{l} \delta_{x_{i}} \in \mcal_{p}(\mathbb{R}^{d})$, the set of atoms of $\Xi$ is denoted $A(\Xi)$. Our goal is to show the following result. 

\begin{prop}\label{prop:cvg_dual}
Let $E^{0} \in \mathcal{E}^{cf}$, and let $(p_{n}^{0})_{n \geq 1}$ be a sequence of i.i.d. random variables sampled uniformly at random in $E^{0}$. Let $(\Xi_{t}^{\infty})_{t \geq 0} = (m(E_{t}^{\infty}))_{t \geq 0}$ be the $\infty$-parent ancestral process with initial condition $m(E^{0})$ and associated to $µ$, and for all $k \geq 2$, let $(\Xi_{t}^{k})_{t \geq 0}$ be the $k$-parent ancestral process with initial condition
\begin{equation*}
\Xi_{0}^{k} = \sum_{i = 1}^{k} \delta_{p_{n}^{0}}
\end{equation*}
and associated to $µ$. Then, it is possible to couple $(\Xi_{t}^{\infty})_{t \geq 0}$ and $(\Xi_{t}^{k})_{t \geq 0}$, $k \geq 2$ in such a way that for all $t \geq 0$, 
\begin{equation*}
A(\Xi_{t}^{k}) \xrightarrow[k \to + \infty]{} E_{t}^{\infty} \quad \text{ a.s.}
\end{equation*}
in the sense of Painlevé-Kuratowski and as compact subsets of $\mathbb{R}^{d}$. 
\end{prop}
The proof of Proposition~\ref{prop:cvg_dual} can be found at the end of the section. 

Looking at the convergence of the support of the measures allows us to circumvent the facts that the $k$-parent and $\infty$-parent ancestral processes are defined on different state spaces, and that the mass of the $k$-parent ancestral process becomes infinite when $k \to + \infty$. We also need to consider specific initial conditions for the sequence of $k$-parent ancestral processes in order to obtain a limiting set with nonzero Lebesgue measure. 

In order to show Proposition~\ref{prop:cvg_dual}, we first construct the corresponding coupling, which relies on the use of the same extended Poisson point process to construct all processes. Let $\Pi^{c}$ be a Poisson point process on $\mathbb{R} \times \mathbb{R}^{d} \times (0,+\infty) \times U$ with intensity
$dt \otimes dx \otimes µ(d\rcal) \otimes \tilde{u}\left(
d(p_{n})_{n \geq 1}
\right)$.
Let $(\Xi_{t}^{\infty,\Pi^{c}})_{t \geq 0}$ be the $\infty$-parent ancestral process with initial condition $E^{0}$ constructed using $\Pi^{c}$. Moreover, for all $k \geq 2$, let $(\Xi_{t}^{k,\Pi^{c}})_{t \geq 0}$ be the $k$-parent ancestral process with initial condition $\sum_{k = 1}^{k} \delta_{p_{n}^{0}}$ constructed using $\Pi^{c}$: whenever the process jumps, if the corresponding reproduction event is $(t,x,\rcal,(p_{n})_{n \geq 1})$, then the locations of the $k$ potential parents are given by $x+\rcal p_{1},...,x+\rcal p_{k}$. The following result is a direct consequence of the coupling and of the choice of initial conditions. 

\begin{lem}\label{lem:initial_condition_dual}
\begin{enumerate}
\item We have 
\begin{equation*}
A\left(
\Xi_{0}^{k,\Pi^{c}}
\right) \xrightarrow[k \to + \infty]{} E^{0} \quad \text{ a.s.}
\end{equation*}
in the sense of Painlevé-Kuratowski. 
\item For all $t \geq 0$, for all $k' \geq k \geq 2$, 
\begin{equation*}
A\left(
\Xi_{t}^{k,\Pi^{c}}
\right) \subseteq A \left(
\Xi_{t}^{k',\Pi^{c}}
\right).
\end{equation*}
\end{enumerate}
\end{lem}

Moreover, we also have the following result.
\begin{lem}\label{lem:inclusion_sequence_duals}
Let $T \geq 0$. Then, the two following assertions are true a.s.:
\begin{align*}
& \forall 0 \leq t \leq T, \forall k \geq 2, A\left(
\Xi_{t}^{k,\Pi^{c}}
\right) \subseteq E_{t}^{\infty,\Pi^{c}}, \\
\text{and } \quad &\forall 0 \leq t \leq T, \forall k \geq 2, \text{ if } E_{t-}^{\infty,\Pi^{c}} = E_{t}^{\infty,\Pi^{c}}, \text{ then } \Xi_{t-}^{k,\Pi^{c}} = \Xi_{t}^{k,\Pi^{c}}. 
\end{align*}
\end{lem}

\begin{proof}
For $t = 0$, by definition, $A(\Xi_{0}^{k,\Pi^{c}}) \subseteq E_{0}^{\infty,\Pi^{c}}$ for all $k \geq 2$. Then, since all ancestral processes are constructed using the same underlying Poisson point process, each jump of a $k$-parent ancestral process corresponds to a jump of the $\infty$-parent ancestral process, unless the corresponding reproduction event affects an area $\bcal(x,\rcal)$ such that
\begin{equation*}
\bcal(x,\rcal) \cap E_{t-}^{\infty,\Pi^{c}} \neq \emptyset \quad \text{ but } \quad 
\mathrm{Vol}(\bcal(x,\rcal) \cap E_{t-}^{\infty,\Pi^{c}}) = 0,
\end{equation*}
which almost surely never occurs. 
\end{proof}

Therefore, due to the coupling, each jump of the $k$-parent ancestral process corresponds to a jump of the $\infty$-parent ancestral process, as well as of each $k'$-parent ancestral process, $k' \geq k$ (by Lemma~\ref{lem:initial_condition_dual}). Moreover, since $µ$ satisfies Condition~(\ref{eqn:condition_infty_SLFV}), for all $T \geq 0$, the number of jumps of the $\infty$-parent ancestral process over the time interval $[0,T]$ is finite a.s. As a result, in order to show Proposition~\ref{prop:cvg_dual}, we can show that the convergence property is preserved whenever $(\Xi_{t}^{\infty,\Pi^{c}})_{t \geq 0}$ jumps. 

\begin{lem}\label{lem:transfert_cvg_property}
For all $t \geq 0$, if
\begin{align}
&A\left(
\Xi_{t-}^{k,\Pi^{c}}
\right) \xrightarrow[k \to + \infty]{} E_{t-}^{\infty,\Pi^{c}}\quad \text{in the sense of Painlevé-Kuratowski}, \label{eqn:CVG_lem_1} \\
\text{then we also have} \quad & A\left(
\Xi_{t}^{k,\Pi^{c}}
\right) \xrightarrow[k \to + \infty]{} E_{t}^{\infty,\Pi^{c}} \quad \text{ a.s.}
\end{align}
\end{lem}

\begin{proof}
Let $t \geq 0$. Assume that (\ref{eqn:CVG_lem_1}) is true. The result is clear if $E_{t}^{\infty,\Pi^{c}} = E_{t-}^{\infty,\Pi^{c}}$. Otherwise, let $(t,x,\rcal,(p_{n}^{t})_{n \geq 1})$ be the corresponding reproduction event. This implies that 
\begin{equation*}
\mathrm{Vol}\left(
E_{t-}^{\infty,\Pi^{c}} \cap \bcal(x,\rcal)
\right) \neq 0. 
\end{equation*}
By (\ref{eqn:CVG_lem_1}), this means that there exists $k^{t} \geq 2$ such that $A(\Xi_{t-}^{k^{t},\Pi^{c}}) \cap \bcal(x,\rcal) \neq 0$. Such a result can be extended to all $k \geq k^{t}$ by Lemma~\ref{lem:initial_condition_dual}. Moreover, for all $k \geq k^{t}$, 
\begin{align*}
A\left(
\Xi_{t}^{k,\Pi^{c}}
\right) 
&= \left(
A\left(
\Xi_{t-}^{k,\Pi^{c}}
\right) \backslash I_{x,\rcal} \left(
\Xi_{t-}^{k,\Pi^{c}}
\right)\right) \bigcup \left\{
x + \rcal p_{n}^{t}: n \in \llbracket 1,k \rrbracket
\right\} \\
&\xrightarrow[k \to + \infty]{} E_{t-}^{\infty,\Pi^{c}} \cup \bcal(x,\rcal) 
= E_{t}^{\infty,\Pi^{c}}.
\end{align*}
\end{proof}

We can now show Proposition~\ref{prop:cvg_dual}.

\begin{proof}
(Proposition~\ref{prop:cvg_dual}) Let $t \geq 0$. We set $t_{0} = 0$, and let $0 < t_{1} < t_{2}  < ... < t_{N} \leq t$ be the times at which $(\Xi_{s}^{\infty,\Pi^{c}})_{s \geq 0}$ jumps (such a sequence of jump times exist almost surely). First, we observe that by Lemma~\ref{lem:inclusion_sequence_duals}, it is sufficient to show that
\begin{equation*}
A\left(
\Xi_{t_{N}}^{k,\Pi^{c}}
\right) \xrightarrow[k \to + \infty]{} E_{t_{N}}^{\infty,\Pi^{c}} \quad \text{ a.s.}
\end{equation*}
in the sense of Painlevé-Kuratowski. 
In order to do so, we will show by induction that
\begin{equation*}
\forall n \in \llbracket 0,N \rrbracket, 
A\left(
\Xi_{t_{n}}^{k,\Pi^{c}}
\right) \xrightarrow[k \to + \infty]{} E_{t_{n}}^{\infty,\Pi^{c}} \quad \text{ a.s.}
\end{equation*}
The case $n = 0$ corresponds to Lemma~\ref{lem:initial_condition_dual}. Then, let $n \in \llbracket 0,N-1 \rrbracket$, and assume that
\begin{equation*}
A\left(
\Xi_{t_{n}}^{k,\Pi^{c}}
\right) \xrightarrow[k \to + \infty]{} E_{t_{n}}^{\infty,\Pi^{c}} \quad \text{ a.s.}
\end{equation*}
Again by Lemma~\ref{lem:inclusion_sequence_duals}, we obtain that
\begin{equation*}
A\left(
\Xi_{t_{n+1}-}^{k,\Pi^{c}}
\right) \xrightarrow[k \to + \infty]{} E_{t_{n+1}-}^{\infty,\Pi^{c}} \quad \text{ a.s.}
\end{equation*}
and we conclude using Lemma~\ref{lem:transfert_cvg_property}. 
\end{proof}

\subsection{Another characterization of the $\infty$-parent SLFV}\label{subsec:3.2}
We can now show Proposition~\ref{prop:stochastic_growth_model}, which provides another characterization of the $\infty$-parent SLFV. 

\begin{proof}(Proposition~\ref{prop:stochastic_growth_model})
Let $F \in C^{1}(\mathbb{R})$ and $f \in C_{c}(\mathbb{R}^{d})$. Our goal is to show that 
\begin{equation*}
\left(
\Psi_{F,f}(\widehat{M}_{t}^{\infty})-\Psi_{F,f}(\widehat{M}_{0}^{\infty}) - \int_{0}^{t} \lcal_{µ}^{\infty} \Psi_{F,f}(\widehat{M}_{s}^{\infty})ds
\right)_{t \geq 0}
\end{equation*}
is a martingale. In order to do so, for all $x \in [-||f|| \mathrm{Vol}(\mathrm{Supp}(f)),||f|| \mathrm{Vol}(\mathrm{Supp}(f))]$, we set
\begin{equation*}
\widetilde{F}_{f}(x) = F\left(
\int_{\mathbb{R}^{d}} f(z)dz - x
\right).
\end{equation*}
We extend $\widetilde{F}_{f}$ to $\mathbb{R}$ in such a way that $\widetilde{F}_{f} \in C_{b}^{1}(\mathbb{R})$. 
Then, let $t \geq 0$. We observe that
\begin{align*}
\Psi_{F,f}(\widehat{M}^{\infty}_{t}) &= F\left(
\int_{\mathbb{R}^{d}} f(z)\hat{\omega}^{\infty}_{t}(z) dz 
\right) \\
&= F\left(
\int_{\mathbb{R}^{d}} f(z) \left(
1-\mathds{1}_{E_{t}^{\infty}}(z)
\right)dx
\right) \\
&= F\left(
\int_{\mathbb{R}^{d}} f(z)dz - \int_{E_{t}^{\infty}}f(z)dz
\right) \\
&= \widetilde{F}_{f}\left(
\int_{E_{t}^{\infty}} f(z)dz
\right) \\
&= \Phi_{\widetilde{F}_{f},f} (\Xi_{t}^{\infty}).
\end{align*}
Similarly, 
\begin{equation*}
\Psi_{F,f}(\widehat{M}_{0}^{\infty}) = \Phi_{\widetilde{F}_{f},f}(\Xi_{0}^{\infty}). 
\end{equation*}
Moreover, for all $x \in \mathbb{R}^{d}$ and $s \in [0,t]$, 
\begin{align*}
F\left(
\langle \hat{\omega}^{\infty}_{s}, f \rangle
\right) &= \widetilde{F}_{f}\left(
\int_{E_{s}^{\infty}} f(z)dz
\right) \\
&= \widetilde{F}_{f}\left(
\langle m(E_{s}^{\infty}),f \rangle
\right) \\
\text{and } \quad F\left(
\langle \Theta_{x,\rcal}^{-}(\hat{\omega}_{s}^{\infty}),f \rangle 
\right) 
&= F\left(
\int_{\mathbb{R}^{d}} f(z) \left(
1 - \mathds{1}_{\mathcal{B}(x,\rcal)}(z)
\right) \, \left(
1 - \mathds{1}_{E^{\infty}_{s}}(z)
\right)dz
\right) \\
&= F\left(
\int_{\mathbb{R}^{d}} f(z)dz - \int_{\bcal(x,\rcal)} f(z)dz - \int_{E_{s}^{\infty}}f(z)dz + \int_{\bcal(x,\rcal) \cap E_{s}^{\infty}} f(z)dz
\right) \\
&= F\left(
\int_{\mathbb{R}^{d}} f(z)dz - \int_{\bcal(x,\rcal) \cup E_{s}^{\infty}}f(z)dz
\right) \\
&= \widetilde{F}_{f}\left(
\langle 
m\left(
\bcal(x, \rcal) \cup E_{s}^{\infty}
\right), f
\rangle\right).
\end{align*}
We also have
\begin{align*}
1 - \delta_{0}\left(
\int_{\bcal(x,\rcal)} (1-\hat{\omega}_{s}^{\infty}(z))dz
\right) &= 1 - \delta_{0}\left(
\int_{\bcal(x,\rcal)} \mathds{1}_{E_{s}^{\infty}}(z)dz
\right) \\
&= 1 - \delta_{0}\left(
\int_{\mathbb{R}^{d}} \mathds{1}_{\bcal(x,\rcal)}(z) \mathds{1}_{E_{s}^{\infty}}(z)dz
\right),
\end{align*}
which is equal to $\mathds{1}_{S^{\rcal}(E_{s}^{\infty})}(x)$ for Lebesgue-almost all $x \in \mathbb{R}^{d}$. Therefore, 
\begin{equation*}
\lcal_{µ}^{\infty} \Psi_{F,f} (\widehat{M}_{s}^{\infty}) = \gcal_{µ}^{\infty} \Phi_{\widetilde{F}_{f},f} (\Xi_{s}^{\infty}).
\end{equation*}
This implies that for all $t \geq 0$,
\begin{align*}
&\Psi_{F,f}(\widehat{M}_{t}^{\infty}) - \Psi_{F,f}(\widehat{M}_{0}^{\infty}) - \int_{0}^{t} \lcal_{µ}^{\infty} \Psi_{F,f}(\widehat{M}_{s}^{\infty})ds \\
&= \Phi_{\widetilde{F}_{f},f}(\Xi_{t}^{\infty}) - \Phi_{\widetilde{F}_{f},f}(\Xi_{0}^{\infty}) 
- \int_{0}^{t} \gcal_{µ}^{\infty} \Phi_{\widetilde{F}_{f},f}(\Xi_{s}^{\infty})ds,
\end{align*}
and we conclude using Proposition~\ref{prop:backwards_kinf}. 
\end{proof}

\section{Uniqueness of the solution to the martingale problem characterizing the $\infty$-parent SLFV}\label{sec:4}
The goal of this section is to provide another characterization of the $\infty$-parent SLFV as the unique solution to a martingale problem, as stated in Theorem~\ref{thm:characterization_infty_SLFV}. In order to do so, we show that the $\infty$-parent SLFV (as constructed in Definition~\ref{defn:infty_SLFV}) is a solution to the martingale problem associated to $\lcal_{µ}^{\infty}$. Then, we will extend the set of functions over which the operators $\lcal_{µ}^{\infty}$ and $\gcal_{µ}^{\infty}$ are defined. 

In all that follows, let $µ$ be a $\sigma$-finite measure on $(0,+\infty)$ satisfying Condition (\ref{eqn:condition_intensite}), and let $M^{0} \in \mcal_{\lambda}$ with density $\omega : \mathbb{R}^{d} \to \{0,1\}$. 
We recall that $U = \bcal(0,1)^{\mathbb{N}}$ and that $\tilde{u}$ is the law of a sequence of i.i.d random variables $(\mathcal{P}_{n})_{n \geq 1}$ uniformly distributed over $\bcal(0,1)$.
Let $\Pi^{c}$ be a Poisson point process on $\mathbb{R} \times \mathbb{R}^{d} \times (0,+\infty) \times U$ with intensity 
\begin{equation*}
dt \otimes dx \otimes µ(d\rcal) \otimes \tilde{u}\left(
d(p_{n})_{n \geq 1}
\right). 
\end{equation*}

\subsection{Existence of a solution to the martingale problem associated to $\lcal_{µ}^{\infty}$}\label{subsec:2.2}
We recall that the operator $\lcal_{µ}^{\infty}$ is defined by
\begin{align*}
\lcal_{µ}^{\infty}\Psi_{F,f}(M) = \int_{0}^{+ \infty}\int_{Supp^{\rcal}(f)}&\left(1-\delta_{0}\left(\int_{\bxr}\left(1-\omega_{M}(z)\right)dz\right)\right) \\
&\times \Big[F(\langle\Theta_{x,\rcal}^{-}(\omega_{M}),f\rangle) - F(\langle\omega_{M},f\rangle)
\Big]dxµ(d\rcal).
\end{align*}
The goal of this section is to demonstrate the following result, which is also the first part of Theorem~\ref{thm:characterization_infty_SLFV}.
\begin{prop}\label{prop:char_infty_slfv}
Let $(M_{t}^{\infty})_{t \geq 0}$ be the $\infty$-parent SLFV with initial density $\omega$, associated to $\Pi^{c}$. Then, for all $F \in C^{1}(\mathbb{R})$ and $f \in C_{c}(\rd)$,
\begin{equation*}
\left(
\Psi_{F,f}(M_{t}^{\infty}) - \Psi_{F,f}(M_{0}^{\infty}) - \int_{0}^{t} \lcal_{µ}^{\infty} \Psi_{F,f}(M_{s}^{\infty})ds
\right)_{t \geq 0}
\end{equation*}
is a martingale.
\end{prop}

In other words, $(M_{t}^{\infty})_{t \geq 0}$ is a solution of the martingale problem $(\lcal_{µ}^{\infty},\delta_{M_{0}^{\infty}})$, but this solution is not necessarily unique. In fact, we will show in Section \ref{sec:4} that this solution is unique when $µ$ satisfies the stronger Condition (\ref{eqn:condition_infty_SLFV}), but the question of uniqueness when $µ$ does not satisfy Condition (\ref{eqn:condition_infty_SLFV}) remains open. 

We start by justifying why the operator $\lcal_{µ}^{\infty}$ is a suitable candidate for an operator characterizing the limit $k \to + \infty$ of the $k$-parent SLFV. 
\begin{lem}\label{lem:CVG_partie_linf}
Let $\omega : \rd \to \{0,1\}$ be measurable, and let $x \in \mathbb{R}$. Then, for all $\rcal > 0$,
\begin{equation*}
\delta_{0}\left(\int_{\bxr} \left(1 - \omega(z)\right)dz\right)
= \lim\limits_{k \to + \infty} \frac{1}{V_{\rcal}^{k}}\int_{\bxr^{k}}\left(\prod_{j = 1}^{k} \omega(y_{j})\right)dy_{1}...dy_{k}.
\end{equation*}
\end{lem}

\begin{proof}
For all $k \geq 2$,
\begin{equation*}
\frac{1}{V_{\rcal}^{k}}\int_{\bxr^{k}} \left[\prod_{j = 1}^{k} \omega(y_{j}) \right] dy_{1}...dy_{k}
= \left(\frac{1}{V_{\rcal}}\int_{\bxr} \omega(y)dy \right)^{k}.
\end{equation*}

As $V_{\rcal}^{-1}\int_{\bxr} \omega(y)dy  \in [0,1]$,
\begin{align*}
\lim\limits_{k \to + \infty} \frac{1}{V_{\rcal}^{k}}\int_{\bxr^{k}} \left[\prod_{j = 1}^{k} \omega(y_{j}) \right] dy_{1}...dy_{k} &= 1 \\
\Longleftrightarrow \frac{1}{V_{\rcal}}\int_{\bxr} \omega(y)dy  &= 1 \\
\Longleftrightarrow \int_{\bxr} \left(1 - \omega(z)\right)dz &= 0. \\
\intertext{Moreover,}
\lim\limits_{k \to + \infty} \frac{1}{V_{\rcal}^{k}}\int_{\bxr^{k}} \left[\prod_{j = 1}^{k} \omega(y_{j}) \right] dy_{1}...dy_{k} &= 0 \\
\Longleftrightarrow \frac{1}{V_{\rcal}}\int_{\bxr} \omega(y)dy &< 1 \\
\Longleftrightarrow \int_{\bxr} \left(1 - \omega(z)\right)dz &> 0, 
\end{align*}
and we can conclude.
\end{proof}

Let $F \in C^{1}(\mathds{R})$ and $f \in C_{c}(\rd)$. For all $M \in \mcal_{\lambda}$, 
\begin{equation}\label{eqn:bounded1_2_3}
\left|F\left(\langle \omega_{M},f \rangle \right) \right|
\leq \max\{|F(x)| : x \in [- \mathrm{Vol}(Supp(f)), \mathrm{Vol}(Supp(f))]\},
\end{equation}
which means in particular that for all $x \in \rd$ and for all $\rcal > 0$, 
\begin{align*}
|F(\langle \Theta_{x,\rcal}^{+}(\omega_{M}),f \rangle)| 
&\leq \max\{|F(x)| : x \in [- \mathrm{Vol}(Supp(f)), \mathrm{Vol}(Supp(f))]\} \\
\text{and \quad }
|F(\langle \Theta_{x,\rcal}^{-}(\omega_{M}),f \rangle)| 
&\leq \max\{|F(x)| : x \in [- \mathrm{Vol}(Supp(f)), \mathrm{Vol}(Supp(f))]\}.
\end{align*}
Therefore, a direct consequence of the dominated convergence theorem is the following lemma.

\begin{lem}\label{lem:lem1_2_3}
Let $M \in \mcal_{\lambda}$, and let $(M_{n})_{n \in \mathbb{N}} \in \mcal_{\lambda}$ be such that $M_{n}$ converges vaguely to $M$. Then, for all $x \in \rd$ and for all $\rcal > 0$, 
\begin{align*}
F(\langle \omega_{M_{n}},f \rangle) &\xrightarrow[n \to + \infty]{} F(\langle \omega_{M},f \rangle) \\
F(\langle \Theta_{x,\rcal}^{+}(\omega_{M_{n}}),f \rangle) &\xrightarrow[n \to + \infty]{}
F(\langle \Theta_{x,\rcal}^{+}(\omega_{M}),f \rangle) \\
F(\langle \Theta_{x,\rcal}^{-}(\omega_{M_{n}}),f \rangle) &\xrightarrow[n \to + \infty]{}
F(\langle \Theta_{x,\rcal}^{-}(\omega_{M}),f \rangle).
\end{align*}
\end{lem}

In contrast with $\lcal_{µ}^{k}\Psi_{F,f}$, the function $\lcal_{µ}^{\infty}\Psi_{F,f}$ is not continuous. However, we have the following result.

\begin{lem}\label{lem:continuite_partielle_linf}
Let $M \in \mcal_{\lambda}$, and $(M_{n})_{n \in \mathbb{N}} \in \mcal_{\lambda}$ such that $M_{n}$ converges to $M$ in the topology of vague convergence. Assume that there exists a density $\omega$ of $M$ and densities $\omega_{n}$ of $M_{n}$ for all $n \in \mathbb{N}
$ such that:
\begin{equation*}
\forall n \in \mathbb{N}, \forall z \in \rd, \omega(z) \leq \omega_{n}(z).
\end{equation*}
Then, 
\begin{equation*}
\lim\limits_{n \to + \infty} \lcal_{µ}^{\infty}\Psi_{F,f}(M_{n}) = \lcal_{µ}^{\infty} \Psi_{F,f}(M).
\end{equation*}
\end{lem}

\begin{proof}
First, since $(M_{n})_{n \in \mathbb{N}}$ converges vaguely to $M$, by Lemma \ref{lem:lem1_2_3}, for all $\rcal > 0$ and for all $x \in Supp^{\rcal}(f)$, 
\begin{equation*}
F(\langle\Theta_{x,\rcal}^{-}(\omega_{n}),f\rangle) \xrightarrow[n \to + \infty]{} F(\langle\Theta_{x,\rcal}^{-}(\omega),f\rangle).
\end{equation*}

Then, let $\rcal > 0$ and $n \in \mathbb{N}$. Since for all $z \in \bcal(x,\rcal)$, $\omega(z) \leq \omega_{n}(z)$,
\begin{equation*}
\int_{\bxr}\left(1 - \omega(z)\right) dz \geq \int_{\bxr}\left(1 - \omega_{n}(z)\right)dz.
\end{equation*}
Moreover, since
\begin{align*}
\lim\limits_{n \to + \infty} \int_{\bxr} \left(1 - \omega_{n}(z)\right)dz &= \int_{\bxr} \left(1 - \omega(z)\right)dz \\
&\geq \int_{\bxr} \left(1 - \omega_{n}(z)\right)dz,
\end{align*}
if $\lim\limits_{n \to + \infty} \int_{\bxr} \left(1 - \omega_{n}(z)\right)dz = 0$, then for all $n \in \mathbb{N}$, $\int_{\bxr} \left(1 - \omega_{n}(z)\right)dz = 0$, and thus :
\begin{equation*}
\lim\limits_{n \to + \infty} \delta_{0}\left(\int_{\bxr}\left( 1 - \omega_{n}(z)\right)dz\right) = \delta_{0}\left(\int_{\bxr}\left( 1 - \omega(z)\right)dz \right).
\end{equation*}

Conversely, if $\lim\limits_{n \to + \infty} \int_{\bxr}\left( 1 - \omega_{n}(z)\right)dz \neq 0$, since $\delta_{0}(\bullet)$ is continuously equal to $0$ over $\mathbb{R}_{+}^{*}$, 
\begin{equation*}
\lim\limits_{n \to + \infty} \delta_{0}\left(\int_{\bxr} \left(1 - \omega_{n}(z)\right)dz\right) = \delta_{0}\left(\int_{\bxr}\left( 1 - \omega(z)\right)dz \right).
\end{equation*}
We conclude by using the dominated convergence theorem.
\end{proof}

\begin{lem}\label{lem:cvg_pb_martingale}
Let $(M_{t}^{\infty})_{t \geq 0}$ be the $\infty$-parent SLFV of initial condition $M^{0}$, constructed using the initial density $\omega$ and $\Pi^{c}$. For all $k \geq 2$, let $(M_{u}^{k})_{u \geq 0}$ be a sequence of coupled $k$-parent SLFVs associated to $\Pi^{c}$ and with initial condition $\omega$, constructed as in Definition~\ref{defn:quenched_k_SLFV}.
Then, for all $F \in C^{1}(\mathbb{R})$ and $f \in C_{c}(\rd)$, for all $l \geq 1$, for all $0 \leq t_{1} < ... < t_{l} \leq t < t+s$, for all $h_{1},...,h_{l} \in C_{b}(\mcal_{\lambda})$, 
\begin{equation*}
\lim\limits_{k \to + \infty} \esp\left[\left(
\Psi_{F,f}(M_{t+s}^{k}) - \Psi_{F,f}(M_{t}^{k}) - \int_{t}^{t+s}\lcal_{µ}^{\infty} \Psi_{F,f}(M_{u}^{k})du
\right) \times
\left(\prod_{i = 1}^{l} h_{i}(M_{t_{i}}^{k})\right) 
\right] = 0.
\end{equation*}
\end{lem}

\begin{proof}
For all $k \geq 2$ and $u \geq 0$, let $\omega_{u}^{k}$ be a density of $M_{u}^{k}$.

Let $l \geq 1$, $0 \leq t_{1} < ... < t_{l} \leq t < t+s$ and $h_{1},...,h_{l} \in C_{b}(\mcal_{\lambda})$. Then, for all $k \geq 2$, 
\begin{align*}
&\esp\left[
\left(
\Psi_{F,f}(M_{t+s}^{k}) - \Psi_{F,f}(M_{t}^{k}) - \int_{t}^{t+s}\lcal_{µ}^{\infty}\Psi_{F,f}(M_{u}^{k})du
\right) \times \left(\prod_{i = 1}^{l} h_{i}(M_{t_{i}}^{k})\right)
\right] \\
= &\esp\left[
\left(
\Psi_{F,f}(M_{t+s}^{k}) - \Psi_{F,f}(M_{t}^{k}) - \int_{t}^{t+s}\lcal_{µ}^{k}\Psi_{F,f}(M_{u}^{k})du
\right) \times \left(\prod_{i = 1}^{l} h_{i}(M_{t_{i}}^{k})\right)
\right] \\
&+ \esp\left[
\left(\int_{t}^{t+s}\lcal_{µ}^{k}\Psi_{F,f}(M_{u}^{k}) - \lcal_{µ}^{\infty}\Psi_{F,f}(M_{u}^{k})du
\right) \times \left(\prod_{i = 1}^{l} h_{i}(M_{t_{i}}^{k})\right)
\right]. \\
\intertext{Since $(M_{u}^{k})_{u \geq 0}$ is solution to the martingale problem associated to $(\lcal^{k}, \delta_{M_{0}^{k}})$, the above is equal to}
&0 + \esp\left[
\left(\int_{t}^{t+s}\lcal_{µ}^{k}\Psi_{F,f}(M_{u}^{k}) - \lcal_{µ}^{\infty}\Psi_{F,f}(M_{u}^{k})du
\right) \times \left(\prod_{i = 1}^{l} h_{i}(M_{t_{i}}^{k})\right)
\right]. \\
\intertext{From Lemmas \ref{lem:Lk_well_defined} and \ref{lem:Linfty_well_defined} in Section \ref{sec:5}, we can apply the dominated convergence theorem to}
& \esp\left[
\left(\int_{t}^{t+s}|\lcal_{µ}^{k}\Psi_{F,f}(M_{u}^{k}) - \lcal_{µ}^{\infty}\Psi_{F,f}(M_{u}^{k})|du
\right) \times \left(\prod_{i = 1}^{l} |h_{i}(M_{t_{i}}^{k})|\right)
\right], \\
\intertext{and we obtain}
& \lim\limits_{k \to + \infty} \esp\left[
\left(
\Psi_{F,f}(M_{t+s}^{k}) - \Psi_{F,f}(M_{t}^{k}) - \int_{t}^{t+s}\lcal_{µ}^{\infty}\Psi_{F,f}(M_{u}^{k})du
\right) \times \left(\prod_{i = 1}^{l} h_{i}(M_{t_{i}}^{k})\right)
\right] \\ 
= &\esp\left[
\left(\int_{t}^{t+s}\lim\limits_{k \to + \infty}\left(\lcal_{µ}^{k}\Psi_{F,f}(M_{u}^{k}) - \lcal_{µ}^{\infty}\Psi_{F,f}(M_{u}^{k})\right)du
\right) \times \left(\lim\limits_{k \to + \infty} \prod_{i = 1}^{l} h_{i}(M_{t_{i}}^{k})\right)
\right],
\end{align*}
assuming that the different limits exist. 

Now, let $k \geq 2$ and $u \in [t,t+s]$. We have
\begin{align*}
\lcal_{µ}^{k}\Psi_{F,f}(M_{u}^{k}) &- \lcal_{µ}^{\infty}\Psi_{F,f}(M_{u}^{k}) \\
= \int_{0}^{\infty}\int_{Supp^{\rcal}(f)} & \left(
F(\langle\Theta_{x,\rcal}^{+}(\omega_{u}^{k}),f\rangle) - F(\langle\omega_{u}^{k},f\rangle)
\right) \\
&\times \left[
\int_{\bxr^{k}}\prod_{j = 1}^{k} \left(\frac{\omega_{u}^{k}(y_{j})}{V_{\rcal}}\right)dy_{1}...dy_{k} - \delta_{0}\left(\int_{\bxr} \left(1 - \omega_{u}^{k}(y)\right)dy\right)
\right] \\
+& \left(
F(\langle\Theta_{x,\rcal}^{-}(\omega_{u}^{k}),f\rangle) - F(\langle\omega_{u}^{k},f\rangle)
\right) \\
&\times \left[
\delta_{0}\left(\int_{\bxr} \left(1 - \omega_{u}^{k}(y)\right)dy\right) - \int_{\bxr^{k}}\prod_{j = 1}^{k} \left(\frac{\omega_{u}^{k}(y_{j})}{V_{\rcal}}\right)dy_{1}...dy_{k}
\right]dxµ(d\rcal) \\
= \int_{0}^{\infty}\int_{Supp^{\rcal}(f)}& \left(
F(\langle\Theta_{x,\rcal}^{+}(\omega_{u}^{k}),f\rangle) - F(\langle\Theta_{x,\rcal}^{-}(\omega_{u}^{k}),f\rangle)
\right) \\
&\times
\left[
\int_{\bxr^{k}}\prod_{j = 1}^{k} \left(\frac{\omega_{u}^{k}(y_{j})}{V_{\rcal}}\right)dy_{1}...dy_{k} - \delta_{0}\left(\int_{\bxr}\left( 1 - \omega_{u}^{k}(y)\right)dy\right)
\right]dxµ(d\rcal).
\end{align*}
The term inside the integral is bounded in absolute value, by Lemma \ref{lem:appendix_A_1} in Section \ref{sec:5}. Moreover, as $(M_{u}^{k})_{k \geq 2}$ converges vaguely to $M_{u}^{\infty}$ by Lemma \ref{lem:cvg_vague_processus_tps}, we can apply Lemma \ref{lem:lem1_2_3} and we obtain
\begin{equation*}
\lim\limits_{k \to + \infty} F(\langle\Theta_{x,\rcal}^{+}(\omega_{u}^{k}),f\rangle) - F(\langle\Theta_{x,\rcal}^{-}(\omega_{u}^{k}),f\rangle) = F(\langle\Theta_{x,\rcal}^{+}(\omega_{u}^{\infty}),f\rangle) - F(\langle\Theta_{x,\rcal}^{-}(\omega_{u}^{\infty}),f\rangle).
\end{equation*}
Therefore, we have to show that
\begin{equation*}
\lim\limits_{k \to + \infty} \int_{\bxr^{k}}\prod_{j = 1}^{k} \left(\frac{\omega_{u}^{k}(y_{j})}{V_{\rcal}}\right)dy_{1}...dy_{k} - \delta_{0}\left(\int_{\bxr} \left(1 - \omega_{u}^{k}(y)\right)dy\right) = 0.
\end{equation*}
We cannot apply directly Lemma \ref{lem:CVG_partie_linf}, because the density also depends on $k$. However,
\begin{align*}
&\left|\int_{\bxr^{k}}\prod_{j = 1}^{k} \left(\frac{\omega_{u}^{k}(y_{j})}{V_{\rcal}}\right)dy_{1}...dy_{k} - \delta_{0}\left(\int_{\bxr}\left( 1 - \omega_{u}^{k}(y)\right)dy\right)\right| \\
\leq &\left|\int_{\bxr^{k}} \left(\prod_{j = 1}^{k} \frac{\omega_{u}^{k}(y_{j})}{V_{\rcal}} - \prod_{j = 1}^{k}\frac{\omega_{u}^{\infty}(y_{j})}{V_{\rcal}}\right) dy_{1}...dy_{n_{k}}\right| \\
+ &\left|\delta_{0}\left(\int_{\bxr}\left( 1 - \omega_{u}^{k}(y)\right)dy\right) - \delta_{0}\left(\int_{\bxr} \left(1 - \omega_{u}^{\infty}(y)\right)dy\right)\right| \\
+ &\left|\int_{\bxr^{k}}\prod_{j = 1}^{k} \left(\frac{\omega_{u}^{\infty}(y_{j})}{V_{\rcal}}\right)dy_{1}...dy_{k} - \delta_{0}\left(\int_{\bxr} \left(1 - \omega_{u}^{\infty}(y)\right)dy\right)\right|.
\end{align*}
We can apply Lemma \ref{lem:CVG_partie_linf} to the third term. Since for all $y \in \rd$, $\omega_{u}^{\infty}(y) \leq \omega_{u}^{k}(y)$, we showed in the proof of Lemma \ref{lem:continuite_partielle_linf} that
\begin{equation*}
\lim\limits_{k \to + \infty} \left|\delta_{0}\left(\int_{\bxr}\left( 1 - \omega_{u}^{k}(y)\right)dy\right) - \delta_{0}\left(\int_{\bxr}\left( 1 - \omega_{u}^{\infty}(y)\right)dy\right)\right| = 0.
\end{equation*}

Regarding the first term, we distinguish two cases. If $V_{\rcal}^{-1}\int_{\bxr} \omega_{u}^{\infty}(y) dy = 1$, since
\begin{equation*}
\int_{\bxr} \frac{\omega_{u}^{\infty}(y)}{V_{\rcal}}dy \leq \int_{\bxr} \frac{\omega_{u}^{k}(y)}{V_{\rcal}}dy \leq 1,
\end{equation*}
we obtain that in fact for every $k \geq 2$
\begin{equation*}
\int_{\bxr^{k}} \prod_{j = 1}^{k} \left(\frac{\omega_{u}^{k}(y_{j})}{V_{\rcal}} - \prod_{j = 1}^{k}\frac{\omega_{u}^{\infty}(y_{j})}{V_{\rcal}} \right)dy_{1}...dy_{n_{k}} = 0.
\end{equation*}
Conversely, assume $V_{\rcal}^{-1}\int_{\bxr} \omega_{u}^{\infty}(y)dy < 1$. Since,
\begin{equation*}
\int_{\bxr} \frac{\omega_{u}^{k}(y)}{V_{\rcal}}dy \xrightarrow[k \to + \infty]{} \int_{\bxr}\frac{\omega_{u}^{\infty}(y)}{V_{\rcal}}dy,
\end{equation*}
there exist $0 < M < 1$ and $k' \geq 2$ such that :
\begin{equation*}
\forall k \geq k', \int_{\bxr} \frac{\omega_{u}^{k}(y)}{V_{\rcal}}dy \leq M.
\end{equation*}
Therefore, 
\begin{align*}
\left|\int_{\bxr^{k}} \prod_{j = 1}^{k}\left(\frac{\omega_{u}^{k}(y_{j})}{V_{\rcal}}\right) - \prod_{j = 1}^{k}\left(\frac{\omega_{u}^{\infty}(y_{j})}{V_{\rcal}}\right) dy_{1}...dy_{k}\right| 
&= \left|\left(
\int_{\bxr}\frac{\omega_{u}^{k}(y)}{V_{\rcal}}dy
\right)^{k} - \left(
\int_{\bxr} \frac{\omega_{u}^{\infty}(y)}{V_{\rcal}}dy
\right)^{k}\right| \\
&\xrightarrow[k \to + \infty]{} 0,
\end{align*}
and we can conclude.
\end{proof}

We can now show that the $\infty$-parent SLFV is solution to the martingale problem introduced in Proposition \ref{prop:char_infty_slfv}.

\begin{proof}(Proposition \ref{prop:char_infty_slfv}) 
Let $F \in C^{1}(\mathbb{R})$ and $f \in C_{c}(\rd)$. For all $k \geq 2$, we set $(M_{t}^{k})_{t \geq 0} = (M_{k,t}^{\Pi^{c},\omega})_{t \geq 0}$. Let $l \geq 1$, let $0 \leq t_{1} < ... < t_{l} \leq t < t+s$, and let $h_{1},...,h_{l} \in C_{b}(\mcal_{\lambda})$. By Lemma \ref{lem:cvg_pb_martingale},
\begin{equation*}
\lim\limits_{k \to + \infty} \esp\left[\left(
\Psi_{F,f}(M_{t+s}^{k}) - \Psi_{F,f}(M_{t}^{k}) - \int_{t}^{t+s}\lcal_{µ}^{\infty} \Psi_{F,f}(M_{u}^{k})du
\right) \times
\left(\prod_{i = 1}^{l} h_{i}(M_{t_{i}}^{k})\right) 
\right] = 0.
\end{equation*}

Since $(M_{t+s}^{k})_{k \geq 2}$ (resp. $(M_{t}^{k})_{k \geq 2}$) converges vaguely to $M_{t+s}^{\infty}$ (resp. $M_{t}^{\infty}$) by Lemma \ref{lem:cvg_vague_processus_tps}, we can apply Lemma \ref{lem:lem1_2_3} and we obtain
\begin{align*}
\lim\limits_{k \to + \infty} \Psi_{F,f}(M_{t+s}^{k}) &= \Psi_{F,f}(M_{t+s}^{\infty}) \\
\text{and } \lim\limits_{k \to + \infty} \Psi_{F,f}(M_{t}^{k}) &= \Psi_{F,f}(M_{t}^{\infty}).
\end{align*}

Moreover, by Lemma \ref{lem:continuite_partielle_linf}, for all $u \in [t,t+s]$, 
\begin{equation*}
\lim\limits_{k \to + \infty} \lcal_{µ}^{\infty}\Psi_{F,f}(M_{u}^{k})
= \lcal_{µ}^{\infty}\Psi_{F,f}(M_{u}^{\infty}),
\end{equation*}
which is uniformly bounded in $M \in \mcal_{\lambda}$ by Lemma \ref{lem:Linfty_well_defined} in Section \ref{sec:5}. Since for all $i \in \llbracket 1, l \rrbracket$, $h_{i} \in C_{b}(\mcal_{\lambda})$, 
\begin{equation*}
\forall i \in \llbracket 1,l \rrbracket, \lim\limits_{k \to + \infty} h_{i}(M_{t_{i}}^{k}) = h_{i}(M_{t_{i}}^{\infty}).
\end{equation*}

Therefore, by Eq.(\ref{eqn:bounded1_2_3}) and by Lemmas \ref{lem:Lk_well_defined}, \ref{lem:Linfty_well_defined} in Section \ref{sec:5}, we can apply the dominated convergence theorem and obtain
\begin{equation*}
\esp\left[\left(
\Psi_{F,f}(M_{t+s}^{\infty}) - \Psi_{F,f}(M_{t}^{\infty}) - \int_{t}^{t+s}\lcal_{µ}^{\infty} \Psi_{F,f}(M_{u}^{\infty})du
\right) \times
\left(\prod_{i = 1}^{l} h_{i}(M_{t_{i}}^{\infty})\right) 
\right] = 0.
\end{equation*}

We conclude that
\begin{equation*}
\left(
\Psi_{F,f}(M_{t}^{\infty}) - \Psi_{F,f}(M_{0}^{\infty}) - \int_{0}^{t} \lcal_{µ}^{\infty} \Psi_{F,f} (M_{u}^{\infty})du
\right)_{t \geq 0}
\end{equation*}
is indeed a martingale.
\end{proof}

\subsection{Extended martingale problem for the $\infty$-parent SLFV}\label{subsec:4.1}
For all $\alpha \in \mathbb{R}$, we set $F^{\alpha} : x \to \mathds{1}_{\{\alpha\}}(x)$, and for all $E \in \mathcal{E}^{cf}$, we set $f^{E} : x \to \mathds{1}_{x \in E}$. The goal of this section is to prove the following result. 

\begin{lem}\label{lem:extension_pb_forwards_linf}
Let $\widetilde{M}$ be a solution to the martingale problem associated to $(\lcal_{µ}^{\infty}, \delta_{M^{0}})$. Then, for all $E \in \mathcal{E}^{cf}$, 
\begin{equation*}
\left(
\Psi_{F^{\mathrm{Vol}(E)},f^{E}}(\widetilde{M}_{t})-\Psi_{F^{\mathrm{Vol}(E)},f^{E}}(\widetilde{M}_{0}) - \int_{0}^{t}\lcal_{µ}^{\infty}\Psi_{F^{\mathrm{Vol}(E)},f^{E}}(\widetilde{M}_{s})ds
\right)_{t \geq 0}
\end{equation*}
is a martingale, where $\Psi_{F^{\mathrm{Vol}(E)},f^{E}} : M \in \mcal_{\lambda} \to \Psi_{F^{\mathrm{Vol}(E)},f^{E}}(M)$ is the function defined by
\begin{align*}
\forall M \in \mcal_{\lambda}, 
\Psi_{F^{\mathrm{Vol}(E)},f^{E}}(M) :&= F^{\mathrm{Vol}(E)}\left(\langle \omega_{M},f \rangle \right) \\
&= \mathds{1}_{\{\mathrm{Vol}(E)\}}\left(\langle \omega_{M},f \rangle \right) \\
&= \delta_{0}\left(
\mathrm{Vol}(E) - \langle \omega_{M},f \rangle
\right).
\end{align*}
\end{lem}

This lemma is a direct consequence of the following lemma.
\begin{lem}\label{lem:reformulation_extension_pb_forwards_linf}
Let $\widetilde{M}$ be a solution to the martingale problem associated to $(\lcal_{µ}^{\infty},\delta_{M^{0}})$. Then, for all $E \in \mathcal{E}^{cf}$, for all $l \geq 1$, for all $0 \leq t_{1} < ... < t_{l} \leq t < t+s$, for all $h_{1},...,h_{l} \in C_{b}(\mcal_{\lambda})$,
\begin{equation*}
\esp\left[
\left(
\Psi_{F^{\mathrm{Vol}(E)},f^{E}}(\widetilde{M}_{t+s}) - \Psi_{F^{\mathrm{Vol}(E)},f^{E}}(\widetilde{M}_{t}) - \int_{t}^{t+s}\lcal_{µ}^{\infty} \Psi_{F^{\mathrm{Vol}(E)},f^{E}}(\widetilde{M}_{u})du\right)\times \left(
\prod_{i = 1}^{l} h_{i}(\widetilde{M}_{t_{i}})
\right)
\right] = 0
\end{equation*}
\end{lem}

Let $E \in \mathcal{E}^{cf}$. Let $(F_{n}^{\mathrm{Vol}(E)})_{n \in \mathbb{N}} \in C^{1}(\mathbb{R})$ and $(f_{n}^{E})_{n \in \mathbb{N}} \in C_{c}(\rd)$ be two sequences satisfying the following conditions.
\begin{align*}
\text{(A) }& F_{n}^{\mathrm{Vol}(E)} \xrightarrow[n \to + \infty]{} F^{\mathrm{Vol}(E)}\text{ pointwise and in }L^{1}, \\
\text{(B) }& f_{n}^{E} \xrightarrow[n \to + \infty]{} f^{E} \text{ pointwise and in }L^{1}, \\
\text{(C) }& \forall n \in \mathbb{N}, \forall x \in \mathbb{R}, 0 \leq F_{n}^{\mathrm{Vol}(E)}(x) \leq 1 \text{ and } F_{n}^{\mathrm{Vol}(E)}(\mathrm{Vol}(E)) = 1, \\
\text{(D) }& \forall n \in \mathbb{N}, \forall x \in \rd, 0 \leq f_{n}^{E}(x) \leq 1 \text{ and } \forall z \in E, f_{n}^{E}(z) = 1, \\
\text{(E) }& \forall n \in \mathbb{N}, F_{n}^{\mathrm{Vol}(E)} \text{ is increasing over } (-\infty, \mathrm{Vol}(E)] \text{ and decreasing over } [\mathrm{Vol}(E),+\infty), \\
\text{(F) }& \forall n \in \mathbb{N}, \mathrm{Vol}(Supp(f_{n}^{E})\backslash E) \leq n^{-1}, \text{ and } Supp(f_{n+1}^{E}) \subseteq Supp(f_{n}^{E})\\
\text{(G) }& \forall n \in \mathbb{N}, F_{n}^{\mathrm{Vol}(E)}\left(\mathrm{Vol}(E)+n^{-1}\right) \geq 1 - n^{-1} \text{ and } F_{n}^{\mathrm{Vol}(E)}\left(\mathrm{Vol}(E)-n^{-1}\right) \geq 1 - n^{-1}.
\end{align*}

First, we observe that since $F^{\mathrm{Vol}(E)}$ and $(F_{n}^{\mathrm{Vol}(E)})_{n \in \mathbb{N}^{*}}$ are bounded by one (by Hypothesis (C)), for all $M \in \mcal_{\lambda}$ and $n \in \mathbb{N}^{*}$
\begin{align}
\left|\Psi_{F^{\mathrm{Vol}(E)},f^{E}}(M)\right| &\leq 1 \label{eqn:eq_psi_1} \\
\left|\Psi_{F_{n}^{\mathrm{Vol}(E)},f^{E}}(M)\right| &\leq 1 \label{eqn:eq_psi_2} \\
\left|\Psi_{F_{n}^{\mathrm{Vol}(E)},f_{n}^{E}}(M)\right| &\leq 1 \label{eqn:eq_psi_3}
\end{align}
Moreover, there exists $C_{E} > 0$ such that for all $\rcal > 0$, 
\begin{equation}\label{eqn:bound_sre}
\mathrm{Vol}(S^{\rcal}(E)) \leq C^{E} \times \left(\rcal^{d} \vee 1 \right),
\end{equation}
where we recall that $S^{\rcal}(E)$ is defined by
\begin{equation*}
S^{\rcal}(E) := \{ x \in \rd : \exists y \in E, ||x-y|| \leq \rcal\}.
\end{equation*}
Therefore, we have the following lemma.

\begin{lem}\label{lem:bound_linfty}
There exists $C_{2}^{E} > 0$ such that for all $E \in \mathcal{E}^{cf}$, $M \in \mcal_{\lambda}$ and $n \in \mathbb{N}^{*}$, 
\begin{align*}
\left|\lcal_{µ}^{\infty} \Psi_{F^{\mathrm{Vol}(E)},f^{E}}(M)\right|
&\leq C_{2}^{E} \\
\left|\lcal_{µ}^{\infty} \Psi_{F_{n}^{\mathrm{Vol}(E)},f^{E}}(M)\right|
&\leq C_{2}^{E} \\
\left|\lcal_{µ}^{\infty} \Psi_{F_{n}^{\mathrm{Vol}(E)},f_{n}^{E}}(M)\right|
&\leq C_{2}^{E}.
\end{align*}
\end{lem}

\begin{proof}
Let $M \in \mcal_{\lambda}$.
\begin{align*}
\left|\lcal_{µ}^{\infty} \Psi_{F^{\mathrm{Vol}(E)},f^{E}}(M)\right| 
\leq &\int_{0}^{\infty}\int_{S^{\rcal}(E)}
\left|1 - \delta_{0}\left(
\int_{\bxr}\left(1-\omega_{M}(z)\right)dz
\right)\right| \\
&\quad \quad  \quad \quad \quad \times
\left|
F^{\mathrm{Vol}(E)}\left(\langle \Theta_{x,\rcal}^{-}(\omega_{M}),f^{E} \rangle\right) - F^{\mathrm{Vol(E)}}\left(\langle \omega_{M},f^{E} \rangle\right)
\right| dx µ(d\rcal) \\
\leq &\int_{0}^{\infty} \int_{S^{\rcal}(E)} 2 dx µ(d\rcal)  \\
\leq & 2 \, \int_{0}^{\infty} C^{E} \, \left(\rcal^{d} \vee 1\right)µ(d\rcal) \\
< & + \infty
\end{align*}
since $µ$ satisfies Condition (\ref{eqn:condition_intensite}). Here we passed from line $1$ to line $2$ using the fact that $F^{\mathrm{Vol}(E)}$ is bounded by $1$, and from line $2$ to line $3$ using Eq. (\ref{eqn:bound_sre}).

Setting $C_{2}^{E} = 2 C^{E} \times \int_{0}^{\infty} \left(\rcal^{d} \vee 1 \right) µ(d\rcal)$, we obtain
\begin{equation*}
\left|\lcal_{µ}^{\infty} \Psi_{F^{\mathrm{Vol}(E)},f^{E}}(M)\right|
\leq C_{2}^{E}.
\end{equation*}

Similarly, we can show that for all $n \in \mathbb{N}^{*}$,
\begin{align*}
\left|\lcal_{µ}^{\infty} \Psi_{F_{n}^{\mathrm{Vol}(E)},f^{E}}(M)\right|
&\leq C_{2}^{E} \\
\text{and \quad} \left|\lcal_{µ}^{\infty} \Psi_{F_{n}^{\mathrm{Vol}(E)},f_{n}^{E}}(M)\right|
&\leq C_{2}^{E}.
\end{align*}
\end{proof}

This lemma, along with Eqs. (\ref{eqn:eq_psi_1}, \ref{eqn:eq_psi_2}, \ref{eqn:eq_psi_3}, \ref{eqn:bound_sre}), will allow us to use the dominated convergence theorem in the proof of Lemma \ref{lem:reformulation_extension_pb_forwards_linf}. 

Since by Hypothesis (A) the sequence $(F_{n}^{\mathrm{Vol}(E)})_{n \in \mathbb{N}^{*}}$ converges pointwise to $F^{\mathrm{Vol}(E)}$, we obtain that
\begin{equation}\label{eqn:cvg_Fnf}
\forall M \in \mcal_{\lambda}, \Psi_{F_{n}^{\mathrm{Vol}(E)},f^{E}}(M)
\xrightarrow[n \to + \infty]{}
\Psi_{F^{\mathrm{Vol}(E)},f^{E}}(M).
\end{equation}
We want to show a similar result regarding $\left(\Psi_{F_{n}^{\mathrm{Vol}(E)},f_{n}^{E}}(M)\right)_{n \in \mathbb{N}^{*}}$. 

\begin{lem}\label{lem:cvg_Fnfn}
For all $M \in \mcal_{\lambda}$ and $E \in \mathcal{E}^{cf}$, 
\begin{equation*}
\Psi_{F_{n}^{\mathrm{Vol}(E)},f_{n}^{E}}(M) - \Psi_{F_{n}^{\mathrm{Vol}(E)}, f^{E}}(M) \xrightarrow[n \to + \infty]{} 0. 
\end{equation*}
\end{lem}

\begin{proof}
Let $M \in \mcal_{\lambda}$. We distinguish two cases. 

\underline{\textsl{Case 1 :}}  $\int_{E} \omega_{M}(z)dz = \mathrm{Vol}(E)$. 

Let $n \in \mathbb{N}^{*}$. Then, since by Hypothesis (D) we have $E \subseteq Supp(f_{n}^{E})$, 
\begin{align*}
\mathrm{Vol}(E) \leq \langle \omega_{M},f_{n}^{E} \rangle 
&\leq \mathrm{Vol}(E) + \int_{Supp(f_{n}^{E})\backslash E} f_{n}^{E}(z)\omega_{M}(z)dz \\
&\leq \mathrm{Vol}(E) + \mathrm{Vol}(Supp(f_{n}^{E})\backslash E)  \\
&\leq \mathrm{Vol}(E) + \frac{1}{n} 
\end{align*}
using Hypotheses (D) and (F).
Therefore, since $F_{n}^{\mathrm{Vol}(E)}$ is decreasing over $[\mathrm{Vol}(E),+\infty)$ by Hypothesis (E),
\begin{align*}
F_{n}^{\mathrm{Vol}(E)}(\mathrm{Vol}(E)) \geq \Psi_{F_{n}^{\mathrm{Vol}(E)},f_{n}^{E}}(M) &\geq F_{n}^{\mathrm{Vol}(E)}(\mathrm{Vol}(E) + \frac{1}{n}) \\
\intertext{or, in other words,}
1 \geq \Psi_{F_{n}^{\mathrm{Vol}(E)},f_{n}^{E}}(M) &\geq 1 - \frac{1}{n}
\end{align*}
by Hypothesis (C) and (G). 
Moreover, 
\begin{align*}
\Psi_{F_{n}^{\mathrm{Vol}(E)},f^{E}}(M) &= F_{n}^{\mathrm{Vol}(E)}\left(\int_{E} \omega_{M}(z)dz\right) \\
&= F_{n}^{\mathrm{Vol}(E)}(\mathrm{Vol}(E)) \\
&= 1 
\end{align*}
by Hypothesis (C), and we can conclude.

\underline{\textsl{Case 2 :}}  $\int_{E} \omega_{M}(z)dz < \mathrm{Vol}(E)$. 

Let $N \in \mathbb{N}^{*}$ be such that 
$N^{-1} \leq 2^{-1} \, \left(
\mathrm{Vol}(E) - \int_{E} \omega_{M}(z)dz
\right)$. Then, for all $n \geq N$, using Hypotheses (D) and (F),
\begin{align*}
0 \leq \langle \omega_{M},f_{n}^{E} \rangle
\leq & \int_{E} \omega_{M}(z)dz 
+ \int_{Supp(f_{n}^{E})\backslash E} \omega_{M}(z)dz \\
\leq & \int_{E} \omega_{M}(z)dz + \mathrm{Vol}(Supp(f_{n}^{E}) \backslash E) \\
\leq & \int_{E} \omega_{M}(z)dz + \frac{1}{n} \\
\leq & \int_{E} \omega_{M}(z)dz + \frac{1}{N} \\
\leq & \frac{1}{2} \, \int_{E} \omega_{M}(z)dz + \frac{1}{2} \, \mathrm{Vol}(E) \\
< & \mathrm{Vol}(E),
\end{align*}
so by Hypothesis (E),
\begin{equation*}
\Psi_{F_{n}^{\mathrm{Vol}(E)},f_{n}^{E}}(M) \xrightarrow[n \to + \infty]{} 0.
\end{equation*}
Moreover, since $\langle \omega_{M},f^{E} \rangle < \mathrm{Vol}(E)$, again by Hypothesis (E),
\begin{equation*}
\Psi_{F_{n}^{\mathrm{Vol}(E)},f^{E}}(M) \xrightarrow[n \to + \infty]{} 0,
\end{equation*}
and we can conclude.
\end{proof}

We now prove a similar result involving $\lcal_{µ}^{\infty}$.
\begin{lem}\label{lem:cvg_linf}
For all $M \in \mcal_{\lambda}$ and $E \in \mathcal{E}^{cf}$,
\begin{align*}
\lcal_{µ}^{\infty}\Psi_{F_{n}^{\mathrm{Vol}(E)},f^{E}}(M)
- \lcal_{µ}^{\infty}\Psi_{F^{\mathrm{Vol}(E)},f^{E}}(M)
&\xrightarrow[n \to + \infty]{} 0, \\
\lcal_{µ}^{\infty}\Psi_{F_{n}^{\mathrm{Vol}(E)},f^{E}}(M)
- \lcal_{µ}^{\infty}\Psi_{F_{n}^{\mathrm{Vol}(E)},f_{n}^{E}}(M)
&\xrightarrow[n \to + \infty]{} 0.
\end{align*}
\end{lem}

\begin{proof}
Let $M \in \mcal_{\lambda}$, and let $n \in \mathbb{N}^{*}$. We have
\begin{align*}
\lcal_{µ}^{\infty}\Psi_{F_{n}^{\mathrm{Vol}(E)},f^{E}}&(M)
- \lcal_{µ}^{\infty}\Psi_{F^{\mathrm{Vol}(E)},f^{E}}(M) \\
= \int_{0}^{\infty}\int_{S^{\rcal}(E)}&
\left(
1 - \delta_{0}\left(\int_{\bxr}\left(1-\omega_{M}(z)\right)dz\right)
\right)\\
&\times \left(
F_{n}^{\mathrm{Vol}(E)}\left(\langle \Theta_{x,\rcal}^{-}(\omega_{M}),f^{E}\rangle\right)-F_{n}^{\mathrm{Vol}(E)}\left(\langle \omega_{M},f^{E} \rangle\right)
\right)dxµ(d\rcal) \\
 -\int_{0}^{\infty}\int_{S^{\rcal}(E)}&
\left(
1 - \delta_{0}\left(\int_{\bxr}\left(1-\omega_{M}(z)\right)dz\right)
\right)\\
&\times \left(
F^{\mathrm{Vol}(E)}\left(\langle \Theta_{x,\rcal}^{-}(\omega_{M}),f^{E}\rangle\right)-F^{\mathrm{Vol}(E)}\left(\langle \omega_{M},f^{E} \rangle\right)
\right)dxµ(d\rcal) \\
= \int_{0}^{\infty}\int_{S^{\rcal}(E)} &\left(
1 - \delta_{0}\left(\int_{\bxr}\left(1-\omega_{M}(z)\right)dz\right)
\right)\\
&\times \left(
F_{n}^{\mathrm{Vol}(E)}\left(\langle \Theta_{x,\rcal}^{-}(\omega_{M}),f^{E}\rangle\right) - F^{\mathrm{Vol}(E)}\left(\langle \Theta_{x,\rcal}^{-}(\omega_{M}),f^{E}\rangle\right)
\right)dxµ(d\rcal) \\
+ \int_{0}^{\infty}\int_{S^{\rcal}(E)} &\left(
1 - \delta_{0}\left(\int_{\bxr}\left(1-\omega_{M}(z)\right)dz\right)
\right)\\
&\times \left(
F^{\mathrm{Vol}(E)}\left(\langle \omega_{M},f^{E}\rangle\right) - F_{n}^{\mathrm{Vol}(E)}\left(\langle \omega_{M},f^{E} \rangle\right)
\right)dxµ(d\rcal).
\end{align*}
By Eq. (\ref{eqn:cvg_Fnf}), for all $x \in \rd$ and $\rcal > 0$, 
\begin{align*}
F_{n}^{\mathrm{Vol}(E)}\left(\langle \Theta_{x,\rcal}^{-}(\omega_{M}),f^{E}\rangle\right) - F^{\mathrm{Vol}(E)}\left(\langle \Theta_{x,\rcal}^{-}(\omega_{M}),f^{E}\rangle\right)
&\xrightarrow[n \to + \infty]{} 0\\
\text{and \quad\quad \quad\quad \quad\quad \quad} 
F^{\mathrm{Vol}(E)}\left(\langle \omega_{M},f^{E}\rangle\right) - F_{n}^{\mathrm{Vol}(E)}\left(\langle \omega_{M},f^{E} \rangle\right) 
&\xrightarrow[n \to + \infty]{} 0.
\end{align*}
Therefore, using the bounds from the proof of Lemma \ref{lem:bound_linfty}, we can apply the dominated convergence theorem and obtain
\begin{equation*}
\lcal_{µ}^{\infty}\Psi_{F_{n}^{\mathrm{Vol}(E)},f^{E}}(M)
- \lcal_{µ}^{\infty}\Psi_{F^{\mathrm{Vol}(E)},f^{E}}(M)
\xrightarrow[n \to + \infty]{} 0.
\end{equation*}

We can similarly show that
\begin{equation*}
\lcal_{µ}^{\infty}\Psi_{F_{n}^{\mathrm{Vol}(E)},f^{E}}(M)
- \lcal_{µ}^{\infty}\Psi_{F_{n}^{\mathrm{Vol}(E)},f_{n}^{E}}(M)
\xrightarrow[n \to + \infty]{} 0
\end{equation*}
using Lemma \ref{lem:cvg_Fnfn} instead of Eq. (\ref{eqn:cvg_Fnf}).
\end{proof}

We can now prove Lemma \ref{lem:reformulation_extension_pb_forwards_linf}, from which we will directly deduce Lemma \ref{lem:extension_pb_forwards_linf}. 

\begin{proof}(Lemma \ref{lem:reformulation_extension_pb_forwards_linf}) 
Let $l \geq 1$, let $0 \leq t_{1} < ... < t_{l} \leq t < t+s$ and let $h_{1},...,h_{l} \in C_{b}(\mcal_{\lambda})$. Let $n \in \mathbb{N}^{*}$. Then,
\begin{align*}
\Psi_{F^{\mathrm{Vol}(E)},f^{E}}(\widetilde{M}_{t+s})
=  &\Psi_{F^{\mathrm{Vol}(E)},f^{E}}(\widetilde{M}_{t+s}) - \Psi_{F_{n}^{\mathrm{Vol}(E)},f^{E}}(\widetilde{M}_{t+s}) \\
&+ \Psi_{F_{n}^{\mathrm{Vol}(E)},f^{E}}(\widetilde{M}_{t+s}) - \Psi_{F_{n}^{\mathrm{Vol}(E)},f_{n}^{E}}(\widetilde{M}_{t+s}) + \Psi_{F_{n}^{\mathrm{Vol}(E)},f_{n}^{E}}(\widetilde{M}_{t+s}) \\
\Psi_{F^{\mathrm{Vol}(E)},f^{E}}(\widetilde{M}_{t})
= &\Psi_{F^{\mathrm{Vol}(E)},f^{E}}(\widetilde{M}_{t}) - \Psi_{F_{n}^{\mathrm{Vol}(E)},f^{E}}(\widetilde{M}_{t}) \\
&+ \Psi_{F_{n}^{\mathrm{Vol}(E)},f^{E}}(\widetilde{M}_{t}) - \Psi_{F_{n}^{\mathrm{Vol}(E)},f_{n}^{E}}(\widetilde{M}_{t}) + \Psi_{F_{n}^{\mathrm{Vol}(E)},f_{n}^{E}}(\widetilde{M}_{t}), \\
\intertext{and for all $u \in [t,t+s]$, }
\lcal_{µ}^{\infty}\Psi_{F^{\mathrm{Vol}(E)},f^{E}}(\widetilde{M}_{u})
= &\lcal_{µ}^{\infty}\Psi_{F^{\mathrm{Vol}(E)},f^{E}}(\widetilde{M}_{u}) - \lcal_{µ}^{\infty}\Psi_{F_{n}^{\mathrm{Vol}(E)},f^{E}}(\widetilde{M}_{u}) \\
&+ \lcal_{µ}^{\infty}\Psi_{F_{n}^{\mathrm{Vol}(E)},f^{E}}(\widetilde{M}_{u}) - \lcal_{µ}^{\infty}\Psi_{F_{n}^{\mathrm{Vol}(E)},f_{n}^{E}}(\widetilde{M}_{u}) + \lcal_{µ}^{\infty}\Psi_{F_{n}^{\mathrm{Vol}(E)},f_{n}^{E}}(\widetilde{M}_{u}). 
\end{align*}

Since $\widetilde{M}$ is a solution of the martingale problem associated to $(\lcal_{µ}^{\infty},\delta_{M^{0}})$, for all $n \in \mathbb{N}^{*}$, 
\begin{equation*}
\esp\left[
\left(
\Psi_{F_{n}^{\mathrm{Vol}(E)},f_{n}^{E}}(\widetilde{M}_{t+s}) - \Psi_{F_{n}^{\mathrm{Vol}(E)},f_{n}^{E}}(\widetilde{M}_{t}) - \int_{t}^{t+s}\lcal_{µ}^{\infty} \Psi_{F_{n}^{\mathrm{Vol}(E)},f_{n}^{E}}(\widetilde{M}_{u})du\right)\times \left(
\prod_{i = 1}^{l} h_{i}(\widetilde{M}_{t_{i}})
\right)
\right] = 0.
\end{equation*}
Therefore, since all the equations written above are true for all $n \in \mathbb{N}^{*}$, 
\begin{align*}
&\esp\left[
\left(
\Psi_{F^{\mathrm{Vol}(E)},f^{E}}(\widetilde{M}_{t+s}) - \Psi_{F^{\mathrm{Vol}(E)},f^{E}}(\widetilde{M}_{t}) - \int_{t}^{t+s}\lcal_{µ}^{\infty} \Psi_{F^{\mathrm{Vol}(E)},f^{E}}(\widetilde{M}_{u})du\right)\times \left(
\prod_{i = 1}^{l} h_{i}(\widetilde{M}_{t_{i}})
\right)
\right] \\
= &\lim\limits_{n \to + \infty}
\esp\left[
\left(
\Psi_{F^{\mathrm{Vol}(E)},f^{E}}(\widetilde{M}_{t+s}) - \Psi_{F_{n}^{\mathrm{Vol}(E)},f^{E}}(\widetilde{M}_{t+s})\right) \times \left(
\prod_{i = 1}^{l} h_{i}(\widetilde{M}_{t_{i}})
\right)
\right] \\
&+ \lim\limits_{n \to + \infty}
\esp\left[
\left(
\Psi_{F_{n}^{\mathrm{Vol}(E)},f_{n}^{E}}(\widetilde{M}_{t+s}) - \Psi_{F_{n}^{\mathrm{Vol}(E)},f^{E}}(\widetilde{M}_{t+s})\right) \times \left(
\prod_{i = 1}^{l} h_{i}(\widetilde{M}_{t_{i}})
\right)
\right] \\
&- \lim\limits_{n \to + \infty}
\esp\left[
\left(
\Psi_{F^{\mathrm{Vol}(E)},f^{E}}(\widetilde{M}_{t}) - \Psi_{F_{n}^{\mathrm{Vol}(E)},f^{E}}(\widetilde{M}_{t})\right) \times \left(
\prod_{i = 1}^{l} h_{i}(\widetilde{M}_{t_{i}})
\right)
\right] \\
&- \lim\limits_{n \to + \infty}
\esp\left[
\left(
\Psi_{F_{n}^{\mathrm{Vol}(E)},f^{E}}(\widetilde{M}_{t}) - \Psi_{F_{n}^{\mathrm{Vol}(E)},f_{n}^{E}}(\widetilde{M}_{t})\right) \times \left(
\prod_{i = 1}^{l} h_{i}(\widetilde{M}_{t_{i}})
\right)
\right] \\
&- \lim\limits_{n \to + \infty}
\esp\left[\left(\int_{t}^{t+s}\lcal_{µ}^{\infty} \Psi_{F^{\mathrm{Vol}(E)},f^{E}}(\widetilde{M}_{u})- \lcal_{µ}^{\infty} \Psi_{F_{n}^{\mathrm{Vol}(E)},f^{E}}(\widetilde{M}_{u})du\right)\times \left(
\prod_{i = 1}^{l} h_{i}(\widetilde{M}_{t_{i}})
\right)
\right] \\
&- \lim\limits_{n \to + \infty}
\esp\left[\left(\int_{t}^{t+s}\lcal_{µ}^{\infty} \Psi_{F_{n}^{\mathrm{Vol}(E)},f^{E}}(\widetilde{M}_{u})- \lcal_{µ}^{\infty} \Psi_{F_{n}^{\mathrm{Vol}(E)},f_{n}^{E}}(\widetilde{M}_{u})du\right)\times \left(
\prod_{i = 1}^{l} h_{i}(\widetilde{M}_{t_{i}})
\right)
\right]
\end{align*}
under the condition that all these limits exist. 

By Eq. (\ref{eqn:cvg_Fnf}), Lemma \ref{lem:cvg_Fnfn} and Lemma \ref{lem:cvg_linf}, all the terms inside the expectations converge to $0$ when $n \to + \infty$. Using the bounds given by Eq. (\ref{eqn:eq_psi_1}), (\ref{eqn:eq_psi_2}), (\ref{eqn:eq_psi_3}) and Lemma \ref{lem:bound_linfty}, we can apply the dominated convergence theorem and obtain the desired result. 

\end{proof}

\subsection{Extended martingale problem for the $\infty$-parent ancestral process}\label{subsec:4.2}
In this section we prove the following result. 
\begin{lem}\label{lem:extension_pb_backwards_ginf}
Let $µ$ be a $\sigma$-finite measure on $(0,+\infty)$ satisfying Condition (\ref{eqn:condition_infty_SLFV}). Let $\Xi^{0} \in \mathcal{M}^{cf}$, and let $(\Xi_{t}^{\infty})_{t \geq 0} = (m(E_{t}^{\infty}))_{t \geq 0}$ be the $\infty$-parent ancestral process associated to $µ$ and with initial condition $\Xi^{0}$. 

Then, for all measurable function $f : \rd \to \{0,1\}$,  
\begin{equation*}
\left(\Phi_{\delta_{0},f}(\Xi_{t}^{\infty}) - \Phi_{\delta_{0},f}(\Xi_{0}^{\infty}) - \int_{0}^{t}\gcal_{µ}^{\infty}\Phi_{\delta_{0},f}(\Xi_{s}^{\infty})ds
\right)_{t \geq 0}
\end{equation*}
is a martingale, where $\Phi_{\delta_{0},f} : M \in \mcal^{cf} \to \Phi_{\delta_{0},f}(\Xi)$ is the function defined by
\begin{equation*}
\forall m(E) \in \mcal^{cf}, \Phi_{\delta_{0},f}(m(E)) := \delta_{0}\left( \int_{E} f(x)dx\right).
\end{equation*}
\end{lem}

\begin{proof}
Let $(\fcal_{t})_{t \geq 0}$ be the filtration generated by $(\Xi_{t}^{\infty})_{t \geq 0}$. 
Let $(F_{n})_{n \in \mathbb{N}^{*}} \in C_{b}^{1}(\mathbb{R})$ be a sequence of functions converging pointwise to $\mathds{1}_{0}$ such that 
\begin{align*}
&\text{(A) } \forall n \in \mathbb{N}^{*}, F_{n} \text{ is increasing on }\mathbb{R}_{-}\text{ and decreasing on } \mathbb{R}_{+}, \\
&\text{(B) } \forall n \in \mathbb{N}^{*}, F_{n}(0) = 1 \text{ and } \forall x \in \mathbb{R}, 0 \leq F_{n}(x) \leq 1, \\
&\text{(C) } \forall n \in \mathbb{N}^{*}, Supp(F_{n}) \subseteq [-n^{-3},n^{-3}].
\end{align*}
The interest of this sequence lies in the fact that for all $n \in \mathbb{N}^{*}$ and for all measurable function $f : \rd \to \{0,1\}$, 
\begin{equation*}
\left(\Phi_{F_{n},f}(\Xi_{t}^{\infty}) - \Phi_{F_{n},f}(\Xi_{0}^{\infty}) - \int_{0}^{t}\gcal_{µ}^{\infty}\Phi_{F_{n},f}(\Xi_{s}^{\infty})ds
\right)_{t \geq 0}
\end{equation*}
is a martingale.  

Let $f : \rd \to \{0,1\}$ be a measurable function, and let $0 \leq s \leq t$. $\Phi_{\delta_{0},f}$ is bounded by $1$, and by Hypothesis (B), the functions $(\Phi_{F_{n},f})_{n \in \mathbb{N}^{*}}$ are bounded by $1$ as well. Moreover, since $u \to \mathrm{Vol}(\Xi_{u}^{\infty}))$ is increasing, and as there exists $C_{t} > 0$ such that for all $\rcal > 0$,
\begin{equation*}
\mathrm{Vol}(S^{\rcal}(\Xi_{t}^{\infty})) \leq C_{t} \, \left(\rcal^{d} \vee 1\right),
\end{equation*}
we can deduce that for all $s \in [0,t]$ and for all $n \in \mathbb{N}^{*}$, by Hypothesis (B), 
\begin{align*}
\left|\gcal_{µ}^{\infty} \Phi_{F_{n},f} (\Xi_{s}^{\infty})\right|
&\leq \int_{0}^{\infty} 2  \, \mathrm{Vol}(S^{\rcal}(\Xi_{s}^{\infty}))µ(d\rcal) \\
&\leq \int_{0}^{\infty} 2  \, \mathrm{Vol}(S^{\rcal}(\Xi_{t}))µ(d\rcal) \\
&\leq 2C_{t}  \, \int_{0}^{\infty} \left(\rcal^{d} \vee 1 \right) µ(d\rcal).
\end{align*}
Similarly, we obtain that
\begin{equation*}
\left|\gcal_{µ}^{\infty} \Phi_{\delta_{0},f} (\Xi_{s}^{\infty})\right|
\leq 2C_{t}  \, \int_{0}^{\infty} \left(\rcal^{d} \vee 1 \right) µ(d\rcal).
\end{equation*}
Since $µ$ satisfies Condition (\ref{eqn:condition_intensite}), both quantities are finite. Therefore, by Fubini's theorem, for all $n \in \mathbb{N}^{*}$, 
\begin{align*}
&\mathbf{E}\left[\left.
\Phi_{\delta_{0}, f}(\Xi_{t}^{\infty}) - \Phi_{\delta_{0},f}(\Xi_{0}^{\infty}) - \int_{0}^{t}\gcal_{µ}^{\infty}\Phi_{\delta_{0},f}(\Xi_{u}^{\infty})du\right|\fcal_{s}
\right] \\
= &\mathbf{E}\left[\left.
\Phi_{\delta_{0},f}(\Xi_{t}^{\infty}) - \Phi_{F_{n},f}(\Xi_{t}^{\infty})
\right|\fcal_{s}\right] + \mathbf{E}\left[\left.
\Phi_{F_{n},f}(\Xi_{0}^{\infty}) - \Phi_{\delta_{0},f}(\Xi_{0}^{\infty})\right|\fcal_{s}
\right] \\
&+ \int_{0}^{t}\mathbf{E}\left[\left.
\gcal_{µ}^{\infty}\Phi_{F_{n},f}(\Xi_{u}^{\infty}) - \gcal_{µ}^{\infty}\Phi_{\delta_{0},f}(\Xi_{u}^{\infty})\right|\fcal_{s}
\right]du \\
&+ \mathbf{E}\left[\left.
\Phi_{F_{n},f}(\Xi_{t}^{\infty}) - \Phi_{F_{n},f}(\Xi_{0}^{\infty}) - \int_{0}^{t}\gcal_{µ}^{\infty}\Phi_{F_{n},f}(\Xi_{u}^{\infty})du\right|\fcal_{s}
\right].
\end{align*}
Using Proposition \ref{prop:backwards_kinf}, we obtain that
\begin{align*}
&\mathbf{E}\left[\left.
\Phi_{\delta_{0}, f}(\Xi_{t}^{\infty}) - \Phi_{\delta_{0},f}(\Xi_{0}^{\infty}) - \int_{0}^{t}\gcal_{µ}^{\infty}\Phi_{\delta_{0},f}(\Xi_{u}^{\infty})du\right|\fcal_{s}
\right] \\
= &\Phi_{F_{n},f}(\Xi_{s}^{\infty}) - \Phi_{F_{n},f}(\Xi_{0}^{\infty})
- \int_{0}^{s} \gcal_{µ}^{\infty} \Phi_{F_{n},f}(\Xi_{u}^{\infty})du\\
& + \mathbf{E}\left[\left.
\Phi_{\delta_{0},f}(\Xi_{t}^{\infty}) - \Phi_{F_{n},f}(\Xi_{t}^{\infty})\right|\fcal_{s}
\right] + \mathbf{E}\left[\left.
\Phi_{F_{n},f}(\Xi_{0}^{\infty}) - \Phi_{\delta_{0},f}(\Xi_{0}^{\infty})
\right|\fcal_{s}\right] \\
&+ \int_{0}^{t}\mathbf{E}\left[\left.
\gcal_{µ}^{\infty}\Phi_{F_{n},f}(\Xi_{u}^{\infty}) - \gcal_{µ}^{\infty}\Phi_{\delta_{0},f}(\Xi_{u}^{\infty})\right|\fcal_{s}
\right]du .
\end{align*}
Since this is true for all $n \in \mathbb{N}^{*}$, 
\begin{align*}
&\mathbf{E}\left[\left.
\Phi_{\delta_{0}, f}(\Xi_{t}^{\infty}) - \Phi_{\delta_{0},f}(\Xi_{0}^{\infty}) - \int_{0}^{t}\gcal_{µ}^{\infty}\Phi_{\delta_{0},f}(\Xi_{u}^{\infty})du\right|\fcal_{s}
\right] \\
= &\lim\limits_{n \to + \infty} \Phi_{F_{n},f}(\Xi_{s}^{\infty}) - \lim\limits_{n \to + \infty} \Phi_{F_{n},f}(\Xi_{0}^{\infty})
- \lim\limits_{n \to + \infty}\int_{0}^{s} \gcal_{µ}^{\infty} \Phi_{F_{n},f}(\Xi_{u}^{\infty})du \\
& + \lim\limits_{n \to + \infty}\mathbf{E}\left[\left.
\Phi_{\delta_{0},f}(\Xi_{t}^{\infty}) - \Phi_{F_{n},f}(\Xi_{t}^{\infty})\right|\fcal_{s}
\right] + \lim\limits_{n \to + \infty}\mathbf{E}\left[\left.
\Phi_{F_{n},f}(\Xi_{0}^{\infty}) - \Phi_{\delta_{0},f}(\Xi_{0}^{\infty})
\right|\fcal_{s}\right] \\
&+ \lim\limits_{n \to + \infty}\int_{0}^{t}\mathbf{E}\left[\left.
\gcal_{µ}^{\infty}\Phi_{F_{n},f}(\Xi_{u}^{\infty}) - \gcal_{µ}^{\infty}\Phi_{\delta_{0},f}(\Xi_{u}^{\infty})\right|\fcal_{s}
\right]du ,
\end{align*}
under the condition that all these limits exist. 

First, since $\Phi_{F_{n},f}$ converges pointwise to $\Phi_{\delta_{0},f}$, 
\begin{equation*}
\lim\limits_{n \to + \infty} \Phi_{F_{n},f} (\Xi_{s}^{\infty}) - \Phi_{F_{n},f}(\Xi_{0}^{\infty})
= \Phi_{\delta_{0},f}(\Xi_{s}^{\infty}) - \Phi_{\delta_{0},f}(\Xi_{0}^{\infty}) \quad \text{ almost surely},
\end{equation*}
and by the dominated convergence theorem, 
\begin{align*}
\lim\limits_{n \to + \infty} \mathbf{E}\left[\left.
\Phi_{\delta_{0},f}(\Xi_{t}^{\infty}) - \Phi_{F_{n},f}(\Xi_{t}^{\infty})
\right|\fcal_{s}\right] =
\lim\limits_{n \to + \infty} \mathbf{E}\left[\left.
\Phi_{\delta_{0},f}(\Xi_{0}^{\infty}) - \Phi_{F_{n},f}(\Xi_{0}^{\infty})
\right|\fcal_{s}\right] = 0.
\end{align*}
Moreover, since for all $n \in \mathbb{N}^{*}$, 
\begin{equation*}
\int_{0}^{s} \left|
\gcal_{µ}^{\infty} \Phi_{F_{n},f} (\Xi_{u}^{\infty})
\right|du \leq 2s  \, C_{t} \,\int_{0}^{\infty} \left( \rcal^{d} \vee 1 \right)µ(d\rcal),
\end{equation*}
again by the dominated convergence theorem, we obtain
\begin{equation*}
\lim\limits_{n \to + \infty} \int_{0}^{s} \gcal_{µ}^{\infty} \Phi_{F_{n},f} (\Xi_{u}^{\infty})du
= \int_{0}^{s} \gcal_{µ}^{\infty} \Phi_{\delta_{0},f}(\Xi_{u}^{\infty})du. 
\end{equation*}

Then, let $n \in \mathbb{N}^{*}$. Recalling that $\Xi_{u}^{\infty}$ is also denoted $m(E_{u}^{\infty})$, 
\begin{align*}
&\int_{0}^{t} \egras \left[\left.\gcal_{µ}^{\infty}\Phi_{F_{n},f}(\Xi_{u}^{\infty}) - \gcal_{µ}^{\infty}\Phi_{\delta_{0},f} (\Xi_{u}^{\infty})\right|\fcal_{s}\right]du \\
= &\int_{0}^{t} \egras \left[\left.
\int_{0}^{\infty}\int_{S^{\rcal}(E_{u}^{\infty})} \left(
F_{n}\left(\langle m(E_{u}^{\infty} \cup \bxr),f \rangle \right) 
- \delta_{0} \left(\langle m(E_{u}^{\infty} \cup \bxr),f \rangle \right)
\right)
dxµ(d\rcal)\right|\fcal_{s}\right]du\\
&+ \int_{0}^{t} \egras \left[\left.
\int_{0}^{\infty}\int_{S^{\rcal}(E_{u}^{\infty})} \left(\delta_{0} \left(\langle m(E_{u}^{\infty}),f \rangle \right)
-F_{n}\left(\langle m(E_{u}^{\infty}),f \rangle \right) 
\right)
dxµ(d\rcal)\right|\fcal_{s}\right]du.
\end{align*}
Since for all $x \in \rd$, $u \in [0,t]$ and $\rcal > 0$, 
\begin{align*}
\lim\limits_{n \to + \infty} F_{n}\left(\langle m(E_{u}^{\infty} \cup \bxr),f \rangle\right) &= \delta_{0}\left(\langle m(E_{u}^{\infty} \cup \bxr),f \rangle\right) \\
\text{and \quad \quad \quad\quad} \lim\limits_{n \to + \infty} F_{n}\left(\langle m(E_{u}^{\infty}),f \rangle\right) &= \delta_{0}\left(\langle m(E_{u}^{\infty}),f \rangle\right),
\end{align*}
using the dominated convergence theorem, we obtain that
\begin{equation*}
\lim\limits_{n \to + \infty} \int_{0}^{t} \egras\left[\left.
\gcal_{µ}^{\infty} \Phi_{F_{n},f}(\Xi_{u}^{\infty}) - \gcal_{µ}^{\infty} \Phi_{\delta_{0},f}(\Xi_{u}^{\infty})
\right|\fcal_{s}\right] = 0,
\end{equation*}
and we can conclude that
\begin{align*}
&\egras\left[\left.
\Phi_{\delta_{0},f}(\Xi_{t}^{\infty}) - \Phi_{\delta_{0},f}(\Xi_{0}^{\infty}(\Xi_{0}^{\infty}) 
- \int_{0}^{t} \gcal_{µ}^{\infty} \Phi_{\delta_{0},f}(\Xi_{u}^{\infty})du
\right|\fcal_{s}\right] \\
= &\Phi_{\delta_{0},f}(\Xi_{s}^{\infty}) - \Phi_{\delta_{0},f}(\Xi_{0}^{\infty}(\Xi_{0}^{\infty}) 
- \int_{0}^{s} \gcal_{µ}^{\infty} \Phi_{\delta_{0},f}(\Xi_{u}^{\infty})du.
\end{align*}

\end{proof}

\subsection{Uniqueness of the solution to the martingale problem characterizing the $\infty$-parent SLFV}\label{subsec:4.3}
We now use the extended martingale problem in order to prove Proposition \ref{prop:dualite_kinf}, i.e, that the $\infty$-parent ancestral process is the dual of the $\infty$-parent SLFV. 

\begin{proof}(Proposition \ref{prop:dualite_kinf}) 
In order to ease notation, for all $t \geq 0$, we will denote as $\omega_{t}$ the density of $M_{t}^{\infty}$ we will be considering, and
let $(E_{t})_{t \geq 0}$ be such that 
\begin{equation*}
(\Xi_{t}^{\infty})_{t \geq 0} = (m(E_{t}))_{t \geq 0}.
\end{equation*}  

For all $s, t \geq 0$, we set :
\begin{align*}
F(s,t) &= \esp_{M^{0}}\left[\mathbf{E}_{m(E^{0})}\left[
D(M_{s}^{\infty}, \Xi_{t}^{\infty})
\right]\right]. \\
\intertext{Then, } 
F(s,t) &= \esp_{M^{0}}\left[\mathbf{E}_{m(E^{0})}\left[
\Phi_{\delta_{0},1-\omega_{s}}(\Xi_{t}^{\infty})
\right]\right] \\
\intertext{and by Lemma \ref{lem:extension_pb_backwards_ginf},}
F(s,t) &= \esp_{M^{0}}\left[\mathbf{E}_{m(E^{0})}\left[
\Phi_{\delta_{0},1-\omega_{s}}(\Xi_{0})\right]\right]
+ \esp_{M^{0}}\left[\mathbf{E}_{m(E^{0})}\left[
\int_{0}^{t}\gcal_{µ}^{\infty} \Phi_{\delta_{0},1 - \omega_{s}}(\Xi_{u}^{\infty})du \right]\right].\\
\intertext{By Fubini's theorem, we obtain }
F(s,t) &= F(s,0) + \int_{0}^{t}\esp_{M^{0}}\left[\mathbf{E}_{m(E^{0})}\left[\gcal_{µ}^{\infty}\Phi_{\delta_{0},1-\omega_{s}}(\Xi_{u}^{\infty})\right]\right]du. \\
\intertext{Then, } 
F(s,t) &= \mathbf{E}_{m(E^{0})}\left[\esp_{M^{0}}\left[\widetilde{D}(M_{s}^{\infty},\Xi_{t}^{\infty})\right]\right] \\
&= \mathbf{E}_{m(E^{0})}\left[\esp_{M^{0}}\left[\Psi_{F^{\mathrm{Vol}(E_{t})},f^{E_{t}}}(M_{s}^{\infty})\right]\right], \\
\intertext{and by Lemma \ref{lem:extension_pb_forwards_linf},}
F(s,t) &= \mathbf{E}_{m(E^{0})}\left[\esp_{M^{0}}\left[
\Psi_{F^{\mathrm{Vol}(E_{t})},f^{E_{t}}}(M_{0}^{\infty})\right]\right] + \mathbf{E}_{m(E^{0})}\left[\esp_{M^{0}}\left[
\int_{0}^{s}\lcal_{µ}^{\infty}\Psi_{F^{\mathrm{Vol}(E_{t})},f^{E_{t}}}(M_{u}^{\infty})du
\right]\right] \\
&= F(0,t) + \mathbf{E}_{m(E^{0})}\left[\esp_{M^{0}}\left[
\int_{0}^{s} \lcal_{µ}^{\infty}\Psi_{F^{\mathrm{Vol}(E_{t})},f^{E_{t}}}(M_{u}^{\infty})du
\right]\right]. \\
\intertext{Again by Fubini's theorem, we obtain }
F(s,t) &= F(0,t) + \int_{0}^{s}\esp_{M^{0}}\left[\mathbf{E}_{m(E^{0})}\left[ \lcal_{µ}^{\infty} \Psi_{F^{\mathrm{Vol}(E_{t})},f^{E_{t}}}(M_{u}^{\infty})\right]\right]du.
\end{align*}
Combining both expressions for $F(s,t)$, by Lemma 4.4.10 in \cite{ethier1986}, we obtain :
\begin{align*}
&F(t,0)-F(0,t) \\
= &\int_{0}^{t} \left(\esp_{M^{0}}\left[\mathbf{E}_{m(E^{0})}\left[\lcal_{µ}^{\infty} \Psi_{F^{\mathrm{Vol}(E_{t-u})},f^{E_{t-u}}}(M_{u}^{\infty})\right]\right]
- \esp_{M^{0}}\left[\mathbf{E}_{m(E^{0})}\left[\gcal_{µ}^{\infty}\Phi_{\delta_{0},1-\omega_{u}}(\Xi_{t-u}^{\infty})\right]\right]\right)du.
\end{align*}

Let $u \in [0,t]$. We have
\begin{align*}
&\gcal_{µ}^{\infty}\Phi_{\delta_{0},1-\omega_{u}}(\Xi_{t-u}^{\infty})\\
= &\int_{0}^{\infty}\int_{S^{\rcal}(E_{t-u})} \left(
\delta_{0}\left(\int_{E_{t-u} \cup \bxr} \left(1 - \omega_{u}(z)\right)dz\right) - \delta_{0}\left(\int_{E_{t-u}} \left(1 - \omega_{u}(z)\right)dz\right)
\right)dxµ(d\rcal)
\end{align*}
and
\begin{align*}
\lcal_{µ}^{\infty} \Psi_{F^{\mathrm{Vol}(E_{t-u})},f^{E_{t-u}}}(M_{u}^{\infty}) 
= &\int_{0}^{\infty}\int_{S^{\rcal}(E_{t-u})} \left(1 - \delta_{0}\left(\int_{\bxr} \left(1 - \omega_{u}(z)\right)dz\right)\right) \\
&\times \left[\delta_{0}\left(
\mathrm{Vol}(E_{t-u}) - \langle\Theta_{x,\rcal}^{-}(\omega_{u}),\mathds{1}_{E_{t-u}}\rangle
\right) 
- \delta_{0}\left(\mathrm{Vol}(E_{t-u}) - 
\langle \omega_{u},\mathds{1}_{E_{t-u}} \rangle
\right)\right]dx \\
= &\int_{0}^{\infty}\int_{S^{\rcal}(E_{t-u})} \left(1 - \delta_{0}\left(\int_{\bxr} \left(1 - \omega_{u}(z)\right)dz\right)\right) \\
&\times \left[\delta_{0}\left(
\mathrm{Vol}(E_{t-u}) - \langle\Theta_{x,\rcal}^{-}(\omega_{u}),\mathds{1}_{E_{t-u}}\rangle
\right) 
- \delta_{0}\left(
\int_{E_{t-u}} \left(1 - \omega_{u}(z)\right)dz
\right)\right]dx.
\end{align*}
For all $\rcal > 0$ and $x \in S^{\rcal}(E_{t-u})$, 
\begin{align*}
\delta_{0}\left(
\mathrm{Vol}(E_{t-u}) - \langle\Theta_{x,\rcal}^{-}(\omega_{u}),\mathds{1}_{E_{t-u}}\rangle
\right)
&= \delta_{0}\left(
\mathrm{Vol}(E_{t-u}) - \int_{E_{t-u} \backslash \bxr} \omega_{u}(z)dz
\right) \\
&= \delta_{0}\left(\mathrm{Vol}(E_{t-u} \cap \bxr) + \int_{E_{t-u} \backslash \bxr} \left(1 - \omega_{u}(z)\right)dz\right). \\
\intertext{Since $x \in S^{\rcal}(E_{t-u})$, $\mathrm{Vol}(E_{t-u} \cap \bxr) \neq 0$, and hence }
\delta_{0}\left(
\mathrm{Vol}(E_{t-u}) - \langle\Theta_{x,\rcal}^{-}(\omega_{u}),\mathds{1}_{E_{t-u}}\rangle
\right)
&= 0. \\
\intertext{Moreover, notice that}
\delta_{0}\left(\int_{E_{t-u} \cup \bxr} \left(1 - \omega_{u}(z)\right)dz\right)
&= \delta_{0}\left(\int_{E_{t-u}} \left(1 - \omega_{u}(z)\right)dz\right)
 \, \delta_{0}\left(\int_{\bxr} \left(1 - \omega_{u}(z)\right)dz\right).
\end{align*}
Therefore,
\begin{align*}
&\lcal_{µ}^{\infty} \Psi_{F^{\mathrm{Vol}(E_{t-u})},f^{E_{t-u}}} (M_{u}^{\infty}) \\
= &\int_{0}^{\infty}\int_{S^{\rcal}(E_{t-u})} 
\delta_{0}\left(\int_{\bxr} \left(1 - \omega_{u}(z)\right)dz\right)  \, \delta_{0}\left(\int_{E_{t-u}} \left(1 - \omega_{u}(z)\right)dz\right)dx µ(d\rcal)\\
&-\int_{0}^{\infty}\int_{S^{\rcal}(E_{t-u})} \delta_{0}\left(\int_{E_{t-u}} \left(1 - \omega_{u}(z)\right)dz\right)dx µ(d\rcal)\\
= &\int_{0}^{\infty}\int_{\rd} \mathds{1}_{x \in S^{\rcal}(E_{t-u})}  \, \left[
\delta_{0}\left(\int_{E_{t-u} \cup \bxr} \left(1 - \omega_{u}(z)\right)dz\right) - \delta_{0}\left(\int_{E_{t-u}} \left(1 - \omega_{u}(z)\right)dz\right)
\right]dx µ(d\rcal) ,
\end{align*}
which is equal to $\gcal_{µ}^{\infty} \Phi_{\delta_{0},1-\omega_{u}}(\Xi_{t-u}^{\infty})$.
Thus
\begin{align*}
F(t,0) &= F(0,t) \\
\text{i.e }\quad \quad \quad \quad \quad \quad \quad \quad \esp_{M^{0}}\left[\mathbf{E}_{m(E^{0})}\left[\widetilde{D}(M_{t}^{\infty},\Xi_{0}^{\infty})\right]\right] &= 
\esp_{M^{0}}\left[\mathbf{E}_{m(E^{0})}\left[\widetilde{D}(M_{0}^{\infty},\Xi_{t}^{\infty})\right]\right] \\
\Longleftrightarrow \text{\quad \quad } \esp_{M^{0}}\left[\mathbf{E}_{m(E^{0})}\left[\delta_{0}\left(\int_{E_{0}} \left(1 - \omega_{t}(x)\right)dx\right)\right]\right]
&= \esp_{M^{0}}\left[\mathbf{E}_{m(E_{t})}\left[\delta_{0}\left(\int_{E_{0}} \left(1 - \omega_{0}(x)\right)dx\right)\right]\right]. \\
\intertext{Therefore}
\esp_{M^{0}}\left[
\delta_{0}\left(\int_{E_{0}} \left(1 - \omega_{t}(x)\right)dx
\right)\right]
&= \mathbf{E}_{m(E^{0})}\left[
\delta_{0}\left(\int_{E_{t}} \left(1 - \omega_{0}(x)\right)dx
\right)\right]
\end{align*}
and we can conclude.
\end{proof}

Finally, we can prove the second part of Theorem \ref{thm:characterization_infty_SLFV}, i.e, the uniqueness of the solution to the martingale problem satisfied by the $\infty$-parent SLFV when $µ$ satisfies Condition (\ref{eqn:condition_infty_SLFV}). The first part of this theorem was proved in Section \ref{sec:3} (Proposition \ref{prop:char_infty_slfv}).
\begin{proof}(Theorem \ref{thm:characterization_infty_SLFV}) 

Let $(M_{t}^{a})_{t \geq 0}$ and $(M_{t}^{b})_{t \geq 0}$ be two solutions to the martingale problem $(\lcal_{µ}^{\infty},\delta_{M^{0}})$. Then, due to the form of the operator $\lcal_{µ}^{\infty}$, there exists densities $(\omega_{t}^{a})_{t \geq 0}$ and $(\omega_{t}^{b})_{t \geq 0}$ of $(M_{t}^{a})_{t \geq 0}$ and $(M_{t}^{b})_{t \geq 0}$ such that
\begin{equation*}
\forall t \geq 0, \forall x \in \rd, \omega_{t}^{a}(x) \in \{0,1\} \text{ and } \omega_{t}^{b}(x) \in \{0,1\}.
\end{equation*}
Then, let $t \geq 0$, let $E \in \mathcal{E}^{cf}$ and let $(\Xi_{t}^{\infty})_{t \geq 0}$ be the $\infty$-parent ancestral process started from $m(E)$. We have 
\begin{align*}
\proba_{M^{0}}\left(\delta_{0}\left(\int_{E} \left(1 - \omega_{t}^{a}(x)\right)dx\right) = 1\right) 
&= \esp_{M^{0}}\left[\delta_{0}\left(\int_{E} \left(1 - \omega_{t}^{a}(x)\right)dx\right)\right] \\
&= \esp_{M^{0}}\left[D(M_{t}^{a},\Xi_{0}^{\infty})\right] \\
&= \mathbf{E}_{m(E)}\left[D(M^{0}, \Xi_{t}^{\infty})\right] 
\text{ by Proposition \ref{prop:dualite_kinf} } \\
&= \esp_{M^{0}}\left[D(M_{t}^{b},\Xi_{0}^{\infty})\right] 
\text{ by the same proposition} \\
&= \esp_{M^{0}}\left[\delta_{0}\left(\int_{E}\left( 1 - \omega_{t}^{b}(x)\right)dx\right)\right] \\
&= \proba_{M^{0}}\left(\delta_{0}\left(\int_{E}\left( 1 - \omega_{t}^{b}(x)\right)dx\right) = 1\right),
\end{align*}
using Proposition \ref{prop:dualite_kinf} to pass from line $2$ to line $3$, and from line $3$ to line $4$. We can conclude that $(M_{t}^{a})_{t \geq 0}$ and $(M_{t}^{b})_{t \geq 0}$ have the same distribution.
\end{proof}

\section{Technical lemmas}\label{sec:5}

\subsection{Properties of the operators $\lcal_{µ}^{k}$ and $\lcal_{µ}^{\infty}$}
The goal of this section is to show that the operators $\lcal_{µ}^{k}$ and $\lcal_{µ}^{\infty}$ introduced in Section \ref{sec:nouvelle_section_2} are well-defined, as well as to prove some properties they satisfy. 

In all that follows, let $F \in C^{1}(\mathbb{R})$, $f \in C_{c}(\rd)$, and $M \in \mcal_{\lambda}$. 
Let $\omega : \rd \to \{0,1\}$ be a measurable function, let $µ$ be a $\sigma$-finite measure on $\mathbb{R}_{+}^{*}$ satisfying Condition (\ref{eqn:condition_intensite}), and let $k \geq 2$. Since $f$ is of compact support, there exist constants $C_{1}, C_{2} > 0$ such that for all $\rcal > 0$, 
\begin{equation}\label{eqn:eqn_1}
\mathrm{Vol}(Supp^{\rcal}(f)) \leq C_{2}  \, \left( \rcal^{d} \vee 1\right),
\end{equation}
and for all $\tilde{\omega} : \rd \to \{0,1\}$ measurable,
\begin{equation}\label{eqn:eqn_2}
\left| \langle \mathds{1}_{\bxr} \, \tilde{\omega}, f \rangle \right|
\leq C_{1}  \, ||f||_{\infty} \, \left( \rcal^{d} \wedge 1 \right).
\end{equation}

\begin{lem}\label{lem:appendix_A_1}
For all $x \in \rd$ and for all $\rcal > 0$,
\begin{align*}
\left|\langle \Theta_{x,\rcal}^{+}(\omega),f \rangle - \langle \omega,f \rangle \right|
&\leq ||f||_{\infty}  \,\mathrm{Vol}(Supp(f)) \\
\text{and \quad} 
\left|\langle \Theta_{x,\rcal}^{-}(\omega),f \rangle - \langle \omega,f \rangle \right|
&\leq ||f||_{\infty}  \, \mathrm{Vol}(Supp(f)).
\end{align*}
\end{lem}

\begin{proof}
Let $x \in \rd$ and $\rcal > 0$. 
\begin{align*}
\left|\langle \Theta_{x,\rcal}^{+}(\omega),f \rangle - \langle \omega,f \rangle\right| 
\leq &\left|
\langle \mathds{1}_{\bxr^{c}}  \,\omega, f \rangle
+ \langle \mathds{1}_{\bxr},f \rangle 
- \langle \mathds{1}_{\bxr^{c}}  \, \omega, f \rangle
- \langle \mathds{1}_{\bxr}  \, \omega, f \rangle \right| \\
\leq &\left|
\langle \mathds{1}_{\bxr}  \, \left(1-\omega \right), f \rangle \right| \\
\leq & \left|
\int_{\bxr} \left(1-\omega(y)\right) \, f(y) dy \right| \\
\leq & \int_{\bxr} |f(y)|dy \\
\leq & ||f||_{\infty} \times \mathrm{Vol}(Supp(f)).
\end{align*}

We can similarly show the corresponding result for
$
\left|\langle \Theta_{x,\rcal}^{-}(\omega),f \rangle - \langle \omega,f \rangle\right|.
$
\end{proof}

\begin{lem}\label{lem:appendix_A_3}
For all $\rcal > 0$, for all $x \in \rd \backslash Supp^{\rcal}(f)$, 
\begin{equation*}
\langle \Theta_{x,\rcal}^{+}(\omega),f \rangle - \langle \omega,f \rangle
= \langle \Theta_{x,\rcal}^{-}(\omega),f \rangle - \langle \omega,f \rangle
= 0
\end{equation*}
\end{lem}

\begin{proof}
Let $\rcal > 0$, and let $x \in \rd \backslash Supp^{\rcal}(f)$, 
\begin{align*}
\left|\langle \Theta_{x,\rcal}^{+}(\omega),f \rangle - \langle \omega,f \rangle\right|
= &\left|\langle \mathds{1}_{\bxr} \, (1-\omega),f \rangle \right| \\
\leq & \int_{\bxr} |f(y)|dy \\
= & 0
\end{align*}
since $x \in \rd \backslash Supp^{\rcal}(f)$. Similarly,
\begin{align*}
\left|\langle \Theta_{x,\rcal}^{-}(\omega),f \rangle - \langle \omega,f \rangle\right|
= &\left|\langle \mathds{1}_{\bxr} \, \omega,f \rangle \right| \\
\leq & \int_{\bxr} |f(y)|dy \\
= & 0
\end{align*}
for the same reason, and we can conclude. 
\end{proof}

\begin{lem}\label{lem:appendix_A_2}
For all $x \in \rd$ and for all $\rcal > 0$, 
\begin{align*}
\left|F\left(\langle \Theta_{x,\rcal}^{+}(\omega),f \rangle \right) - F\left(\langle \omega,f \rangle \right) \right|
\leq &\quad C_{1}  \, ||f||_{\infty}  \,\left( \rcal^{d} \wedge 1 \right)  \, C(F,f) \\
\text{and \quad}
\left|F\left(\langle \Theta_{x,\rcal}^{-}(\omega),f \rangle \right) - F\left(\langle \omega,f \rangle \right) \right|
\leq &\quad C_{1}  \,||f||_{\infty}  \, \left( \rcal^{d} \wedge 1 \right) \, C(F,f)
\end{align*}
where
\begin{equation*}
C(F,f) = \sup_{z \in [-||f||_{\infty}\mathrm{Vol}(Supp(f)),||f||_{\infty}\mathrm{Vol}(Supp(f))]} \left|F'(z)\right|.
\end{equation*}
\end{lem}

\begin{proof}
Let $x \in \rd$ and $\rcal > 0$. First, we notice that as in the proof of Lemma \ref{lem:appendix_A_1}, we only need to show the result for $\Theta_{x,\rcal}^{+}(\omega)$. 

By Taylor-Lagrange inequality and by Lemma \ref{lem:appendix_A_1}, 
\begin{align*}
\left|F\left(\langle \Theta_{x,\rcal}^{+}(\omega),f \rangle \right) - F\left(\langle \omega,f \rangle \right) \right|
\leq & \quad \left| \langle \Theta_{x,\rcal}^{+}(\omega),f \rangle - \langle \omega,f \rangle \right|  \,  C(F,f) \\
\leq & \quad \left|\langle \mathds{1}_{\bxr} \times \left(1 - \omega\right),f \rangle \right|  \, C(F,f) \\
\leq & \quad C_{1}  \, ||f||_{\infty} \, \left( \rcal^{d} \wedge 1 \right)  \, C(F,f)
\end{align*}
by Eq. (\ref{eqn:eqn_2}).
\end{proof}

We can now show that the operator $\lcal_{µ}^{k}$ is well-defined. 
\begin{lem}\label{lem:Lk_well_defined}
The operator $\lcal_{µ}^{k}$ is well-defined. Moreover, the function $\lcal_{µ}^{k} \Psi_{F,f} : \mcal_{\lambda} \to \mathbb{R}$ is bounded.
\end{lem}

\begin{proof}
Let $M \in \mcal_{\lambda}$. Then
\begin{align*}
& \left|\lcal_{µ}^{k} \Psi_{F,f}(M)\right| \\
\leq & \left|\int_{\rd}\int_{0}^{\infty}\int_{\bxr^{k}} 
\frac{1}{V_{\rcal}^{k}}  \,\left(\prod_{j = 1}^{k} \omega_{M}(y_{j})\right)  \, \left(
F\left(\langle \Theta_{x,\rcal}^{+}(\omega_{M}),f \rangle \right) - F\left(\langle \omega_{M},f\rangle \right)
\right)dy_{1}...dy_{k}µ(d\rcal)dx\right|\\
 & +  \left|\int_{\rd}\int_{0}^{\infty}\int_{\bxr^{k}} 
\frac{1}{V_{\rcal}^{k}}  \, \left(1-\prod_{j = 1}^{k} \omega_{M}(y_{j})\right) \, \left(
F\left(\langle \Theta_{x,\rcal}^{-}(\omega_{M}),f \rangle \right) - F\left(\langle \omega_{M},f\rangle \right)
\right)dy_{1}...dy_{k}µ(d\rcal)dx\right|\\
\leq & \left|\int_{\rd}\int_{0}^{\infty}\int_{\bxr^{k}} 
\frac{1}{V_{\rcal}^{k}}  \, \left(
F\left(\langle \Theta_{x,\rcal}^{+}(\omega_{M}),f \rangle \right) - F\left(\langle \omega_{M},f\rangle \right)
\right)dy_{1}...dy_{k}µ(d\rcal)dx\right|\\
 & +  \left|\int_{\rd}\int_{0}^{\infty}\int_{\bxr^{k}} 
\frac{1}{V_{\rcal}^{k}} \, \left(
F\left(\langle \Theta_{x,\rcal}^{-}(\omega_{M}),f \rangle \right) - F\left(\langle \omega_{M},f\rangle \right)
\right)dy_{1}...dy_{k}µ(d\rcal)dx\right|\\
\leq & \left|\int_{0}^{\infty}\int_{Supp^{\rcal}(f)}\int_{\bxr^{k}} 
\frac{1}{V_{\rcal}^{k}}  \, \left(
F\left(\langle \Theta_{x,\rcal}^{+}(\omega_{M}),f \rangle \right) - F\left(\langle \omega_{M},f\rangle \right)
\right)dy_{1}...dy_{k}dx µ(d\rcal)\right|\\
 & +  \left|\int_{0}^{\infty}\int_{Supp^{\rcal}(f)}\int_{\bxr^{k}} 
\frac{1}{V_{\rcal}^{k}}  \, \left(
F\left(\langle \Theta_{x,\rcal}^{-}(\omega_{M}),f \rangle \right) - F\left(\langle \omega_{M},f\rangle \right)
\right)dy_{1}...dy_{k}dx µ(d\rcal)\right|.\\
\intertext{Using Lemma \ref{lem:appendix_A_2},} 
& \left|\lcal_{µ}^{k} \Psi_{F,f}(M)\right| \\
\leq & \int_{0}^{\infty}\int_{Supp^{\rcal}(f)}\int_{\bxr^{k}}
\frac{2}{V_{\rcal}^{k}}  \, C_{1}  \, ||f||_{\infty}  \, \left(\rcal^{d} \wedge 1 \right)  \, C(F,f) dy_{1}...dy_{k}dxµ(d\rcal). \\
\leq & \int_{0}^{\infty} 2 \mathrm{Vol}(Supp^{\rcal}(f)) \, C_{1}  \, ||f||_{\infty}  \, \left(\rcal^{d} \wedge 1 \right)  \,C(F,f) µ(d\rcal), \\
\intertext{and by Eq. (\ref{eqn:eqn_1}),}
& \left|\lcal_{µ}^{k} \Psi_{F,f}(M)\right| \\
\leq & \int_{0}^{\infty} 2 C_{1}C_{2}  \, ||f||_{\infty} \,C(F,f)  \, \left(\rcal^{d} \wedge 1 \right)  \,  \left(\rcal ^{d} \vee 1 \right) µ(d\rcal)  \\
\leq & 2 C_{1} C_{2} ||f||_{\infty}  \, C(F,f)  \, \int_{0}^{\infty} \rcal^{d} µ(d\rcal) \\
< & + \infty
\end{align*}
since $µ$ satisfies Condition (\ref{eqn:condition_intensite}). 

The second part of the lemma is a direct consequence of the fact that
\begin{equation*}
2 C_{1} C_{2} ||f||_{\infty}  \, C(F,f)  \, \int_{0}^{\infty} \rcal^{d} µ(d\rcal)
\end{equation*}
does not depend on the choice of $M$.
\end{proof}

A consequence of this lemma and of Lemma \ref{lem:appendix_A_3} is that for all $M \in \mcal_{\lambda}$, $\lcal_{µ}^{k} \Psi_{F,f}(M)$ can be rewritten as :
\begin{align*}
\lcal_{µ}^{k} \Psi_{F,f}(M) = \int_{0}^{\infty}\int_{Supp^{\rcal}(f)}\int_{\bxr^{k}} \frac{1}{V_{\rcal}^{k}} \, 
 & \left[\prod_{j = 1}^{k} \omega_{M}(y_{j})  \, F(\langle\Theta_{x,\rcal}^{+}(\omega_{M}),f\rangle)\right. \\ &+ (1-\prod_{j = 1}^{k}\omega_{M}(y_{j}))  \, F(\langle\Theta_{x,\rcal}^{-}(\omega_{M}),f\rangle) \\
&\left. \vphantom{\prod_{j = 1}^{k}} - F(\langle\omega_{M},f\rangle)
\right]
dy_{1}...dy_{k}dx µ(d\rcal).
\end{align*}

We now prove that the operator $\lcal_{µ}^{\infty}$ is well-defined.

\begin{lem}\label{lem:Linfty_well_defined}
The operator $\lcal_{µ}^{\infty}$ is well-defined. Moreover, the function $\lcal_{µ}^{\infty} \Psi_{F,f} : \mcal_{\lambda} \to \mathbb{R}$ is bounded.
\end{lem}

\begin{proof}
Let $M \in \mcal_{\lambda}$. Then,
\begin{align*}
&\left|\lcal_{µ}^{\infty} \Psi_{F,f}(M) \right| \\
\leq & \int_{0}^{\infty} \int_{Supp^{\rcal}(f)} \left|
\left(1 - \delta_{0}\left(\int_{\bxr} 1 - \omega_{M}(z)dz\right)\right)  \, \left[F\left(\langle \Theta_{x,\rcal}^{-}(\omega_{M}),f \rangle \right) - F\left(\langle \omega_{M},f \rangle \right)\right]
\right|dxµ(d\rcal) \\
\leq & \int_{0}^{\infty} \int_{Supp^{\rcal}(f)} C_{1}  \, ||f||_{\infty}  \, \left( \rcal^{d} \wedge 1 \right)  \, C(F,f) dx µ(d\rcal) \\
\leq & \int_{0}^{\infty} Vol(Supp^{\rcal}(f)) C_{1}  \, ||f||_{\infty}  \, \left( \rcal^{d} \wedge 1 \right)  \, C(F,f) dx µ(d\rcal) \\
\leq & C_{1}C_{2}C(F,f)  \, ||f||_{\infty}  \, \int_{0}^{\infty} \left(\rcal^{d} \wedge 1 \right)  \, \left(\rcal^{d} \vee 1 \right) µ(d\rcal)  \\
\leq & C_{1}C_{2}C(F,f)  \, ||f||_{\infty}  \, \int_{0}^{\infty} \rcal^{d}µ(d\rcal) \\
< & + \infty
\end{align*}
since $µ$ satisfies Condition (\ref{eqn:condition_intensite}). Here we used Lemma \ref{lem:appendix_A_2} to pass from the second to the third line, and Lemma \ref{lem:appendix_A_1} to pass from the fourth to the fifth line.

As before, the second part of the lemma is the consequence of the fact that 
\begin{equation*}
C_{1}C_{2}C(F,f)  \, ||f||_{\infty}  \, \int_{0}^{\infty} \rcal^{d}µ(d\rcal)
\end{equation*}
does not depend on the choice of $M$.
\end{proof}

\subsection{Properties of the operator $\gcal_{µ}^{\infty}$}
In all the following, let $µ$ be a $\sigma$-finite measure on $\mathbb{R}_{+}^{*}$ satisfying Condition (\ref{eqn:condition_infty_SLFV}), let $F \in C_{b}^{1}(\mathbb{R})$ and let $f \in \mathcal{B}(\rd)$. 

\begin{lem}\label{lem:appendix_B_1}
The operator $\gcal_{µ}^{\infty}$ introduced at the end of Section~\ref{subsec:3.1} is well-defined, and the function $\gcal_{µ}^{\infty} \Phi_{F,f}$ is bounded.
\end{lem}

\begin{proof}
Let $m(E) \in \mcal^{cf}$. Then,
\begin{align*}
 \left|\gcal_{µ}^{\infty} \Phi_{F,f}(m(E)) \right| 
\leq & \int_{0}^{\infty}\int_{S^{\rcal}(E) \cap Supp^{\rcal}(f)} 
\left|F\left(\langle m(E \cup \bxr),f \rangle\right) - F\left(\langle m(E),f \rangle\right)
\right|dx µ(d\rcal) \\
\leq & \int_{0}^{\infty} \int_{S^{\rcal}(E) \cap Supp^{\rcal}(f)} 2 ||F||_{\infty} dx µ(d\rcal) \\
\leq & 2 ||F||_{\infty} \int_{0}^{\infty} \mathrm{Vol}(S^{\rcal}(E) \cap Supp^{\rcal}(f))µ(d\rcal) \\
\leq & 2 ||F||_{\infty} \int_{0}^{\infty} \mathrm{Vol}(Supp^{\rcal}(f))µ(d\rcal) \\
\leq & 2 ||F||_{\infty} \int_{0}^{\infty} C_{2}  \, \left( \rcal^{d} \vee 1 \right) µ(d\rcal) \\
< & + \infty,
\end{align*}
since $µ$ satisfies Condition (\ref{eqn:condition_intensite}).
\end{proof}

\begin{lem}\label{lem:appendix_B_2}
Let $\Xi \in \mcal^{cf}$, and let $(\Xi_{t})_{t \geq 0}$ be the $\infty$-parent ancestral process associated to $µ$ with initial condition $\Xi$. Then, for all $t \geq 0$, 
\begin{equation*}
\egras \left[
\int_{0}^{t} \gcal_{µ}^{\infty} \Phi_{F,f}(\Xi_{s})ds
\right] = \int_{0}^{t} \egras \left[
\gcal_{µ}^{\infty} \Phi_{F,f}(\Xi_{s})
\right]ds.
\end{equation*}
\end{lem}

\begin{proof}
In the proof of Lemma~\ref{lem:appendix_B_1}, we showed that for all $u \geq 0$, 
\begin{equation*}
\left|
\gcal_{µ}^{\infty} \Psi_{F,f}(\Xi_{u})
\right| \leq 2 ||F||_{\infty} \, \int_{0}^{\infty} C_{2} \left(
\rcal^{d} \vee 1
\right) µ(d\rcal).
\end{equation*}
Therefore, 
\begin{align*}
\esp\left[
\int_{0}^{t} \left|
\gcal_{µ}^{\infty} \Psi_{F,f}(\Xi_{s})
\right|ds
\right]
&\leq 2 ||F||_{\infty} \,t \int_{0}^{\infty} C_{2} \left(
\rcal^{d} \vee 1
\right) µ(d\rcal) \\
&< + \infty. 
\end{align*}
We conclude by applying Fubini's theorem. 
\end{proof}

\begin{lem}\label{lem:appendix_B_3}
Let $\Xi \in \mcal^{cf}$, and let $(\Xi_{t})_{t \geq 0}$ be the $\infty$-parent ancestral process associated to $µ$ with initial condition $\Xi$. Then, for all $t \geq 0$, 
\begin{equation*}
\egras \left[\gcal_{µ}^{\infty} \Phi_{F,f}(\Xi_{t}) \right]
= \left.\frac{d}{du}\egras
\left[\Phi_{F,f}(\Xi_{u})\right]
\right|_{u = t}.
\end{equation*}
\end{lem}

\begin{proof}
Let $t \geq 0$. We have
\begin{align*}
\left.\frac{d}{du}\egras \left[\Phi_{F,f}(\Xi_{t})\right]\right|_{t = 0}
= &\gcal_{µ}^{\infty}\Phi_{F,f}(\Xi), \\
\intertext{so for all $s \in [0,t]$,}
\egras\left[\gcal_{µ}^{\infty} \Phi_{F,f}(\Xi_{s}) \right]
= &\egras \left.\left[
\frac{d}{du}\egras\left[ \Phi_{F,f}(\Xi_{t})|\Xi_{s}
\right]\right|_{t = s}\right].\\
\intertext{Since $F'$ is bounded, by the dominated convergence theorem,}
\egras\left[\gcal_{µ}^{\infty} \Phi_{F,f}(\Xi_{s}) \right]
= &\left.\frac{d}{du}\egras\left[\egras\left[
\Phi_{F,f}(\Xi_{t})|\Xi_{s}
\right]\right]\right|_{t = s} \\
= & \left.\frac{d}{du}\egras \left[\Phi_{F,f}(\Xi_{t})\right]\right|_{t = s}
\end{align*}
and we can conclude.
\end{proof}

\subsection{Properties of the densities of coupled $k$-parent SLFVs}

The goal of this section is to prove technical lemmas about the density of coupled $k$-parent SLFVs, which will be used in Section \ref{sec:3} in order to construct the $\infty$-parent SLFV. 

In all that follows, let $µ$ be a $\sigma$-finite measure on $(0,+\infty)$ satisfying Condition (\ref{eqn:condition_intensite}), and let $\Pi^{c}$ be a Poisson point process on $\mathbb{R} \times \rd \times (0,+\infty) \times U$ with intensity 
\begin{equation*}
dt \otimes dx \otimes µ(d\rcal) \otimes \tilde{u}\left(
d(p_{n})_{n \geq 1}
\right).
\end{equation*}
We recall that for all $\Xi = \sum_{i = 1}^{l} \delta_{x_{i}} \in \mcal_{p}(\rd)$, we denote the set of atoms of $\Xi$ by
\begin{equation*}
A(\Xi) := \{x_{i} : i \in \llbracket 1,l \rrbracket\}.
\end{equation*}

\begin{lem}\label{lem:appendixC_1}
For all $k \geq 2$, for all $0 \leq s \leq t$ and for all $x \in \rd$,
\begin{equation*}
A\left(
\Xi_{k,t}^{\Pi^{c},t,\delta_{x}}
\right) = \underset{x' \in A\left(
\Xi_{k,t-s}^{\Pi^{c},t,\delta_{x}}
\right)}{\bigcup}
A\left(
\Xi_{k,s}^{\Pi^{c},s,\delta_{x'}}
\right).
\end{equation*}
\end{lem}

\begin{proof}
Let $k \geq 2$, let $0 \leq s \leq t$ and let $x \in \rd$. Let $y \in A\left(\Xi_{k,t}^{\Pi^{c},t,\delta_{x}}\right)$. Then, we can construct a chain of reproduction events linking the point $x$ at time $t$ to the point $y$ at time $0$. We can split it into two chains :
\begin{itemize}
\item one linking the point $x$ at time $t$ to a point $y' \in \rd$ at time $s$,
\item one linking the point $y'$ at time $s$ to the point $y$ at time $0$.
\end{itemize}
Therefore, $y' \in A\left(\Xi_{k,t-s}^{\Pi^{c},t,\delta_{x}}\right)$ and $y \in A\left(\Xi_{k,s}^{\Pi^{c},s,\delta_{y'}}\right)$, which means that
\begin{equation*}
y \in \underset{x' \in A\left(
\Xi_{k,t-s}^{\Pi^{c},t,\delta_{x}}
\right)}{\bigcup}
A\left(
\Xi_{k,s}^{\Pi^{c},s,\delta_{x'}}
\right).
\end{equation*}

Conversely, let $y$ belonging to this set. It means that there exists $x' \in A\left(\Xi_{k,t-s}^{\Pi^{c},t,\delta_{x}}\right)$ such that $y \in A\left(\Xi_{k,s}^{\Pi^{c},s,\delta_{x'}}\right)$. Therefore, we can construct two chains of reproduction events, linking the point $x$ at time $t$ to the point $x'$ at time $s$, and the point $x'$ at time $s$ to the point $y$ at time $0$. Hence $y \in A\left(\Xi_{k,t}^{\Pi^{c},t,\delta_{x}}\right)$, and we can conclude.
\end{proof}

\begin{lem}\label{lem:appendixC_2}
For all $k \geq 2$, for all $0 \leq s \leq t$ and for all $x \in \rd$,
\begin{equation*}
\omega_{k,t}^{\Pi,\omega}(x) = 
\underset{x' \in A\left(
\Xi_{k,t-s}^{\Pi^{c},t,\delta_{x}}\right)}
{\prod} \omega_{k,s}^{\Pi^{c},\omega}(x').
\end{equation*}
\end{lem}

\begin{proof}
Let $k \geq 2$, let $0 \leq s \leq t$ and let $x \in \rd$. By definition,
\begin{equation}\label{eqn:appendixC_1}
\omega_{k,t}^{\Pi^{c},\omega}(x) = 
\underset{y \in A\left(\Xi_{k,t}^{\Pi^{c},t,\delta_{x}}
\right)}{\prod}\omega(y)
\end{equation}
and
\begin{equation}\label{eqn:appendixC_2}
\underset{x' \in A\left(
\Xi_{k,t-s}^{\Pi^{c},t,\delta_{x}}\right)}
{\prod} \omega_{k,s}^{\Pi^{c},\omega}(x')
= \underset{x' \in A\left(
\Xi_{k,t-s}^{\Pi^{c},t,\delta_{x}}\right)}
{\prod}
\underset{y \in A \left(
\Xi_{k,s}^{\Pi^{c},s,\delta_{x'}}\right)}
{\prod} \omega(y).
\end{equation}

Since by Lemma \ref{lem:appendixC_1}
\begin{equation*}
A\left(
\Xi_{k,t}^{\Pi^{c},t,\delta_{x}}
\right) = \underset{x' \in A\left(
\Xi_{k,t-s}^{\Pi^{c},t,\delta_{x}}
\right)}{\bigcup}
A\left(
\Xi_{k,s}^{\Pi^{c},s,\delta_{x'}}
\right),
\end{equation*}
the same terms appear in both products. However, some terms may appear more than once in Eq. (\ref{eqn:appendixC_2}), while they can appear only once in Eq. (\ref{eqn:appendixC_1}). But $\omega$ is $\{0,1\}$-valued, so for all $y \in \rd$ and $j \in \mathbb{N}^{*}$, $\omega^{j}(y) = \omega(y)$, and we can conclude.
\end{proof}

\begin{lem}\label{lem:appendixC_3}
For all $\tilde{k} \geq 2$, for all $0 \leq s \leq t$ and for all $x \in \rd$,
\begin{equation*}
\lim\limits_{k \to + \infty} 
\underset{x' \in A\left(\Xi_{\tilde{k},t-s}^{\Pi^{c},t,\delta_{x}}
\right)}{\prod} \omega_{k,s}^{\Pi^{c},\omega}(x')
\leq \underset{x' \in A\left(\Xi_{\tilde{k},t-s}^{\Pi^{c},t,\delta_{x}}
\right)}{\prod} \lim\limits_{k \to + \infty} \omega_{k,s}^{\Pi^{c},\omega}(x').
\end{equation*}
\end{lem}

\begin{proof}
Let $\tilde{k} \geq 2$, let $0 \leq s \leq t$ and let $x \in \rd$. Since both quantities are $\{0,1\}$-valued, we only need to show that if
\begin{align*}
\underset{x' \in A\left(\Xi_{\tilde{k},t-s}^{\Pi^{c},t,\delta_{x}}
\right)}{\prod} \lim\limits_{k \to + \infty} \omega_{k,s}^{\Pi^{c},\omega}(x') &= 0 \\
\intertext{then}
\lim\limits_{k \to + \infty} 
\underset{x' \in A\left(\Xi_{\tilde{k},t-s}^{\Pi^{c},t,\delta_{x}}
\right)}{\prod} \omega_{k,s}^{\Pi^{c},\omega}(x') &= 0.
\end{align*}

Assume that the first equality is true. Then, there exists $x' \in A\left(\Xi_{\tilde{k},t-s}^{\Pi^{c},t,\delta_{x}}\right)$ such that $\lim\limits_{k \to + \infty} \omega_{k,s}^{\Pi^{c},\omega}(x') = 0$. But since $\left(\omega_{k,s}^{\Pi^{c},\omega}(x')\right)_{k \geq 2}$ is decreasing and $\{0,1\}$-valued, there exists $k' \geq 2$ such that for all $k \geq k'$, $\omega_{k,s}^{\Pi^{c},\omega}(x') = 0$. Therefore, for all $k \geq k'$,
\begin{align*}
\underset{x' \in A\left(\Xi_{\tilde{k},t-s}^{\Pi^{c},t,\delta_{x}}
\right)}{\prod} \omega_{k,s}^{\Pi^{c},\omega}(x') &= 0, \\
\intertext{which means that}
\lim\limits_{k \to + \infty} 
\underset{x' \in A\left(\Xi_{\tilde{k},t-s}^{\Pi^{c},t,\delta_{x}}
\right)}{\prod} \omega_{k,s}^{\Pi^{c},\omega}(x') &= 0.
\end{align*}
\end{proof}

\begin{acknowledgements}
\quad The author would like to thank her thesis supervisor Amandine Véber for fruitful discussions and helpful comments.
The author is also grateful to the two anonymous reviewers for their valuable comments and suggestions, regarding in particular other possible convergence results of interest and the characterization of the $\infty$-parent SLFV as taking its values in the class of measurable subsets of $\mathbb{R}^{d}$.  
This work was partly supported by the chaire program "Modélisation Mathématique et Biodiversité" of Veolia Environnement-Ecole Polytechnique-Museum National d’Histoire Naturelle-Fondation X.
\end{acknowledgements}

\bibliographystyle{plain}
\bibliography{biblio_definition_infty_SLFV}

\end{document}